\documentclass[trsc]{informs3}

\renewcommand{\theARTICLETOP}{}
\renewcommand{\theRRHSecondLine}{}
\renewcommand{\theLRHSecondLine}{}

\OneAndAHalfSpacedXI %

\usepackage{natbib}
 \bibpunct[, ]{(}{)}{,}{a}{}{,}%
 \def\bibfont{\small}%

\TheoremsNumberedThrough     %

\EquationsNumberedThrough    %

\usepackage[utf8]{inputenc}

\usepackage{bm}
\usepackage{bbm}
\usepackage{bbold}
\usepackage{hyperref}
\usepackage{soul}
\usepackage{enumitem}

\usepackage{algorithm}
\usepackage{algpseudocode}

\usepackage{nicefrac}

\usepackage{cleveref}

\usepackage{array}
\usepackage{booktabs}
\usepackage{multirow}
\usepackage{makecell}

\usepackage{tikz}
\usepackage{tikz-network}
\usetikzlibrary{shapes, arrows, positioning, decorations.pathreplacing}
\usepackage{subcaption}

\newcommand{\personicon}{%
  \begin{tikzpicture}[x=1pt,y=1pt,baseline=-2pt]
    \draw[line width=1.2pt] (0,7) circle (4);
    \draw[line width=1.2pt]
      (0,2)
      .. controls (-5,-1) and (-5,-6) .. (-5,-8)
      .. controls (-2,-10) and (2,-10) .. (5,-8)
      .. controls (5,-6) and (5,-1) .. (0,2);
  \end{tikzpicture}%
}

\tikzset{
  person/.style={
    circle,
    draw=black,
    line width=1.1pt,
    minimum size=0.62cm,
    inner sep=0pt
  },
  highlighted/.style={
    person,
    fill=gray!25
  },
  link/.style={
    line width=1.1pt,
    ->
  },
  plainlink/.style={
    line width=1.1pt
  },
  dashedlink/.style={
    line width=1.1pt,
    dashed,
    dash pattern=on 4pt off 3pt,
    gray!65
  }
}
\tikzset{
  halfgrayperson/.style={
    person,
    path picture={
      \fill[gray!25]
        (path picture bounding box.south west) --
        (path picture bounding box.north west) --
        (path picture bounding box.north) --
        (path picture bounding box.south) -- cycle;
    }
  }
}

\usepackage{xcolor}

\allowdisplaybreaks

\begin{document}

\RUNAUTHOR{Shen et al.} %

\RUNTITLE{On-Trip Matching and Pricing}

\TITLE{On-Trip Matching and Pricing for Shared Rides}

\ARTICLEAUTHORS{%
\AUTHOR{Yifan Shen}
\AFF{Department of Industrial and Systems Engineering, University of Washington, Seattle, \EMAIL{syf1996@uw.edu}}
\AUTHOR{Junlin Chen}
\AFF{Department of Industrial Engineering and Operations Research, University of California, Berkeley, \EMAIL{junlin.chen@berkeley.edu}}
\AUTHOR{Julia Yan}
\AFF{Sauder School of Business, University of British Columbia, Vancouver, \EMAIL{julia.yan@sauder.ubc.ca}}
\AUTHOR{Chiwei Yan}
\AFF{Department of Industrial Engineering and Operations Research, University of California, Berkeley, \EMAIL{chiwei@berkeley.edu}} %
} %

\ABSTRACT{

Shared rides, which match compatible riders in one vehicle, have the potential to increase vehicle utilization and reduce congestion and emissions. However, this service has seen limited success outside dense urban areas due to insufficient matching opportunities. The dominant model in the literature, which we refer to as the \emph{pre-trip} model, matches riders before dispatch using non-bipartite matching, and dispatches unmatched riders solo. In practice, augmenting this pre-trip matching mechanism with an \emph{on-trip} matching phase, where riders initially dispatched solo can still match with later arrivals, can expand matching opportunities. Yet the addition of the on-trip phase complicates the dynamic matching problem, as riders' matching opportunities evolve as they travel: both the set of compatible riders and the value of matches can diminish as a rider progresses towards her destination. To address this challenge, we formulate a Markov decision process that tracks riders’ evolving states and weighs
immediate matches against future matching opportunities. Within this framework, we incorporate pricing decisions to maximize long-run average profit, where we develop a fluid relaxation for static pricing and use value function approximation for dynamic (state-dependent) pricing. Experiments using  Chicago ridesharing data show that pre-trip matching performs well with concentrated demand, whereas on-trip matching is particularly beneficial in the sparse outskirts where operating shared rides is especially challenging. Across both static and dynamic pricing settings, adding on-trip matching increases platform profit and match rates while lowering average quoted prices, showing the value of capturing matching opportunities along the route when demand is too dispersed for effective pre-trip matching.

}

\HISTORY{First version, 02/2026. Revised on 10/2026.}

\maketitle

\section{Introduction}
Shared rides, which match compatible riders together in one vehicle, have the potential to increase vehicle utilization and reduce congestion and emissions \citep{agatz2012optimization}. However, although platforms such as Uber and Lyft have offered shared rides since 2014, these products have long struggled with profitability and service quality \citep{uber_zombie}.

A central challenge is that the success of shared rides relies heavily on local demand density, which directly affects the availability of quality matches \citep{allon2022Pool}. An analysis of Toronto ridesharing data finds that downtown Toronto accounts for 66\% of total matches with a match rate of 62\%, whereas the match rate in the outlying areas drops to only 39\% \citep{young2020true}.
Platforms also acknowledge that ``matching with a co-rider is dependent on ... the number of ride requests ... available in a given area" \citep{uber_newshare_gb}, and limit their shared ride offerings accordingly. For example, UberX Share is only available in select major cities with high demand \citep{uber_intro_share}, and Lyft, after discontinuing shared rides entirely in 2023, pivoted in 2025 to offering shared rides only at major airports where demand is concentrated \citep{lyft_newsshare_revive}.
Consequently, the benefits of shared rides have not been realized outside of dense urban centers. Notably, lower-density, outlying areas also tend to coincide with transit deserts \citep{jiao2013transit,aman2020transit}.

\begin{figure}[htbp]
     \centering
     \begin{subfigure}[b]{0.4\textwidth}
        \begin{center}  \includegraphics[width=\textwidth]{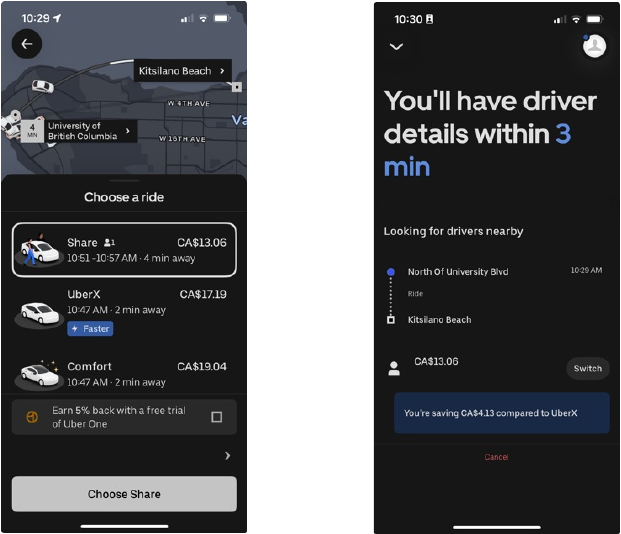}
          \caption{Pre-trip}
          \label{fig:ui-pretrip}
        \end{center}
     \end{subfigure}
     \hspace{2cm}
     \begin{subfigure}[b]{0.4\textwidth}
        \begin{center}  \includegraphics[width=0.45\textwidth]{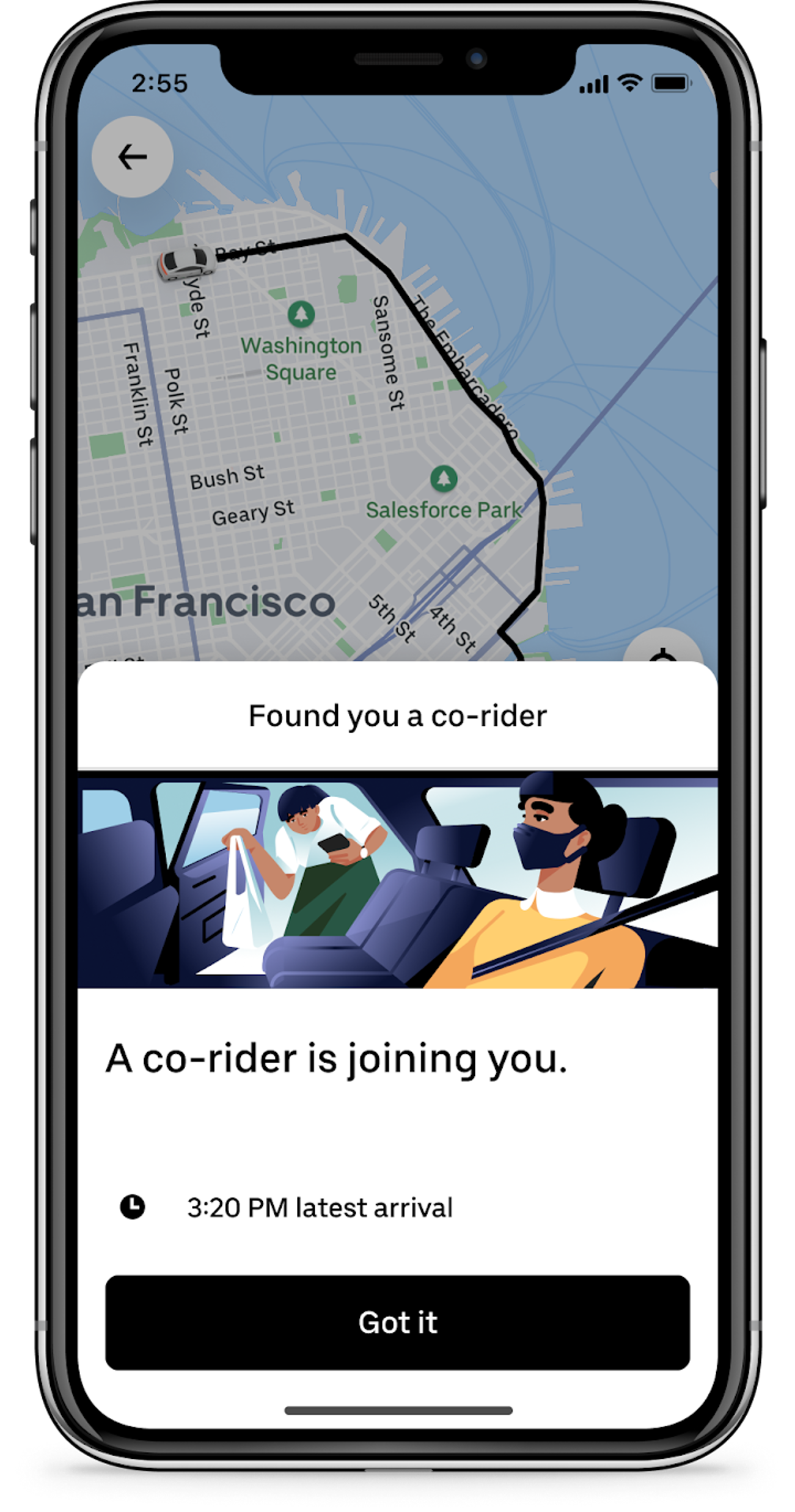}
          \caption{On-trip}
          \label{fig:ui-ontrip}
        \end{center}
     \end{subfigure}
    \caption{Screenshots of the Uber app with both pre-trip and on-trip matching.}
    \label{fig:user-interface}
\end{figure}

Although platforms can attempt to increase demand density by asking riders to wait longer for a ride, the resulting inconvenience can deter riders who expect high-quality, on-demand service. Hence, a goal for platforms is to %
increase matches without excessively inconveniencing riders, especially in low-demand areas. To this end, the platform must capitalize on all matching opportunities, which, in practice, can arise not just before a request is dispatched on a trip (\mbox{\Cref{fig:ui-pretrip}}), but during the trip as well (\mbox{\Cref{fig:ui-ontrip}}).
However, the dominant matching model in the literature, which is based on non-bipartite matching \citep{santi2014quantifying}, neglects the latter. In particular, an underlying assumption in this model is that matching is done \emph{pre-trip}, before riders are dispatched; then, unmatched riders are dispatched in solo rides.
However, matching opportunity can also arise on-trip, when a rider is already en route from origin to destination.

\begin{figure}[htbp]
\begin{center}
\begin{tikzpicture}[scale=0.82, transform shape]
\def\xoi{0}
\def\xdi{10}
\def\xoj{1}
\def\xdj{4}
\def\xok{2}
\def\xdk{8}

 \node at (-1.5, 0) {$t=0$};
 \node[person, fill=gray!25] (Oi) at (\xoi,0) {\personicon};
 \node[person, fill=gray!25] (Di) at (\xdi,0) {\personicon};
 \draw[->, >=stealth, line width=0.8pt]
    (Oi.east) -- (Di.west);
    \node[below=12pt] at (Oi) {$O_i$};
    \node[below=12pt] at (Di) {$D_i$};

  \node at (-1.5, -2) {$t=1$};
  \node[person, fill=gray!25] (Oi1) at (\xoi,-2) {\personicon};
  \node[person, fill=gray!25] (Di1) at (\xdi,-2) {\personicon};
  \node[person] (Oj1) at (\xoj,-2) {\personicon};
  \node[person] (Dj1) at (\xdj,-2) {\personicon};
   \node[below=12pt] at (Oj1) {$O_j$};
    \node[below=12pt] at (Dj1) {$D_j$};
  \draw[-, >=stealth, line width=0.8pt]
    (Oi1.east) -- (Oj1.west);
    \draw[-, >=stealth, line width=0.8pt]
    (Oj1.east) -- (Dj1.west);
  \draw[->, >=stealth, line width=0.8pt]
    (Dj1.east) -- (Di1.west);

  \node at (-1.5, -4) {$t=2$};
  \node[person, fill=gray!25] (Di2) at (\xdi,-4) {\personicon};
  \node[halfgrayperson] (Oj2) at (\xoj,-4) {\personicon};
  \node[person] (Dj2) at (\xdj,-4) {\personicon};
  \node[person, fill=gray!75] (Ok2) at (\xok,-4) {\personicon};
  \node[person, fill=gray!75] (Dk2) at (\xdk,-4) {\personicon};
  \node[below=12pt] at (Ok2) {$O_k$};
    \node[below=12pt] at (Dk2) {$D_k$};
  \draw[dashedlink]
    (\xoi,-4) -- (Oj2.west);
    \draw[-, >=stealth, line width=0.8pt]
    (Oj2.east) -- (Ok2.west);
     \draw[-, >=stealth, line width=0.8pt]
    (Ok2.east) -- (Dj2.west);
    \draw[-, >=stealth, line width=0.8pt]
    (Dj2.east) -- (Dk2.west);
  \draw[->, >=stealth, line width=0.8pt]
    (Dk2.east) -- (Di2.west);

\end{tikzpicture}
  \caption{Three riders arriving on a line over time.}
  \label{fig:intro-example}
\end{center}
\end{figure}

To illustrate how including an on-trip matching phase complicates matching decisions, consider the line-city example in \Cref{fig:intro-example}. Suppose a ridesharing platform has available vehicles of capacity two, that riders are willing to wait one period before being dispatched, and the platform seeks to dispatch all riders at minimum cost. At time $t = 0$, rider $i$ (light gray) arrives, seeking travel from origin $O_i$ to destination $D_i$, and begins to wait. At time $t=1$, rider $j$ (white) arrives, seeking travel from origin $O_j$ to destination $D_j$.
At the same time, rider $i$ must now be dispatched, either solo or in a shared ride with rider $j$.\footnote{For simplicity, we ignore pickup time in this example.} At time $t=2$, rider $k$ (dark gray) arrives, seeking travel from origin $O_k$ to destination $D_k$.

In hindsight, knowing that rider $k$ arrives at $t=2$, it is optimal to first dispatch riders $i$ and $j$ solo and later match riders $i$ and $k$, because riders $i$ and $k$ are a higher-value match (dispatching them together would save more cost). However, designing a matching policy without knowing which riders will arrive in the future is challenging. This challenge is related to the main technical difficulty in modeling on-trip matching: a rider's matching opportunities can deteriorate as she travels. As shown in \Cref{fig:intro-example}, after rider $i$ is dispatched solo, the feasibility and value of matching her with rider $k$ depend on her position. After rider $i$ has passed rider $k$'s origin, the value of matching them deteriorates until they can no longer be matched. This differs from standard dynamic stochastic matching models, which assume that pairwise compatibilities and matching values remain fixed while agents are in the system \citep{collina2020dynamic, aouad2022dynamic, eom2024batching}. Evaluating the value of leaving a rider unmatched therefore requires tracking her not only while she waits pre-trip, but also as she progresses through solo travel, resulting in additional modeling complexity.

Because of this technical challenge, myopic decisions are usually suboptimal under on-trip matching. For example, under the greedy policy, the platform matches riders $i$ and $j$ immediately at $t=1$, resulting in a higher cost. One alternative approach is for the platform to accumulate riders in \emph{batches} over fixed windows of time, and then decide matches over the riders in each batch \citep{yan2020dynamic}. However, in this example, rider $k$ arrives after rider $i$ must be dispatched, so they cannot be part of the same batch. Then, a common remedy is for the platform to \emph{reoptimize} \citep{alonso2017demand}: namely, if there remain unmatched riders at the end of a batch window, they can be added to the next batch, possibly with an updated current location if they have been dispatched. However, because the vehicle has already picked up both riders $i$ and $j$ and reached capacity, rider $i$ is no longer available for reoptimization until the vehicle drops off rider $j$ at $D_j$, so rider $k$ must be dispatched solo. This issue arises because when the platform decides to match riders $i$ and $j$, it has not considered the possibility of matching rider $i$ and rider $k$, as rider $k$ has not yet arrived. More broadly, when making pre-trip matching decisions, the platform needs to account for the future on-trip matching value of a rider that is initially dispatched solo.

Based on the discussion above, we aim to develop a non-myopic matching model that accounts for changing matching opportunities in the on-trip phase.  We also investigate the interaction of matching with pricing. Matching and pricing decisions are intertwined: low prices can increase demand and encourage matching, and high-quality matching can reduce dispatch costs and allow the platform to lower prices. Hence, we study the problem of joint matching and pricing, and aim to develop computationally efficient solution methods under both static and dynamic pricing settings.%

To study these interactions, we consider a setting in which riders dynamically arrive, request rides in response to prices quoted by the platform, and wait a limited time for a match to arrive (pre-trip phase). If a rider is dispatched solo, then the platform can continue to evaluate potential matches until she reaches her destination (on-trip phase). The platform seeks to maximize long-run average profit. In this setting, we make the following contributions.
\begin{enumerate}
    \item \textbf{A joint pre-trip and on-trip matching model.} %
    We formulate the joint matching and pricing problem as an average-reward Markov decision process (MDP). The state records riders' progress through pre-trip waiting and solo travel, capturing how their compatibility and matching values evolve after dispatch. Under the MDP framework, the platform weighs immediate matching against leaving riders unmatched for future arrivals, accounting for both current operational costs and future system value.
    \item  \textbf{Solution methods for joint matching and pricing.} We develop computationally tractable matching and pricing policies for the MDP. For static pricing, we formulate a fluid relaxation in terms of steady-state matching and solo flows. We embed this relaxation in a joint matching and pricing optimization problem, and use the dual variables to guide a non-myopic matching policy. For dynamic and state-dependent pricing, we use the fluid dual variables to construct an additive approximation of the relative value function, which yields a heuristic that jointly adapts matching decisions and prices to the current system state.

    \item \textbf{Evaluation on real-world data.} We evaluate our matching and pricing policies with large-scale simulations on Chicago ridesharing data. We show that our proposed policy yields higher profit than myopic and non-myopic batching benchmarks, and also show that adding on-trip matching generally improves profits and efficiency for the platform, and lowers prices for riders. The gains are especially pronounced in the outskirts of the city, where demand is less dense and more spatially dispersed.%
\end{enumerate}

The remainder of the paper is outlined as follows. \Cref{sec:literature} reviews the related literature. \Cref{sec:model} presents the models, \Cref{sec:pricing_matching_policies} develops matching and pricing policies under static and dynamic pricing, and \Cref{sec:computational} reports computational results
on Chicago ridesharing data. \Cref{sec:conclusion} offers concluding remarks. The appendices contain analytical results of a stylized example, proofs, synthetic experiments, and additional computational results.

\section{Literature Review} \label{sec:literature}
In this section, we review two streams of literature closely related to our work. The first stream focuses on dynamic stochastic matching theory. The second stream examines shared ride services, with an emphasis on approaches to expanding matching opportunities.

\paragraph{Dynamic Stochastic Matching.} Our study is related to dynamic stochastic matching theory, which studies how to match compatible agents who arrive and depart stochastically over time.\footnote{Departure is also sometimes called \emph{abandonment} in the literature. The time between arrival and departure is called the \emph{sojourn time}, and riders' limited sojourn time is said to reflect their \emph{impatience}.}
The simplest approach to solving such problems is \emph{greedily} matching riders upon arrival.
In a homogeneous market, \cite{akbarpour2020thickness} show that greedy matching is close to optimal in the large-market limit when imminent departures are non-observable.
In a heterogeneous setting where unmatched riders may depart with some probability after each period, \cite{eom2024batching} find that a randomized greedy algorithm is asymptotically optimal for reward maximization, if riders are not too impatient.
Another myopic approach involves \emph{batching} rider requests over short windows of time (e.g., 3 minutes in \Cref{fig:ui-pretrip}), and then solving the non-bipartite matching problem on the batch \citep{santi2014quantifying}.
\cite{eom2024batching} find that batching is also asymptotically optimal, but is more vulnerable to impatience.
In a setting where sojourn times are uniform and constant, \cite{pavone2022online} and \cite{ashlagi2023edge} find that batching (with randomization to prevent the batch window from splitting up compatible matches) attains a constant-factor guarantee.

Non-myopic matching policies can be particularly useful when agents have relatively short sojourn times. In such cases, fluid relaxations and value function approximations are often used to account for future matching opportunities.
For example, \cite{collina2020dynamic} propose an online matching algorithm based on a fluid relaxation, and provide a competitive-ratio guarantee relative to an optimal-in-hindsight policy.
\cite{aouad2022dynamic} develop a matching policy that combines a fluid relaxation with value function approximations and demonstrate that batching can perform arbitrarily poorly in cost minimization. Both studies emphasize the importance of tailoring the timing of matching decisions.

Similar to \cite{collina2020dynamic} and \cite{aouad2022dynamic}, our study also focuses on developing non-myopic matching policies and uses fluid relaxations and value function approximations. However, the models discussed above assume that pairwise matching compatibilities and values remain fixed while agents are in the system.
A unique challenge associated with modeling on-trip matching is that dispatched requests and their associated matching compatibilities and values evolve over the course of their trips.
Our model therefore needs to track riders' progress so that the value of remaining unmatched reflects both future arrivals and the changes in feasible matches along the route that the existing models do not capture.

\paragraph{Shared Ride Services.} Shared rides, which pool riders with compatible trips in one vehicle, have been studied extensively in the literature. Market thickness is crucial for this application to succeed; a central research question is therefore how to expand the market so that more riders can share vehicle trips.
To this end, pricing can encourage usage of shared rides over solo rides \citep{hu2020share,jacob2021ride,liu2021framework,fielbaum2022split,wang2024Optimizingc,taylor2023shared} or convert riders who can make high-quality matches \citep{yan2023pricing}.
We therefore consider matching and pricing jointly in our study.
 Alternatively, platforms can also expand matching opportunities through service design.
For example, riders can be asked to wait longer \citep{yan2020dynamic}, walk to designated meeting points \citep{stiglic2015benefits,fielbaum2021demand,zhang2023routing}, or accept longer detours \citep{stiglic2015benefits}.
However, such requirements can also deter riders from taking shared rides because of the associated inconvenience \citep{daganzo2020analysis,lobel2025detours}, and are not the focus of this work.

In this study, we investigate on-trip matching, which increases matching opportunities by extending matching beyond dispatch.
For long-term planning, \cite{samaranayake2017Ridepooling} incorporate on-trip matching using a deterministic steady-state routing model, and determine the minimum fleet size.
For real-time operations, several studies implement on-trip matching using reoptimization and insertion procedures to update the requests served by each vehicle. For example, \cite{hosni2014shared} use an incremental cost heuristic to assign new requests by comparing additional route costs across feasible taxis. \cite{alonso2017demand} repeatedly solve integer programs to assign feasible groups of requests to vehicles.
\cite{simonetto2019Realtime} simplify the optimization process by solving linear assignment problems with route costs computed through insertion procedures.
\cite{hua2022optimality} develop a branch-and-price algorithm that allows scheduled requests to be reassigned before pickup.
These methods are myopic in that they update routes and assignments as
requests are observed, without valuing matching opportunities
with unobserved future arrivals.

On the other hand, some studies predict future matching opportunities to guide pricing or routing, but they retain prescribed myopic matching rules.
\cite{wang2021Predicting} track unmatched riders along their routes to predict matching outcomes, and \cite{wang2024Optimizingc} use this framework to jointly optimize up-front discounts across origin-destination pairs.
Their prescribed matching rule matches an arriving request whenever a feasible partner is available, without evaluating whether to leave the rider unmatched for a potentially better future match.
Similarly, \cite{zhu2025Route} optimize the routes of partially occupied vehicles to increase the probability of finding another rider, while retaining a prescribed first-come-first-served assignment
rule that prioritizes empty vehicles.

Only a few studies incorporate future costs or matching opportunities directly into matching decisions. \cite{sayarshad2015nonmyopic} develop non-myopic vehicle assignment and pricing policies using aggregate workload and delay measures to approximate future costs. \cite{yang2024PredictionBased} incorporate predicted distance savings into a bipartite assignment model, allowing the platform to forgo an available match. Instead of relying on prediction, we develop an MDP with a more detailed representation of future matching opportunities, explicitly accounting for how riders' compatibilities and matching savings evolve along their routes. Another closely related study is \cite{yan2023pricing}, which also allows the platform to forgo an available match in anticipation of better matches with future arrivals and analyzes a pricing policy that offers a discount ex-post to encourage high-quality matches.
We similarly consider the interaction of matching and pricing policies.
However, \cite{yan2023pricing} focus on pre-trip matching only, whereas we explicitly model the on-trip matching phase.

\section{Modeling}
\label{sec:model}
\label{sec:model_combined}
In this section, we first introduce the problem setting in \Cref{ss:model_setting}, and discuss both pre-trip and on-trip matching. We then formulate the platform's matching and pricing problem as a Markov decision process in \Cref{ss:model_mdp}.

\subsection{Problem Setting}
\label{ss:model_setting}

We consider a discrete-time, infinite-horizon setting with time indexed by $t \in \mathbbm{Z}_{\geq 0}:=\{0,1,2,\cdots\}$ and $N$ distinct rider types, indexed by $i \in [N] := \{1, \dots, N\}$, where each type corresponds to a unique origin-destination pair $(O_i, D_i)$. We assume that the time discretization is sufficiently fine and that at most one potential rider arrives (i.e., enters the platform and seeks a price quote for the shared ride service) in each time period.
In each period, a
type-$i$ rider arrives with probability $\Lambda_i$ for each $i\in[N]$,
and no rider arrives with probability
$1-\sum_{i\in[N]}\Lambda_i>0$.\footnote{We can further assume that the time discretization is sufficiently fine so that $\sum_{i \in [N]} \Lambda_i \ll 1$. Although this condition is not strictly required for the modeling formulation, it enables a convenient approximation of the rider arrival process by a Poisson process when calibrating arrivals in the computational experiments (\Cref{sec:computational}).
}

The arriving type-$i$ rider converts (i.e., accepts the quoted price and submits the trip request) if her willingness to pay (WTP) is at least as high as the quoted price $p_i$. In this study, we assume that the WTP values of type-$i$ riders are independent and identically distributed with a cumulative distribution function $F_i(\cdot)$.\footnote{WTP can incorporate riders' exogenous disutility from sharing, due to factors such as discomfort or loss of privacy.} Define $\lambda_i:=1 - F_i(p_i) \in [0,1]$ as the \textit{conversion probability}, that is, the probability that an arriving type-$i$ rider accepts the quoted price and requests a ride. We further assume that $F_i(\cdot)$ is continuous and strictly increasing on its support interval, such that there is a unique price $p_i(\lambda_i)=F_i^{-1}(1-\lambda_i)$ for each $\lambda_i\in[0,1]$. Consequently, optimizing prices is equivalent to optimizing $\lambda_i$ for all $i \in [N]$. We denote by $p_i(\lambda_i)$ the price that induces conversion probability $\lambda_i$ for all subsequent analysis.

If a newly converted rider is not immediately matched with another rider, we refer to her as a \emph{solo rider}. Until this solo rider is successfully matched or eventually reaches her destination, the platform continues to seek a match with other converted riders. Specifically, the platform first initiates a waiting window of $T \in \mathbbm{Z}_{\geq 0}$ time periods, during which it attempts to perform \emph{pre-trip matching}. If $T$ time periods have elapsed and the rider is still solo, she departs in an available vehicle and begins traveling toward her destination along a fixed shortest path for her rider type.\footnote{In this study, we focus on rider-side matching and pricing and abstract from detailed driver-side operations, such as driver pricing, relocation, and assignment. One way to interpret this simplification is to assume a sufficiently large fleet that can serve every converted request with a sufficiently short pickup time, consistent with \cite{fielbaum2022split} and \cite{yan2023pricing}. Another way to interpret this assumption is to treat our rider-side decisions as the first stage of a two-stage problem, with driver operations determined in the second stage.
Similarly, \cite{wang2024Optimizingc} treat driver payments as exogenous when optimizing up-front discounts across origin-destination pairs, and
\cite{santi2014quantifying} identify shareable trip groups before taxi dispatch. These simplifications allow us to develop computationally efficient matching and pricing algorithms.} Any successful matching that occurs after a rider has departed her origin is referred to as \emph{on-trip matching}.

In the main model, we assume that each rider may be matched with at most one other rider throughout her trip. That is, we do not consider the possibility of \emph{trip-chaining} \citep{samaranayake2017Ridepooling}, where the rider is matched with another rider after her previous co-rider reaches the destination. In Appendix \ref{ss:stylized_static}, we assess this simplification in a stylized model, in which we compare matching with and without trip-chaining and show that trip-chaining adds little profit in that setting.

For a type-$i$ rider, denote the total solo travel time from $O_i$ to $D_i$ by $\ell_i\in\mathbbm{Z}_{\geq1}$ periods.\footnote{We hold link travel times fixed over time and independent of platform demand, as in \cite{wang2024Optimizingc}. Several empirical studies also suggest that changes in ride-hailing demand can have limited effects on congestion. In downtown Chicago, \cite{liang2023congestion} find that when ridesourcing demand decreased by 10.6\%, the estimated speed increase was only 0.8\% and statistically insignificant. In New York City, \cite{li2024congestionfees} find that an approximately 11\% reduction in ride-hailing and taxi trip volume was not accompanied by a statistically significant congestion reduction.} With a slight abuse of terminology, we assume one unit of trip length represents one period of vehicle travel and refer to travel times equivalently as trip lengths hereafter. We can therefore represent the current state of a solo rider using the tuple $(i, u)$, where $i \in [N]$ denotes the rider type and the integer $u$ tracks the temporal state. The range of $u$ is $-T+1 \leq u \leq \ell_i-1$. When $u \le 0$, the rider is still at the origin with $|u|$ pre-trip waiting periods remaining. When $u > 0$, the rider has already departed the origin, and $u$ denotes the number of time periods elapsed since departure. For example, $u=-T+1$ indicates that the rider converted in the previous period. A rider with $u=0$ departs during the current period if she remains unmatched. A rider with $u=\ell_i-1$ reaches the destination in the next period and exits the system if she remains unmatched. For notational convenience, define $\mathcal{C}:=\{(i,u):i\in[N],\ u\in\mathbbm{Z},\ -T+1\leq u\leq\ell_i-1\}$ to be the set of all states of solo riders.

For a match between a newly converted rider of type $i \in [N]$ and an existing solo rider $(j,u) \in \mathcal{C}$, denote the shared trip length by $\ell_{i,j}^u$. When $u\leq0$, both riders are at their origins. The vehicle picks up the two riders at $O_i$ and $O_j$ and then drops them off at $D_i$ and $D_j$, choosing the pickup and drop-off orders that minimize the shared trip length. Since both riders remain at their origins before pickup, $\ell_{i,j}^u$ is constant throughout the pre-trip phase. When $u>0$, the shared trip starts from the current location of the solo rider $(j,u)$, visits $O_i$ to pick up the new rider, and then drops off the two riders at $D_i$ and $D_j$ in the order that minimizes the shared trip length.

\smallskip

\paragraph{Riders' Matching Compatibility.} In practice, not all potential matches are reasonable. Therefore, we allow the platform to impose compatibility conditions on matches. Let $\mathcal{N}_{j,u}^+$ be the set of rider types compatible with a solo rider $(j,u) \in \mathcal{C}$. Below, we introduce two basic conditions for a solo rider $(j,u)$ to be compatible with a new rider of type $i$ and use them to define $\mathcal{N}_{j,u}^+$ in the remainder of the paper:
\begin{itemize}
\item \emph{Trip Length Condition}: \begin{align*}
     \ell_{i,j}^u < \ell_i+\ell_j-\max\{u,0\},
\end{align*}
that is, the shared trip length must be shorter than the sum of the two riders' remaining solo trip lengths.
\item \emph{Backtracking Condition}:
$$\text{When } u>0,\ (j,u) \text{ should not be on any shortest path from } O_i \text{ (exclusive) to } D_j.$$
Otherwise, picking up the new type-$i$ rider could later involve \emph{backtracking}, i.e., traversing nodes that the on-trip solo rider $(j,u)$ has already visited, creating a bad experience for the solo rider $(j,u)$.
\end{itemize}

\begin{figure}[htbp]
     \centering
     \begin{subfigure}[b]{0.48\textwidth}
        \centering
        \begin{center}
        \begin{tikzpicture}[scale=0.82, transform shape]

  \node[person] (Oj) at (0,0) {\personicon};
  \node[person] (ju) at (2.2,0) {\personicon};
  \node[person] (Dj) at (4.2,0) {\personicon};
  \node[highlighted] (Oi) at (5.8,0) {\personicon};
  \node[highlighted] (Di) at (7.4,0) {\personicon};

  \draw[link]       (Oj) -- (ju);
  \draw[dashedlink] (ju) -- (Dj);
  \draw[dashedlink] (Dj) -- (Oi);
  \draw[dashedlink] (Oi) -- (Di);

  \node[below=12pt] at (Oj) {$O_j$};
  \node[below=12pt] at (ju) {$(j,u)$};
  \node[below=12pt] at (Dj) {$D_j$};
  \node[below=12pt] at (Oi) {$O_i$};
  \node[below=12pt] at (Di) {$D_i$};

\end{tikzpicture}
          \caption{Violation of trip length condition}
          \label{fig:compatible_1}
        \end{center}
     \end{subfigure}
     \hfill
     \begin{subfigure}[b]{0.48\textwidth}
        \begin{center}
        \begin{tikzpicture}[scale=0.82, transform shape]

  \node[person] (Oj) at (0,0) {\personicon};
  \node[highlighted] (Oi) at (1.8,0) {\personicon};
  \node[person] (ju) at (3.6,0) {\personicon};
  \node[person] (Dj) at (5.6,0) {\personicon};
  \node[highlighted] (Di) at (7.4,0) {\personicon};

  \draw[plainlink]  (Oj) -- (Oi);
  \draw[link]       (Oi) -- (ju);
  \draw[dashedlink] (ju) -- (Dj);
  \draw[dashedlink] (Dj) -- (Di);

  \node[below=12pt] at (Oj) {$O_j$};
  \node[below=12pt] at (Oi) {$O_i$};
  \node[below=12pt] at (ju) {$(j,u)$};
  \node[below=12pt] at (Dj) {$D_j$};
  \node[below=12pt] at (Di) {$D_i$};

\end{tikzpicture}
          \caption{Violation of backtracking condition}
          \label{fig:compatible_2}
        \end{center}
     \end{subfigure}
    \caption{Illustrations of incompatibility between a solo rider $(j,u)$ and a new rider of type $i$.}
    \label{fig:compatible}
\end{figure}

Figure \ref{fig:compatible} illustrates two scenarios where a solo rider $(j,u)$ is incompatible with a new rider of type $i$. In \Cref{fig:compatible_1}, serving the riders together requires
additional travel between $D_j$ and $O_i$, so the shared trip length
exceeds the sum of the two riders' remaining solo-trip lengths. In {\Cref{fig:compatible_2}}, although the trip length condition is satisfied, $(j,u)$ is now between $O_i$ and $D_j$, meaning that matching $(j,u)$ with the new type-$i$ rider requires traveling back toward $O_i$ first, {violating} the backtracking condition.

Under the backtracking condition, when $u > 0$, matching $(j,u)$ with a new rider of the same type $j$ is infeasible because the existing rider has already traveled along a shortest path from $O_j$ to $D_j$. Thus, the model excludes same-type matches after departure. This behavior contrasts with pre-trip matching, where riders of the same type can share their entire trip. More generally, both the compatibility set $\mathcal N_{j,u}^+$ and the shared trip length $\ell_{i,j}^u$ can change as a solo rider travels. The platform must therefore account for how waiting to match changes the compatible rider types and the cost savings, which is not captured in standard dynamic stochastic matching models \citep{collina2020dynamic, aouad2022dynamic, eom2024batching}. %

Note that, in addition to the two conditions defined above, the platform can use other criteria to determine which matches are allowed. For example, it can provide service-level guarantees by imposing limits on each rider's additional in-vehicle travel time or detour. These additional service-quality constraints only modify the compatibility set $\mathcal{N}_{j,u}^+$ and leave the subsequent model formulation and technical results unchanged.

\smallskip

\begin{figure}[!htbp]
\centering
\begin{tikzpicture}[
    font=\fontsize{10}{12}\selectfont,
    >=stealth,
    modelbox/.style={
        draw=black!80,
        rounded corners=5pt,
        line width=0.8pt,
        minimum height=10mm,
        text width=34mm,
        align=center,
        inner sep=4pt
    },
    modeldecision/.style={
        diamond,
        draw=black!80,
        line width=0.8pt,
        aspect=2.3,
        text width=34mm,
        align=center,
        inner sep=2pt
    },
    modelend/.style={modelbox,fill=black!5,text width=30mm},
    modelflow/.style={->,draw=black!80,line width=0.8pt},
    modellabel/.style={font=\fontsize{9}{11}\selectfont,align=center,inner sep=3pt}
]
\node[align=center] (arrival) at (0,0.2)
    {A type-$i$ rider arrives with probability $\Lambda_i$\\and receives a price quote $p_i$};
\node[modeldecision] (convert) at (0,-1.95)
    {Rider converts?};
\node[modelend] (decline) at (-6.0,-1.95)
    {No new request};
\node[modeldecision] (immediate) at (0,-5.1)
    {Matched with an\\existing solo rider?};
\node[modelend] (sharednow) at (6.1,-5.1)
    {Shared trip\\with an existing\\solo rider};
\node[modelbox] (wait) at (0,-7.7)
    {Waits at the origin\\for up to $T$ periods};
\node[modelend] (sharedpre) at (6.1,-7.7)
    {Shared trip\\with a newly\\converted rider};
\node[modelbox] (travel) at (0,-10.1)
    {Begins traveling\\toward $D_i$};
\node[modelend] (sharedon) at (6.1,-10.1)
    {Shared trip\\with a newly\\converted rider};
\node[modelend] (solo) at (0,-12.2)
    {Completes the\\solo trip};

\draw[modelflow] (arrival.south) -- (convert.north);
\draw[modelflow] (convert.west) --
    node[modellabel,above] {No} (decline.east);
\draw[modelflow] (convert.south) --
    node[modellabel,right] {Yes (with probability $\lambda_i$)} (immediate.north);
\draw[modelflow] (immediate.east) --
    node[modellabel,above] {Yes} (sharednow.west);
\draw[modelflow] (immediate.south) --
    node[modellabel,right] {No} (wait.north);
\draw[modelflow] (wait.east) --
    node[modellabel,above] {Pre-trip\\matching} (sharedpre.west);
\draw[modelflow] (wait.south) --
    node[modellabel,right] {After $T$ periods} (travel.north);
\draw[modelflow] (travel.east) --
    node[modellabel,above] {On-trip\\matching} (sharedon.west);
\draw[modelflow] (travel.south) --
    node[modellabel,right] {After $\ell_i$ periods} (solo.north);
\end{tikzpicture}%
\caption{An illustration of the problem setting from the perspective of a single rider.}
\label{fig:model_process}
\end{figure}

Thus far, we have specified rider arrivals and conversion, defined the pre-trip and on-trip phases, and discussed possible matching outcomes. Figure \ref{fig:model_process} summarizes these main settings from the perspective of a single rider. As special cases, we can obtain the pure on-trip matching model by setting $T=0$, and obtain the pure pre-trip matching model by disabling matches after riders depart their origins.

\subsection{Markov Decision Process}
\label{ss:model_mdp}

The optimal matching and pricing problem is formulated as an average-reward, discrete-time
Markov decision process (MDP). Let the state of the system at time $t\in\mathbbm{Z}_{\geq 0}$ be $\boldsymbol{\mathrm{s}}_t=(\mathrm{s}_{t,i,u})_{(i,u)\in\mathcal{C}} \in \mathcal{S}$, where $\mathrm{s}_{t,i,u}\in\mathbbm{Z}_{\geq 0}$ is the number of solo riders in state $(i,u)\in\mathcal{C}$ at time $t$ and $\mathcal{S}$ is the state space. We observe this state at the beginning of the period, before arrivals, conversions, and matching decisions.
The platform makes decisions on the matching and pricing policies:
\begin{itemize}
    \item \textit{Matching Policy.} The platform implements a state-dependent matching policy $\phi_i: \mathcal{S} \to \mathcal{C}\cup\{0\}$. When a new rider of type $i\in[N]$ converts at time $t$, given the current system state $\boldsymbol{\mathrm{s}}_t$, the platform either (i) immediately matches the new rider with an existing solo rider $(j,u)\in\mathcal{C}$ such that $i\in\mathcal{N}_{j,u}^+$ and $\mathrm{s}_{t,j,u}>0$, or (ii) leaves the new rider unmatched, initiating the waiting window of length $T$ (or directly dispatching a vehicle for a solo trip if $T=0$). We denote the latter decision by action $0$. Let $\boldsymbol{\phi}:=\phi_{[N]}$ denote the matching policy for all rider types.
    \item \textit{Pricing Policy.} We consider both static and dynamic (state-dependent) pricing policies. As discussed in \Cref{ss:model_setting}, optimizing prices is equivalent to optimizing the conversion probabilities for all rider types. A static pricing policy is therefore a vector $\boldsymbol{\lambda}=(\lambda_i)_{i\in[N]}\in[0,1]^N$ that specifies fixed conversion probabilities, independent of the system state. A dynamic pricing policy is a mapping $\boldsymbol{\lambda}(\cdot):\mathcal S\to[0,1]^N$, where $\lambda_i(\boldsymbol{\mathrm{s}}_t)$ is the conversion probability for type $i$ in state $\boldsymbol{\mathrm{s}}_t$.
\end{itemize}

Starting from an empty system, since each solo rider who remains unmatched advances by one temporal state per period and at most one new rider makes a request in each period, we have $\mathrm{s}_{t,i,u}\in\{0,1\}$. Thus, there is at most one solo rider in each state $(i,u)\in\mathcal C$ at any time. Therefore, the state space $\mathcal{S}$ is finite, with $|\mathcal{S}| \leq 2^{|\mathcal{C}|}$.

The state transition from $\boldsymbol{\mathrm{s}}_t$ to $\boldsymbol{\mathrm{s}}_{t+1}$ depends on the rider arrival and conversion outcomes and the platform's matching decision at time $t$. For ease of exposition, let $f(\boldsymbol{\mathrm{s}}_t)$ denote the state obtained by advancing all existing solo riders by one period:
\begin{align}
 [f(\boldsymbol{\mathrm{s}}_t)]_{i,u}
 =\begin{cases}
 \mathrm{s}_{t,i,u-1},&u>-T+1,\\
 0,&u=-T+1,
 \end{cases}
 \qquad (i,u)\in\mathcal C,
 \nonumber
\end{align}
and let $\boldsymbol e_{(i,u)}$ be the unit vector corresponding to $(i,u)$, with $\boldsymbol e_{(i,\ell_i)}=\boldsymbol0$. The three possible state transitions are therefore
\begin{align}
 \boldsymbol{\mathrm{s}}_{t+1}
 =\begin{cases}
 f(\boldsymbol{\mathrm{s}}_t),&\text{if no rider converts},\\
 f(\boldsymbol{\mathrm{s}}_t)+\boldsymbol e_{(i,-T+1)},&\text{if type $i$ converts and }\phi_i(\boldsymbol{\mathrm{s}}_t)=0,\\
 f(\boldsymbol{\mathrm{s}}_t-\boldsymbol e_{(j,u)}),&\text{if type $i$ converts and }\phi_i(\boldsymbol{\mathrm{s}}_t)=(j,u).
 \end{cases}
 \label{eq:state_transition}
\end{align}

Under either static or dynamic pricing, the platform collects each rider's payment upon conversion. Operational costs are proportional to trip length, at $c$ per unit of trip length. The platform incurs a cost of $c$ per period for each solo rider who travels during that period. When a new type-$i$ rider is matched with a solo rider $(j,u)$, we assume that the full cost $c\ell_{i,j}^u$ of transporting both riders to their destinations is incurred immediately. Denote the number of existing solo riders that would travel during the current period if left unmatched by
$
m(\boldsymbol{\mathrm{s}}_t):=\sum_{(j,u)\in\mathcal C}\mathrm{s}_{t,j,u}\mathbbm{1}\{u\geq0\}
$.
The profit generated in time period $t$ can then be expressed as
\begin{align}
 r_t=\begin{cases}
 -c\,m(\boldsymbol{\mathrm{s}}_t),&\text{if no rider converts},\\
 p_i(\lambda_i(\boldsymbol{\mathrm{s}}_t))-c\mathbbm{1}\{T=0\}-c\,m(\boldsymbol{\mathrm{s}}_t),
 &\text{if type $i$ converts and }\phi_i(\boldsymbol{\mathrm{s}}_t)=0,\\
 p_i(\lambda_i(\boldsymbol{\mathrm{s}}_t))-c\ell_{i,j}^u-c\,m(\boldsymbol{\mathrm{s}}_t-\boldsymbol e_{(j,u)}),
 &\text{if type $i$ converts and }\phi_i(\boldsymbol{\mathrm{s}}_t)=(j,u),
 \end{cases}
 \label{eq:mdp_reward}
\end{align}
where the term $c\mathbbm{1}\{T=0\}$ accounts for the new rider's first period of solo travel when there is no waiting window. Under static pricing, $\lambda_i(\boldsymbol{\mathrm{s}}_t)$ is replaced by $\lambda_i$.

Let $\Pi_t$ denote the total profit over the first $t$ periods, that is,
$\Pi_t:=\sum_{\tau=0}^{t-1}r_\tau$,
and the objective of the platform is to maximize long-run average profit, given by
\begin{align*}
\Pi_{\mathrm{stat}}:=\sup_{\substack{\boldsymbol{\phi},\\\boldsymbol{\lambda}\in[0,1]^N}}\lim_{t\to\infty}\frac1t\mathbbm E[\Pi_t^{\boldsymbol{\phi},\boldsymbol{\lambda}}],\qquad
\Pi_{\mathrm{dyn}}:=\sup_{\boldsymbol{\phi},\,\boldsymbol{\lambda}(\cdot)}\lim_{t\to\infty}\frac1t\mathbbm E[\Pi_t^{\boldsymbol{\phi},\boldsymbol{\lambda}(\cdot)}],
\end{align*}
where the two objectives correspond to static and dynamic pricing, respectively. The superscripts on $\Pi_t$ indicate the matching and pricing policies,
and the expectation is taken over all the randomness, including rider arrivals, conversions, and matching decisions. In the following analysis, we consider stationary policies and use $\boldsymbol{\mathrm{s}}=(\mathrm{s}_{i,u})_{(i,u)\in\mathcal{C}}$ to denote a generic system state.

Under any stationary matching and pricing policy, the induced Markov chain is a finite unichain (because a sufficiently long sequence of periods without arrivals empties the system with positive probability). Thus, each fixed policy has a long-run average profit independent of the initial state. If the revenue functions $\lambda_i p_i(\lambda_i)$ are finite and continuous on $[0,1]$, standard finite-state MDP results also ensure the existence of an optimal stationary policy (see, e.g., \citealt{puterman2014markov}). However, due to the exponential growth of the state space size $|\mathcal{S}|$ with respect to the maximum waiting time $T$ and the trip lengths $(\ell_i)_{i\in[N]}$, solving this MDP exactly is computationally intractable. Therefore, the subsequent analysis focuses on the fluid approximation approach under static pricing, and the value function approximation approach under dynamic pricing.

\section{Matching and Pricing Policies}
\label{sec:pricing_matching_policies}

In this section, we discuss approaches to solve the platform's joint matching and pricing problem under static and dynamic pricing separately. In \Cref{ss:pricing_matching}, we use a fluid relaxation to optimize static prices and develop a dual-based matching policy accordingly. In \Cref{ss:dynamic_bellman}, we formulate the Bellman equation for dynamic pricing and use a value-function approximation to obtain a heuristic matching and pricing policy. We also analytically compare static and dynamic pricing in a stylized model in Appendix \ref{ss:stylized_dynamic}.

\subsection{Policies under Static Pricing}
\label{ss:model_fluid}\label{ss:model_fluid_combined}
\label{ss:pricing_matching}\label{ss:pricing_matching_combined}

To solve the static pricing problem, we formulate a fluid relaxation of the MDP model whose optimal objective value provides an upper bound on optimal long-run average profit under static pricing (Proposition~\ref{prop:fluid_bound}). Define the flow variables $x_{i,j}^u$ for all $(j, u)\in\mathcal{C}$ and $i\in\mathcal{N}_{j,u}^+$, representing the average rate of matching new riders of type $i$ with existing solo riders $(j,u)$ per unit time. Define $y_i^u$ as the average rate at which existing solo riders $(i,u)\in\mathcal{C}$ remain unmatched in the current period and continue solo. Specifically, $y_i^{\ell_i-1}$ corresponds to the flow of type-$i$ riders who end up solo. We also define $y_i^{-T}$ as the average rate of new riders of type $i$ that are not immediately matched, for all $i\in[N]$.

Under the fluid relaxation, the profit maximization problem is formulated as a bilevel optimization problem:
\begin{align}
g:=\max_{\boldsymbol{\lambda}\in[0,1]^N} \quad \sum_{i\in[N]} \Lambda_i\lambda_i p_i(\lambda_i) -  C(\boldsymbol{\lambda}), \label{eq:profit_fluid}
\end{align}
where the first term represents the expected revenue collected per unit time, and $C(\boldsymbol{\lambda})$ is the optimal value of the inner-level cost minimization problem, given by
\begin{subequations}\label{eq:cb}
\begin{align}
    (\textrm{\textsf{CB}}(\boldsymbol{\lambda}))\quad\quad\quad \notag\\
    C(\boldsymbol{\lambda}) := \quad \min_{\boldsymbol{x}, \boldsymbol{y}} & \quad \sum_{i\in[N]} \sum_{u=0}^{\ell_i-1} c y_{i}^u + \sum_{i\in[N]} \sum_{\substack{(j,u)\in\mathcal{C}:\:\\i\in \mathcal{N}_{j,u}^+}} c \ell_{i,j}^u x_{i,j}^u   \label{eq:fluid_obj} \\
    \textrm{s.t.}&\quad \sum_{\substack{
    (j,u)\in\mathcal{C}:\:\\ i\in\mathcal{N}^+_{j,u}
    }}  x_{i,j}^u + y_i^{-T} = \Lambda_i\lambda_i, & \forall i\in[N],\label{eq:fluid_demand}\\
    & \quad \sum_{i\in\mathcal{N}_{j,u}^+}  x_{i,j}^u + y_j^u = y_j^{u-1}, & \forall (j,u)\in\mathcal{C}, \label{eq:fluid_flow_balance}\\
    & \quad \Lambda_i\lambda_i  y_j^u \ge \left( 1-\sum_{k\in\mathcal{N}^+_{j,u}} \Lambda_k\lambda_k \right) x_{i,j}^u, & \forall (j,u)\in\mathcal{C}, i\in\mathcal{N}^+_{j,u}, \label{eq:fluid_ratio} \\
    & \quad x_{i,j}^u \ge 0, & \forall (j,u)\in\mathcal{C}, i\in\mathcal{N}^+_{j,u}, \notag\\
    & \quad y_i^u \ge 0, & \forall i\in[N], -T \leq u \leq \ell_i-1, u\in\mathbbm{Z}. \notag
\end{align}
\end{subequations}
The objective function \eqref{eq:fluid_obj} represents the average total cost per unit time. The first term corresponds to the cost incurred by on-trip solo riders, while the second term captures the cost of serving matched trips. Constraints \eqref{eq:fluid_demand} and \eqref{eq:fluid_flow_balance} impose steady-state flow balance conditions. Specifically, \eqref{eq:fluid_demand} gives a flow balance for new requests of type $i\in[N]$ with request rate $\Lambda_i\lambda_i$. These requests are either matched with existing solo riders $(j,u)\in\mathcal{C}$ such that $i\in\mathcal{N}_{j,u}^+$, resulting in a flow of $x_{i,j}^u$, or left unmatched, resulting in a flow of $y_i^{-T}$. Constraint \eqref{eq:fluid_flow_balance} gives the flow balance for solo riders $(j,u)\in\mathcal{C}$. In the steady state, the inflow of solo riders entering the state $(j,u)$ is $y_j^{u-1}$, and they are either matched with new requests $i\in\mathcal{N}_{j,u}^+$, resulting in a flow of $x_{i,j}^u$, or left unmatched to continue solo, resulting in a flow of $y_j^u$.
Constraints \eqref{eq:fluid_ratio} are ratio constraints that limit the match rate relative to the solo flow, %
tightening the relaxation to reflect the chance that no compatible rider converts. To clarify the role of the ratio constraints, we rearrange \eqref{eq:fluid_ratio} for $x_{i,j}^u>0$ and $\Lambda_i\lambda_i>0$ as follows:
\begin{align}
\frac{y_j^u}{x_{i,j}^u} \ge \frac{1-\sum_{k\in\mathcal{N}^+_{j,u}} \Lambda_k\lambda_k}{\Lambda_i\lambda_i}, \quad \forall (j,u)\in\mathcal{C}, i\in\mathcal{N}^+_{j,u}. \label{eq:fluid_ratio_re}
\end{align}
Intuitively, when a solo rider is in state $(j,u)$, the probability of matching with a type-$i$ rider cannot exceed $\Lambda_i\lambda_i$, because a new type-$i$ rider must convert. The probability of continuing solo is at least $1-\sum_{k\in\mathcal{N}_{j,u}^+}\Lambda_k\lambda_k$, because no match is possible when no compatible rider converts. Multiplying both probability bounds by the long-run frequency with which state $(j,u)$ is occupied gives the corresponding flow bounds. Eliminating this common frequency yields \eqref{eq:fluid_ratio_re}. This constraint is analogous to the limited-patience constraints in \cite{aouad2022dynamic}, which bound matching flow relative to unmatched departures using arrival and abandonment rates in the continuous-time setting.

The following proposition states that the optimal objective value of the fluid relaxation provides an upper bound on the optimal long-run average profit under static pricing.
\begin{proposition} \label{prop:fluid_bound}
$g \geq \Pi_{\mathrm{stat}}$.
\end{proposition}
To prove Proposition~\ref{prop:fluid_bound}, we consider the long-run matching and solo flows generated by any stationary matching policy under static pricing. We show that these flows satisfy the flow balance and ratio constraints of the fluid model and reproduce the policy's long-run average revenue and cost. The static prices and corresponding flows therefore define a feasible solution of the fluid model with the same profit. The optimal fluid profit is at least the profit of this policy. This provides a theoretical basis for using the fluid model as an approximation to the MDP when designing our matching and pricing policies.

\smallskip

\paragraph{Remarks on the Model without Ratio Constraints.} In our fluid relaxation, the ratio constraints in \eqref{eq:fluid_ratio} are introduced to prevent overestimation of match rates. The following result shows that removing these constraints makes the fluid model unable to assess the value of on-trip matching or cross-type matching.
\begin{proposition}\label{prop:policy_on_trip_wo_ratio}
 Suppose $T\geq1$. For any fixed $\boldsymbol{\lambda}$, let $\widetilde C(\boldsymbol{\lambda})$ be the optimal value of \textrm{\textsf{CB}}$(\boldsymbol{\lambda})$ after removing \eqref{eq:fluid_ratio}. Then
 \begin{align}
 \widetilde C(\boldsymbol{\lambda})=\frac c2\sum_{i\in[N]}\Lambda_i\lambda_i\ell_i.
 \nonumber
 \end{align}
 There exists an optimal solution satisfying $x_{i,i}^{-T+1}=y_i^{-T}=\Lambda_i\lambda_i/2$ for all $i\in[N]$, with all other flow variables equal to zero.
\end{proposition}
Proposition \ref{prop:policy_on_trip_wo_ratio} shows that, without the ratio constraints, the fluid model admits an optimal solution that matches all converted demand in same-type pairs before departure. Because these riders share their entire trip, this solution attains the minimum possible fluid cost without any on-trip or cross-type matches. However, flow balance alone does not account for whether another rider of the same type arrives and converts within the waiting window. The resulting bound therefore overestimates matching opportunities and cannot assess the value of allowing on-trip or cross-type matching. This result explains why the ratio constraints are essential for the fluid model to account for stochastic matching availability: they limit matching flows according to arrival and conversion probabilities.

\smallskip

With the help of the fluid relaxation, we can avoid the curse of dimensionality and design scalable algorithms that yield effective matching and pricing solutions. Let $x_{i,j}^u(\boldsymbol{\lambda})$ and $y_i^u(\boldsymbol{\lambda})$ denote an optimal primal solution of $\textrm{\textsf{CB}}(\boldsymbol{\lambda})$, and let $\gamma_i(\boldsymbol{\lambda})$, $\xi_j^u(\boldsymbol{\lambda})$ and $\eta_{i,j}^u(\boldsymbol{\lambda})$ denote corresponding optimal dual variables associated with constraints \eqref{eq:fluid_demand}, \eqref{eq:fluid_flow_balance} and \eqref{eq:fluid_ratio}, respectively. We use the dual sign convention in which $\gamma_i(\boldsymbol{\lambda})$ multiplies the right-hand side minus the left-hand side of \eqref{eq:fluid_demand}, $\xi_j^u(\boldsymbol{\lambda})$ multiplies the right-hand side minus the left-hand side of \eqref{eq:fluid_flow_balance}, and $\eta_{i,j}^u(\boldsymbol{\lambda})\geq0$ multiplies the right-hand side minus the left-hand side of \eqref{eq:fluid_ratio} in the Lagrangian. According to the Envelope Theorem (see, e.g., Corollary 5 in \citealt{milgrom2002envelope}), when $C(\boldsymbol{\lambda})$ is differentiable at an interior point $\boldsymbol{\lambda}$, its gradient, denoted by $\nabla C(\boldsymbol{\lambda})$, is given by
\begin{align*}
    \nabla C(\boldsymbol{\lambda}) &= \left[ \Lambda_i \gamma_i(\boldsymbol{\lambda}) - \Lambda_i \sum_{\substack{(j,u)\in\mathcal{C}: \\ i\in\mathcal{N}^+_{j,u}}} \Big( y_j^u(\boldsymbol{\lambda}) \eta_{i,j}^u(\boldsymbol{\lambda}) + \sum_{k\in\mathcal{N}^+_{j,u}} x_{k,j}^{u}(\boldsymbol{\lambda}) \eta_{k,j}^u(\boldsymbol{\lambda}) \Big) \right]_{i\in[N]}.
\end{align*}
When $C(\boldsymbol{\lambda})$ is nondifferentiable at $\boldsymbol{\lambda}$, we continue to use the notation $\nabla C(\boldsymbol{\lambda})$, which can be interpreted as the subgradient or supergradient (whenever such gradients exist) that informs the ascent direction.

The pricing optimization problem formulated in \eqref{eq:profit_fluid} is, in general, non-convex. Recent work by \cite{chen2025linear} demonstrates the effectiveness of the Minorization-Maximization (MM) algorithm for a similar fluid matching model, subject to both flow balance and ratio constraints. Motivated by this approach, we use a cost-linearization heuristic based on \cite{chen2025linear}.
Let $\boldsymbol{\lambda}^{(k)}$ denote the solution estimate after $k$ iterations. In the $(k+1)$-th iteration, we construct a surrogate function $Q(\boldsymbol{\lambda} \mid \boldsymbol{\lambda}^{(k)})$ defined as follows:
\begin{align}
Q(\boldsymbol{\lambda} \mid \boldsymbol{\lambda}^{(k)}) &= \sum_{i\in[N]} \Lambda_i\lambda_i p_i(\lambda_i) -C(\boldsymbol{\lambda}^{(k)}) - \nabla C(\boldsymbol{\lambda}^{(k)})^{\top} (\boldsymbol{\lambda} - \boldsymbol{\lambda}^{(k)}). \label{eq:mm}
\end{align}
In \eqref{eq:mm}, we replace the cost function by its affine approximation at $\boldsymbol{\lambda}^{(k)}$. We then perform an iteration by letting $\boldsymbol{\lambda}^{(k+1)} \gets \arg\max_{\boldsymbol{\lambda} \in [0,1]^N} Q(\boldsymbol{\lambda} \mid \boldsymbol{\lambda}^{(k)})$. Under the standard revenue management assumption that $\lambda_i p_i(\lambda_i)$ is concave, maximizing $Q(\boldsymbol{\lambda} \mid \boldsymbol{\lambda}^{(k)})$ becomes a convex optimization problem and can thus be solved using computationally efficient methods. We repeat this procedure until the candidate update ceases to improve the original objective value. This method is numerically robust, and we find empirically that it converges to consistent conversion probabilities starting from different random initial solutions.

After obtaining the pricing result, denoted by $\hat{\boldsymbol{\lambda}}$, we implement a heuristic dual-based matching policy.
Whenever a new rider of type $i\in[N]$ converts, given the current state $\boldsymbol{\mathrm{s}}$, let
$$
\kappa_i(\boldsymbol{\mathrm{s}};\hat{\boldsymbol{\lambda}})
:=\min_{(j,u)\in\mathcal{C}:
i\in\mathcal{N}_{j,u}^+,\mathrm{s}_{j,u}=1}
\{c\ell_{i,j}^u-\xi_j^u(\hat{\boldsymbol{\lambda}})\},
$$
where the minimum is defined as $+\infty$ if no compatible solo rider is present. The matching policy is given by
\begin{align} \label{equ:dispatching_policy}
    \phi_i(\boldsymbol{\mathrm{s}}) \in \begin{cases}
        \arg\min_{(j,u)\in\mathcal{C}: i\in\mathcal{N}_{j,u}^+,\: \mathrm{s}_{j,u}=1} \{ c\ell_{i,j}^u -\xi_j^u(\hat{\boldsymbol{\lambda}})\}, & \text{if } \gamma_i(\hat{\boldsymbol{\lambda}}) \geq \kappa_i(\boldsymbol{\mathrm{s}};\hat{\boldsymbol{\lambda}}),\\
        \{0\}, & \text{otherwise.}
    \end{cases}
\end{align}

The dual variables $\gamma_i(\hat{\boldsymbol{\lambda}})$ and $\xi_j^u(\hat{\boldsymbol{\lambda}})$ are the shadow prices associated with the flow-balance constraints in \eqref{eq:fluid_demand} and \eqref{eq:fluid_flow_balance}, respectively. Intuitively, they measure the marginal costs of increasing the inflows of newly converted type-$i$ riders and solo riders entering state $(j,u)$, respectively, in the fluid model. The heuristic in \eqref{equ:dispatching_policy} uses these shadow prices to account for future matching opportunities when evaluating an immediate match. Specifically, it assigns an adjusted cost of $c\ell_{i,j}^u-\xi_j^u(\hat{\boldsymbol{\lambda}})-\gamma_i(\hat{\boldsymbol{\lambda}})$ to matching a new type-$i$ rider with an existing solo rider $(j,u)$. The minimum adjusted cost over available compatible riders is therefore $\kappa_i(\boldsymbol{\mathrm{s}};\hat{\boldsymbol{\lambda}})-\gamma_i(\hat{\boldsymbol{\lambda}})$. The platform matches when this quantity is nonpositive, equivalently when $\kappa_i(\boldsymbol{\mathrm{s}};\hat{\boldsymbol{\lambda}})\leq\gamma_i(\hat{\boldsymbol{\lambda}})$. Otherwise, the new rider remains unmatched.

\subsection{Policies under Dynamic Pricing}
\label{sec:dynamic_pricing}\label{ss:dynamic_bellman}\label{ss:dynamic_value}

In this subsection, we extend our policies to the dynamic pricing setting, where the platform adjusts prices according to the current system state. Under a dynamic pricing policy $\boldsymbol{\lambda}(\cdot)$, a type-$i$ rider converts with probability $\lambda_i(\boldsymbol{\mathrm{s}})$, so the corresponding request rate $\Lambda_i\lambda_i(\boldsymbol{\mathrm{s}})$ is now state-dependent. This new dependence is not captured by the fixed request rates used in the static fluid relaxation.

To account for this dependence, we first formulate the Bellman equation under dynamic pricing. Let $V(\boldsymbol{\mathrm{s}})$ denote the relative value of state $\boldsymbol{\mathrm{s}}$. Using the state transitions in \eqref{eq:state_transition} and the period profit in \eqref{eq:mdp_reward}, the average-reward Bellman equation is
\begin{align}
 &V(\boldsymbol{\mathrm{s}})+\Pi_{\mathrm{dyn}}\notag\\
 ={}&\max_{\boldsymbol0\leq\boldsymbol\lambda(\boldsymbol{\mathrm{s}})\leq\boldsymbol1}
 \Bigg\{\sum_{i\in[N]}\Lambda_i\lambda_i(\boldsymbol{\mathrm{s}})
 \Bigg[p_i(\lambda_i(\boldsymbol{\mathrm{s}}))+\max\Bigg\{
 \max_{\substack{(j,u)\in\mathcal C:\\i\in\mathcal N_{j,u}^+,\ \mathrm{s}_{j,u}=1}}
 \underbrace{\Big[V(f(\boldsymbol{\mathrm{s}}-\boldsymbol e_{(j,u)}))
 -c\ell_{i,j}^u-c\,m(\boldsymbol{\mathrm{s}}-\boldsymbol e_{(j,u)})\Big]}_{\text{match $i$ with $(j,u)$}},\notag\\
 &\hspace{3.4cm}\underbrace{V(f(\boldsymbol{\mathrm{s}})+\boldsymbol e_{(i,-T+1)})
 -c\mathbbm{1}\{T=0\}-c\,m(\boldsymbol{\mathrm{s}})}_{\text{leave $i$ unmatched}}\Bigg\}\Bigg]\notag\\
 &\qquad+\left(1-\sum_{i\in[N]}\Lambda_i\lambda_i(\boldsymbol{\mathrm{s}})\right)
 \underbrace{\Big[V(f(\boldsymbol{\mathrm{s}}))-c\,m(\boldsymbol{\mathrm{s}})\Big]}_{\text{no rider converts}}\Bigg\},
 \qquad\boldsymbol{\mathrm{s}}\in\mathcal S.
 \label{eq:dynamic_bellman}
\end{align}
The inner maximization compares matching a converted rider with leaving that rider unmatched, accounting for current cost and the relative value of the next state. Specifically, matching incurs the shared-trip cost $c\ell_{i,j}^u$ and removes the existing solo rider from the state. Leaving the new rider unmatched adds that rider to the next state, with $c\mathbbm{1}\{T=0\}$ accounting for the first period of solo travel when there is no pre-trip waiting window. The outer maximization chooses the conversion probabilities, and thus the prices, before conversions occur. The last term accounts for periods in which no rider converts.

Solving \eqref{eq:dynamic_bellman} involves determining the value function $V(\boldsymbol{\mathrm{s}})$ at every system state. Therefore, the number of decision variables grows exponentially with the waiting window and the trip lengths, making the exact solution computationally intractable. To develop a computationally efficient heuristic policy under dynamic pricing, we reuse the dual variables from static pricing to approximate the value function. Specifically, let $\hat{\boldsymbol\lambda}$ be the static pricing solution. As discussed in \Cref{ss:pricing_matching}, the dual variable $\xi_j^u(\hat{\boldsymbol\lambda})$ is the shadow price associated with an additional solo rider $(j,u)$. We use its negative as a proxy for the contribution of a rider in state $(j,u)$ to future profit. Summing these contributions gives
\begin{align}
 \widehat V(\boldsymbol{\mathrm{s}})
 =-\sum_{(j,u)\in\mathcal C}\xi_j^u(\hat{\boldsymbol\lambda})\mathrm{s}_{j,u}.
 \label{eq:dual_value_approximation}
\end{align}
For brevity, write $\xi_j^u=\xi_j^u(\hat{\boldsymbol\lambda})$ and set $\xi_j^{\ell_j}=0$. Substituting \eqref{eq:dual_value_approximation} into \eqref{eq:dynamic_bellman} gives the dynamic pricing policy $\widehat{\boldsymbol\lambda}(\cdot)$, with
\begin{equation}
\begin{aligned}
 &\widehat\lambda_i(\boldsymbol{\mathrm{s}})
 \in\arg\max_{\lambda_i\in[0,1]}\lambda_i\Bigg[p_i(\lambda_i)\\
 &\qquad+\max\Bigg\{-c\mathbbm{1}\{T=0\}-\xi_i^{-T+1},\quad
 \max_{\substack{(j,u)\in\mathcal C:\\i\in\mathcal N_{j,u}^+,\ \mathrm{s}_{j,u}=1}}
 \Big(\xi_j^{u+1}+c\mathbbm{1}\{u\geq0\}-c\ell_{i,j}^u\Big)
 \Bigg\}\Bigg].
\end{aligned}
 \nonumber
\end{equation}
Upon conversion of a new type-$i$ rider, the platform then implements the matching policy attaining the inner maximum. The first option leaves the new rider unmatched, incurring the current solo-travel cost $c\mathbbm{1}\{T=0\}$ and the estimated future cost $\xi_i^{-T+1}$ in the next state. The second option forms a shared ride at cost $c\ell_{i,j}^u$, but avoids both the existing rider's current solo-travel cost $c\mathbbm{1}\{u\geq0\}$ and the estimated future cost $\xi_j^{u+1}$ in the next state.

For a given state, the terms in the inner maximum do not depend on the conversion probability being optimized. Therefore, when $\lambda_i p_i(\lambda_i)$ is concave in $\lambda_i$, each $\widehat\lambda_i(\boldsymbol{\mathrm{s}})$ can be obtained by solving a one-dimensional concave maximization problem. The corresponding price is $p_i(\widehat\lambda_i(\boldsymbol{\mathrm{s}}))$. Because the static dual variables are reused across states, the platform can update matching and pricing decisions without repeatedly solving the fluid model, yielding a computationally efficient heuristic policy under dynamic pricing.

\section{Computational Experiments}
\label{sec:computational}
In this section, we evaluate our matching and pricing policies on real ridesharing data in Chicago.
We describe our data, policies, and performance metrics  in \Cref{ss:setup}. Then, we benchmark our dual-based matching policy against batch-matching policies inspired by the literature and common practice in \Cref{ss:benchmark-results}. After demonstrating the effectiveness of our dual-based matching policy, we discuss the impact of the on-trip matching phase on the platform and riders in \Cref{ss:impact-results}. Although we usually refer to our policies as the combined, pre-trip, and on-trip matching policies for brevity, our experiments, unless otherwise specified, compute matching and pricing policies {jointly} as described in \Cref{ss:model_fluid,sec:dynamic_pricing}.

\subsection{Setup} \label{ss:setup}
The data set provided by Chicago Data Portal\footnote{Source: \url{https://tinyurl.com/4dcn9d5d}, accessed 2023-06-30.} contains the pickup and dropoff time and location of each trip, along with an indicator of whether the rider is willing to share the ride (``authorized"). We focus on authorized trips (about 15\% of all trips) recorded over eight weeks in October and November of 2019 before COVID-19. The first seven weeks serve as the training set, and the final week serves as the test set.

\begin{figure}[htbp]
     \centering
     \begin{subfigure}[b]{0.32\textwidth}
        \begin{center}  \includegraphics[height=6cm]{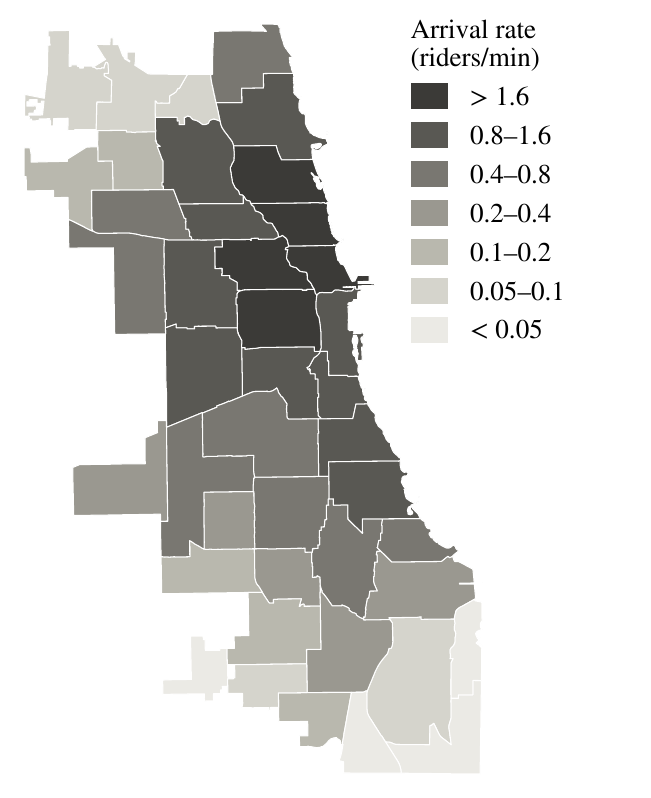}
          \caption{Arrival Rates}
          \label{fig:arrival_rate}
        \end{center}
     \end{subfigure}
     \begin{subfigure}[b]{0.32\textwidth}
        \begin{center}  \includegraphics[height=6cm]{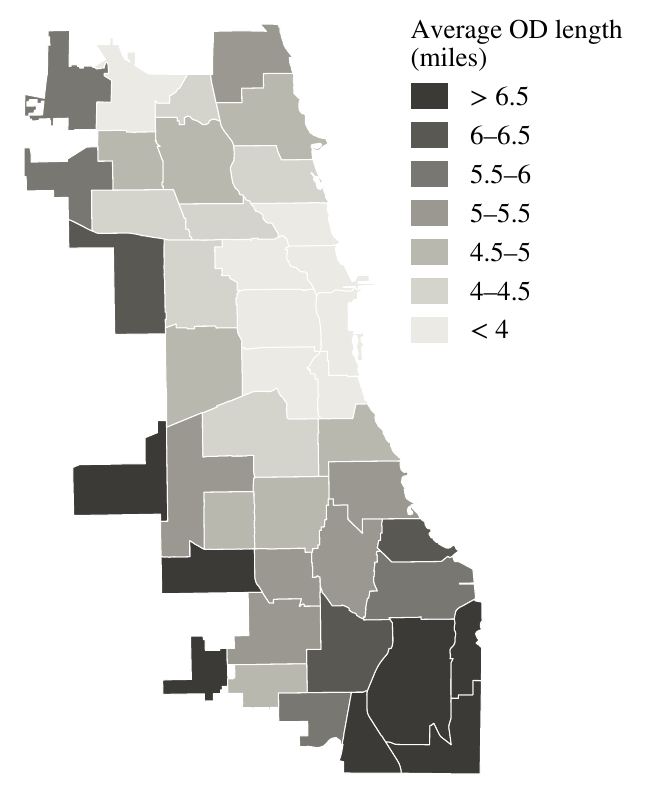}
          \caption{Trip Lengths}
          \label{fig:length}
        \end{center}
     \end{subfigure}
     \begin{subfigure}[b]{0.32\textwidth}
        \begin{center}  \includegraphics[height=6cm]{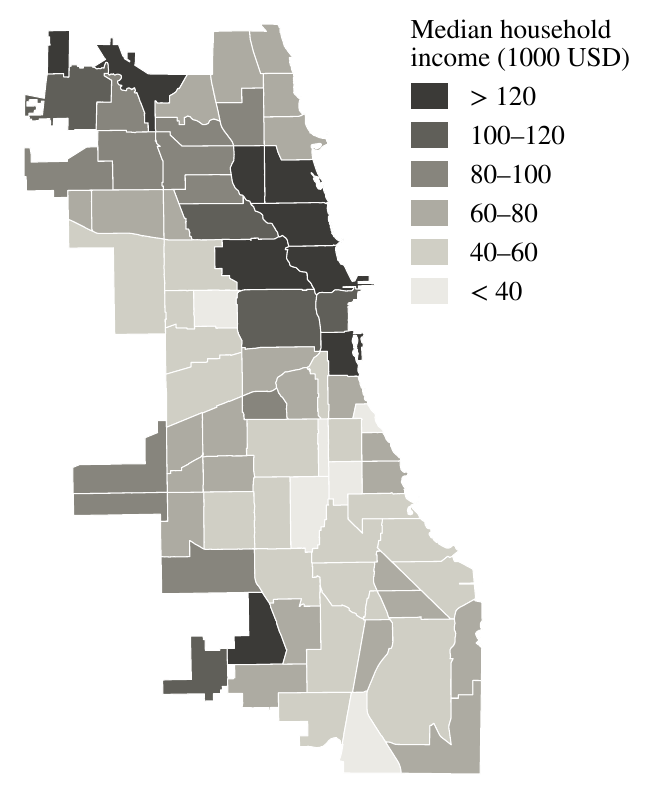}
          \caption{Household Incomes}
          \label{fig:income}
        \end{center}
     \end{subfigure}
    \caption{Geographical distribution of arrival rates, trip lengths, and household incomes in Chicago neighborhoods. }
    \label{fig:arrival_rate_and_length}
\end{figure}

\paragraph{Rider Demand.} Since the data set does not include riders who arrive without requesting a ride, we use trip data as a proxy for arrival data. We focus on the one-hour time window of 7:30 -- 8:30 a.m. on Monday for each week. The total arrival rate in this time window is 28.25 riders per minute in the training set, and 29.43 riders per minute in the test set.

To create a set of request types, we follow the two-step clustering method of \cite{aouad2022dynamic}. First, we cluster the 76 community areas of the City of Chicago\footnote{Source: \url{https://tinyurl.com/4yxcvfaj}, accessed 2023-06-30.} (excluding the O'Hare International Airport) into 42 zones. Then, we cluster the riders originating from each pickup zone by their dropoff locations using $k$-means clustering. We obtain in total 259 rider types. Riders in the test set are mapped to the closest of the 259 types from the training set.

The geographic distributions of arrival rate, trip length, and household income are shown in \Cref{fig:arrival_rate_and_length}.\footnote{\Cref{fig:arrival_rate,fig:length} use the clustered zones described in the previous paragraph. \Cref{fig:income} uses the original community areas in the data from  \url{https://metop.io/insights/agqp}.} The downtown zones (center-right) have the highest demand, the shortest average trip lengths, and the highest household incomes; the zones in the outskirts have the lowest demand, the longest average trip lengths, and the lowest household incomes. Moreover, the zones in the outskirts also tend to have the lowest levels of transit accessibility \citep{ermagun2020equity}. %

\paragraph{Road Network.} Recall that in the discrete-time setting described in \Cref{sec:model}, time periods are short enough that there is at most one arrival per period, with the probability of an arrival $\sum_{i \in [N]} \Lambda_i$ $\ll 1$. Accordingly, we set the period length to be $\delta = \nicefrac{0.1}{\sum_{i \in [N]} \Lambda_i} = \nicefrac{0.1}{28.25 \text{ riders} \cdot \text{min}^{-1}} = \nicefrac{0.1}{0.47 \text{ riders} \cdot \text{s}^{-1}} = 0.21 \text{ s}$. %
We then calibrate link travel times using a common speed of $24\text{ mph} \approx 10.73\ \text{m/s}$ (\citealt{chicago_speed}). Under this calibration, a vehicle travels a distance of $(10.73\ \text{m/s})(0.21 \text{ s}) \approx 2.28\ \text{m}$ per period. We therefore discretize the Chicago road network (obtained using the \texttt{OSMnx} package, \cite{boeing2025modeling}) into segments of length 2.28 m, resulting in a road network with 4,254,632 nodes and 8,606,506 edges. %

\paragraph{Payments and Costs.} We assume riders'  willingness-to-pay is uniformly distributed between $0$ and $1$ in dollars per mile, so that $\lambda(p) = 1-p$.
We consider costs $c \in\{0.7, 0.9, 1.1\}$ in dollars per mile.\footnote{This is representative of Uber driver per-distance rate in regions in the US, see \url{https://rideshareguru.com/uber-driver-pay-rates-by-city-per-mile-and-per-minute/}, accessed 2023-06-30. The corresponding per-period cost in the model equals this rate multiplied by the distance traveled in one period.}

\paragraph{Policies.} %

We evaluate the following matching policies.
\begin{itemize}
\item \textit{Combined matching.} This policy implements both pre-trip and on-trip matching by following the dual-based rule \eqref{equ:dispatching_policy}, as described in \Cref{ss:model_fluid}.%

\item \textit{Pre-trip matching.} This policy implements pre-trip matching alone by adapting the policy described in \Cref{ss:model_fluid} with $\mathcal{N}_{j,u}^+=\emptyset$ for all $u> 0$ for all $j\in[N]$.

\item \textit{Batch-matching benchmarks.} We also benchmark against two batch-matching policies inspired by the literature and common practice  \citep{yan2020dynamic,aouad2022dynamic,eom2024batching}. These policies implement both pre-trip and on-trip matching, but differ in their matching decisions. The first policy is a myopic policy that solves a maximum-weight matching every batch. It implements on-trip matching by allowing riders who are initially dispatched solo to be considered in future batches (as in, for example, \cite{alonso2017demand}), but does not explicitly model this future possibility when solving the matching problem. The second policy is a rolling horizon policy that projects the value of riders who are initially dispatched solo \citep{bent2004ScenarioBased, homsi2024Rolling}.
We describe the policies in more detail in Appendix \ref{ss:benchmarks}.

\end{itemize}

Throughout the numerical study, we primarily use the static
prices $p_i(\hat\lambda_i)$ obtained by solving the pricing problem \eqref{eq:profit_fluid}.
The results under dynamic pricing are qualitatively similar to those under static pricing and lead to similar conclusions. To avoid repetition, we report only a summary table regarding dynamic pricing at the end of \Cref{ss:impact-results}.

We consider waiting window lengths of $\{0.0, 0.5, 1.0, 1.5, 2.0\}$ minutes, which correspond to $T\in\{0, 141, 282, 423, 564 \}$ periods.  All policies are trained on the training data and evaluated via simulation on a semi-synthetic training data set and the test data. For the training evaluation, we generate a synthetic data set of $10^6$ periods (about 59 hours) using arrival rates estimated from the true training set. For the test evaluation, we use the actual arrival sequence from the test week across 100 repetitions, re-drawing riders' willingness to pay in every run. The test results are our primary results and are reported in this section; the training results are reported in Appendix~\ref{sec:training_result}.

\paragraph{Performance Metrics.}  Performance metrics include the profit rate per minute, average quoted price (per rider, %
per mile), average payment (per converted request, %
per mile), throughput (requests per minute), match rate (number of matched requests divided by total requests), on-trip portion (the fraction of matches made when one of the requests had originally been dispatched solo), cost efficiency (one minus total dispatching cost divided by a hypothetical cost if all requests were dispatched solo),\footnote{Cost efficiency can be interpreted as the percentage improvement in cost for shared rides relative to solo rides. Under two-party matching, cost efficiency ranges from zero (pure solo rides) to 0.5 (ideal shared rides, where every rider is dispatched in a shared ride with another rider sharing the exact origin and destination).} and  average detour rate (over matched requests, and defined as the total detour distance divided by a hypothetical total distance of the matched requests assuming all were dispatched solo).

\subsection{Policy Performance} \label{ss:benchmark-results}

First, to evaluate our dual-based matching policy, we compare its performance against the myopic and rolling horizon batch-matching benchmarks under the test data in \Cref{tab:test_benchmarks}. (\Cref{tab:train_benchmarks} in Appendix \ref{sec:training_result} presents the corresponding training results.) All three policies decide both pre-trip and on-trip matches, and use the same static prices in order to focus the comparison on the matching decisions.

Our dual-based policy attains the highest profit across all parameters. Although our policy has a slightly lower match rate than the batch policies in most cases, its cost efficiency is higher and detour rate is lower, showing that the higher profit is driven by \emph{quality} rather than \emph{quantity} of matching. This is especially true as compared to the myopic batching policy, which
matches whatever compatible riders are present in a batch, thereby over-committing to low-efficiency pre-trip matches and failing to reserve requests for later on-trip matches.

Although the rolling horizon batching policy
narrows the gap by valuing the option to wait, its relative improvement over the myopic policy declines as the waiting window increases.
This is because batch matching policies match only at epochs, so compatible matches may pass if the batch window is too long; this is especially detrimental for on-trip riders, whose compatibility with other riders is continuously changing. The dual-based policy, instead, evaluates matches at every conversion, which helps it capture transient on-trip matching opportunities. Despite its higher decision frequency, our dual-based policy is also cheaper to run: each decision is a cheap scan over available partners, whereas the rolling horizon benchmark solves a scenario-based program at every epoch.

In summary, our dual-based policy makes higher quality matches relative to batch matching benchmarks by (i) explicitly valuing the future on-trip value of a solo dispatch, and (ii) making cheaper, frequent decisions that help it respond to the changing states of on-trip riders.

\begin{table}[htbp]
  \centering
  \scalebox{0.68}
  {
    \begin{tabular}{clrrrrrrrrrrrr}
    \toprule
    \multicolumn{2}{c}{\textbf{Cost} (\$/mile)} & \multicolumn{4}{c}{$c=0.7$} & \multicolumn{4}{c}{$c=0.9$} & \multicolumn{4}{c}{$c=1.1$} \\
    \cmidrule(lr){3-6} \cmidrule(lr){7-10} \cmidrule(lr){11-14}
    \multicolumn{2}{c}{\textbf{Waiting Window} (min)} & \multicolumn{1}{c}{$0.5$} & \multicolumn{1}{c}{$1$} & \multicolumn{1}{c}{$1.5$} & \multicolumn{1}{c}{$2$} & \multicolumn{1}{c}{$0.5$} & \multicolumn{1}{c}{$1$} & \multicolumn{1}{c}{$1.5$} & \multicolumn{1}{c}{$2$} & \multicolumn{1}{c}{$0.5$} & \multicolumn{1}{c}{$1$} & \multicolumn{1}{c}{$1.5$} & \multicolumn{1}{c}{$2$} \\
    \cmidrule(r){1-14}
    \multirow{5}[0]{*}{\shortstack{\textbf{Profit}\\(\$/min)}} & Dual-Based & 5.289 & 5.751 & 6.157 & 6.478 & 1.679 & 2.031 & 2.320 & 2.553 & 0.042 & 0.153 & 0.263 & 0.407 \\
          & Batch (Myopic) & 4.390 & 4.906 & 5.354 & 5.721 & 1.029 & 1.363 & 1.630 & 1.947 & 0.000 & 0.000 & 0.000 & 0.060 \\
          && (-17\%) & (-15\%) & (-13\%) & (-12\%) & (-39\%) & (-33\%) & (-30\%) & (-24\%) & (-100\%) & (-100\%) & (-100\%) & (-85\%) \\
          & Batch (Rolling) & 5.119 & 5.466 & 5.772 & 6.052 & 1.539 & 1.806 & 2.002 & 2.255 & 0.023 & 0.071 & 0.101 & 0.210 \\
          && (-3\%) & (-5\%) & (-6\%) & (-7\%) & (-8\%) & (-11\%) & (-14\%) & (-12\%) & (-45\%) & (-53\%) & (-62\%) & (-48\%) \\
    \cmidrule(r){1-14}
    \multirow{3}[0]{*}{\shortstack{\textbf{Match Rate}}} & Dual-Based & 0.614 & 0.652 & 0.680 & 0.697 & 0.617 & 0.663 & 0.693 & 0.709 & 0.733 & 0.735 & 0.754 & 0.792 \\
          & Batch (Myopic) & 0.675 & 0.704 & 0.728 & 0.740 & 0.690 & 0.720 & 0.741 & 0.758 & - & - & - & 0.850 \\
          & Batch (Rolling) & 0.623 & 0.660 & 0.692 & 0.711 & 0.628 & 0.675 & 0.707 & 0.730 & 0.731 & 0.734 & 0.787 & 0.819 \\
    \cmidrule(r){1-14}
    \multirow{3}[0]{*}{\shortstack{\textbf{On-Trip Portion}}} & Dual-Based & 0.679 & 0.488 & 0.381 & 0.318 & 0.723 & 0.511 & 0.393 & 0.326 & 0.815 & 0.647 & 0.467 & 0.341 \\
          & Batch (Myopic) & 0.515 & 0.299 & 0.197 & 0.136 & 0.600 & 0.357 & 0.226 & 0.159 & - & - & - & 0.193 \\
          & Batch (Rolling) & 0.725 & 0.538 & 0.417 & 0.330 & 0.771 & 0.569 & 0.438 & 0.351 & 0.837 & 0.684 & 0.509 & 0.369 \\
    \cmidrule(r){1-14}
    \multirow{3}[0]{*}{\shortstack{\textbf{Cost Efficiency}}} & Dual-Based & 0.180 & 0.206 & 0.228 & 0.245 & 0.176 & 0.203 & 0.224 & 0.240 & 0.223 & 0.232 & 0.255 & 0.277 \\
          & Batch (Myopic) & 0.134 & 0.165 & 0.189 & 0.209 & 0.134 & 0.163 & 0.185 & 0.207 & - & - & - & 0.236 \\
          & Batch (Rolling) & 0.171 & 0.192 & 0.209 & 0.225 & 0.167 & 0.190 & 0.206 & 0.224 & 0.215 & 0.217 & 0.232 & 0.253 \\
    \cmidrule(r){1-14}
    \multirow{3}[0]{*}{\shortstack{\textbf{Detour Rate}\\(matched requests)}} & Dual-Based & 0.093 & 0.086 & 0.079 & 0.071 & 0.092 & 0.085 & 0.079 & 0.074 & 0.056 & 0.070 & 0.067 & 0.067 \\
          & Batch (Myopic) & 0.148 & 0.122 & 0.109 & 0.098 & 0.153 & 0.127 & 0.113 & 0.100 & - & - & - & 0.102 \\
          & Batch (Rolling) & 0.099 & 0.093 & 0.088 & 0.082 & 0.098 & 0.093 & 0.089 & 0.083 & 0.064 & 0.081 & 0.084 & 0.080 \\
    \bottomrule
    \end{tabular}
  }
  \caption{
  Performance of dual-based policy against batch matching benchmarks on Chicago data (test set).
  }
  \label{tab:test_benchmarks}
\end{table}

\subsection{The Impact of On-Trip Matching} \label{ss:impact-results}
Having demonstrated the effectiveness of our dual-based policy in matching, we now optimize matching and pricing jointly. In this section, we compare pre-trip and combined matching to understand how the addition of on-trip matching impacts the platform and riders.

The performance metrics of the pre-trip and combined matching policies in the test data are presented in \Cref{tab:test_matching} (the training results are presented in \Cref{tab:train_matching} in Appendix \ref{sec:training_result}). The results show that combined matching improves profit across most instances compared with pre-trip matching, except when $T=0$ and $c=1.1$, where both policies yield zero profit. The gains are particularly pronounced in operationally challenging settings when the waiting window $T$ is relatively small and the cost $c$ is relatively high. %
In particular, when $T=0$, pre-trip matching reduces to pure solo rides (i.e., all requests are dispatched solo immediately upon arrival, and remain solo). Then, all riders are quoted the price $p=\min\{1,(1+c)/2\}$ (in dollars per mile), which follows from solving the simple profit maximization problem $\max_{p \in [0, 1]} \lambda(p) (p-c)$, with $\lambda(p)=1-p$ under our willingness-to-pay model. Therefore, when $c \geq 1$, all service is shut down, as the cost can never be recouped with only solo rides.

\begin{table}[htbp]
  \centering
  \scalebox{0.6}
  {
    \begin{tabular}{clrrrrrrrrrrrrrrr}
    \toprule
    \multicolumn{2}{c}{\textbf{Cost} (\$/mile)} & \multicolumn{5}{c}{$c=0.7$} & \multicolumn{5}{c}{$c=0.9$} & \multicolumn{5}{c}{$c=1.1$} \\
    \cmidrule(lr){3-7} \cmidrule(lr){8-12} \cmidrule(lr){13-17}
    \multicolumn{2}{c}{\textbf{Waiting Window} (min)} & \multicolumn{1}{c}{$0$} & \multicolumn{1}{c}{$0.5$} & \multicolumn{1}{c}{$1$} & \multicolumn{1}{c}{$1.5$} & \multicolumn{1}{c}{$2$} & \multicolumn{1}{c}{$0$} & \multicolumn{1}{c}{$0.5$} & \multicolumn{1}{c}{$1$} & \multicolumn{1}{c}{$1.5$} & \multicolumn{1}{c}{$2$} & \multicolumn{1}{c}{$0$} & \multicolumn{1}{c}{$0.5$} & \multicolumn{1}{c}{$1$} & \multicolumn{1}{c}{$1.5$} & \multicolumn{1}{c}{$2$} \\
    \cmidrule(r){1-17}
    \multirow{3}[0]{*}{\shortstack{\textbf{Profit}\\(\$/min)}}
    & Combined & 4.750 & 5.289 & 5.751 & 6.157 & 6.478 & 1.343 & 1.679 & 2.031 & 2.320 & 2.553 & 0.000 & 0.042 & 0.153 & 0.263 & 0.407 \\
    & Pre-trip & 2.734 & 4.046 & 4.872 & 5.434 & 5.864 & 0.302 & 0.743 & 1.252 & 1.696 & 2.039 & 0.000 & 0.000 & 0.000 & 0.000 & 0.017 \\
    &  & (-42\%) & (-24\%) & (-15\%) & (-12\%) & (-9\%) & (-77\%) & (-56\%) & (-38\%) & (-27\%) & (-20\%) & - & (-100\%) & (-100\%) & (-100\%) & (-96\%) \\
    \cmidrule(r){1-17}
    \multirow{3}[0]{*}{\shortstack{\textbf{Quoted Price}\\(\$/mile)}}
    & Combined & 0.800 & 0.781 & 0.767 & 0.758 & 0.750 & 0.892 & 0.873 & 0.856 & 0.845 & 0.836 & 1.000 & 0.984 & 0.962 & 0.949 & 0.934 \\
    & Pre-trip & 0.850 & 0.798 & 0.777 & 0.765 & 0.756 & 0.950 & 0.902 & 0.873 & 0.857 & 0.846 & 1.000 & 1.000 & 1.000 & 1.000 & 0.988 \\
    & & (+6\%) & (+2\%) & (+1\%) & (+1\%) & (+1\%) & (+6\%) & (+3\%) & (+2\%) & (+1\%) & (+1\%) & (0\%) & (+2\%) & (+4\%) & (+5\%) & (+6\%) \\
    \cmidrule(r){1-17}
    \multirow{3}[0]{*}{\shortstack{\textbf{Payment}\\(\$/mile)}}
    & Combined & 0.792 & 0.772 & 0.757 & 0.747 & 0.739 & 0.871 & 0.850 & 0.830 & 0.815 & 0.805 & - & 0.880 & 0.876 & 0.850 & 0.833 \\
    & Pre-trip & 0.850 & 0.791 & 0.767 & 0.754 & 0.744 & 0.950 & 0.887 & 0.845 & 0.824 & 0.812 & - & - & - & - & 0.773 \\
    & & (+7\%) & (+2\%) & (+1\%) & (+1\%) & (+1\%) & (+9\%) & (+4\%) & (+2\%) & (+1\%) & (+1\%) & - & - & - & - & (-7\%) \\
    \cmidrule(r){1-17}
    \multirow{3}[0]{*}{\shortstack{\textbf{Throughput}\\(requests/min)}}
    & Combined & 5.867 & 6.423 & 6.811 & 7.095 & 7.329 & 3.152 & 3.735 & 4.213 & 4.557 & 4.817 & 0.000 & 0.481 & 1.101 & 1.500 & 1.920 \\
    & Pre-trip & 4.387 & 5.920 & 6.534 & 6.893 & 7.157 & 1.459 & 2.850 & 3.700 & 4.198 & 4.522 & 0.000 & 0.000 & 0.000 & 0.000 & 0.358 \\
    & & (-25\%) & (-8\%) & (-4\%) & (-3\%) & (-2\%) & (-54\%) & (-24\%) & (-12\%) & (-8\%) & (-6\%) & - & (-100\%) & (-100\%) & (-100\%) & (-81\%) \\
    \cmidrule(r){1-17}
    \multirow{2}[0]{*}{\shortstack{\textbf{Match Rate}}}
    & Combined & 0.566 & 0.614 & 0.652 & 0.680 & 0.697 & 0.570 & 0.617 & 0.663 & 0.693 & 0.709 & - & 0.733 & 0.735 & 0.754 & 0.792 \\
    & Pre-trip & 0.000 & 0.391 & 0.526 & 0.592 & 0.629 & 0.000 & 0.342 & 0.504 & 0.593 & 0.643 & - & - & - & - & 0.650 \\
    \cmidrule(r){1-17}
    \shortstack{\textbf{On-Trip Portion}} & Combined & 1.000 & 0.679 & 0.488 & 0.381 & 0.318 & 1.000 & 0.723 & 0.511 & 0.393 & 0.326 & - & 0.815 & 0.647 & 0.467 & 0.341 \\
    \cmidrule(r){1-17}
    \multirow{2}[0]{*}{\shortstack{\textbf{Cost Efficiency}}}
    & Combined & 0.149 & 0.180 & 0.206 & 0.228 & 0.245 & 0.148 & 0.176 & 0.203 & 0.224 & 0.240 & - & 0.223 & 0.232 & 0.255 & 0.277 \\
    & Pre-trip & 0.000 & 0.103 & 0.157 & 0.190 & 0.214 & 0.000 & 0.080 & 0.144 & 0.184 & 0.209 & - & - & - & - & 0.325 \\
    \cmidrule(r){1-17}
    \multirow{2}[0]{*}{\shortstack{\textbf{Detour Rate}\\(matched requests)}}
    & Combined & 0.100 & 0.093 & 0.086 & 0.079 & 0.071 & 0.099 & 0.092 & 0.085 & 0.079 & 0.074 & - & 0.056 & 0.070 & 0.067 & 0.067 \\
    & Pre-trip & - & 0.100 & 0.088 & 0.078 & 0.070 & - & 0.106 & 0.088 & 0.080 & 0.075 & - & - & - & - & 0.000 \\
    \bottomrule
    \end{tabular}
  }
  \caption{
  Performance of combined matching against pre-trip matching on Chicago data (test set).
  }
  \label{tab:test_matching}
\end{table}

Meanwhile, combined matching reduces to pure on-trip matching when $T = 0$ (i.e., all requests are either matched to pre-existing on-trip riders or dispatched solo immediately upon arrival, and riders dispatched solo can continue to be matched while on-trip). Although combined matching is also unprofitable when $T = 0$ and $c = 1.1$, there exist longer waiting windows (0.5 to 1.5 minutes) for which combined matching is profitable while pre-trip matching remains unprofitable. In fact, as reflected in the on-trip portion $(<1)$ in these settings, adding on-trip matching can also add some pre-trip matching, despite pre-trip matching being unprofitable alone. This is because on-trip matching serves as a backup option after initial solo dispatch. Therefore, the platform is willing to accept more requests, which brings in opportunities to match requests pre-trip.

The profit gains under combined matching are largely driven by {improved matching quantity and quality, as shown by substantial increases in} match rate and cost efficiency. There is only one exception to the improvements in cost efficiency, which arises when $T = 2$ minutes and $c = 1.1$. In this setting, the pricing policy under pre-trip matching essentially shuts down all but a small number of closely-located downtown riders that can effectively self-match, thereby achieving high cost efficiency at the expense of low throughput.

Overall, the improvements in cost efficiency allow the platform to lower the prices quoted to riders and the resulting payments. The lower prices reflect the importance of making matching and pricing decisions \emph{jointly}. Improving matching (via the addition of on-trip matching) then promotes a virtuous cycle: the platform can afford to lower prices, which then increases throughput and makes it even easier for the platform to find higher-quality matches. %

\paragraph{Impact on riders.} As previously mentioned, on-trip matching can benefit riders by enabling the platform to lower prices. However, beyond lower prices, on-trip matching can also benefit riders by reducing waiting windows. To show this, \Cref{fig:profit_comparison} compares the profits of combined and pre-trip matching under different waiting periods, focusing on cost $c = 0.7$. Blue means that the profit of combined matching is higher, and red means that the profit of pre-trip matching is higher.
Notably, the paler cells on the chart show the waiting windows required for pre-trip and combined matching to attain the same profit. Overall, combined matching can decrease the waiting window by roughly 0.5 to 1 minute and still attain a profit comparable to that of pre-trip matching. For example, from the profit numbers in \Cref{tab:test_matching}, pre-trip matching needs a 2-minute waiting window to attain a profit of \$5.864 per minute, while combined matching needs only a 1-minute window to attain a comparable profit of \$5.751 per minute (a gap of only 1.9\%). Hence, rather than absorbing the gains in profit, the platform can redistribute its gains back to riders in the form of substantially shorter pre-trip waiting windows.

\begin{figure}[htbp]
\begin{center}  \includegraphics[width=0.6\textwidth]
{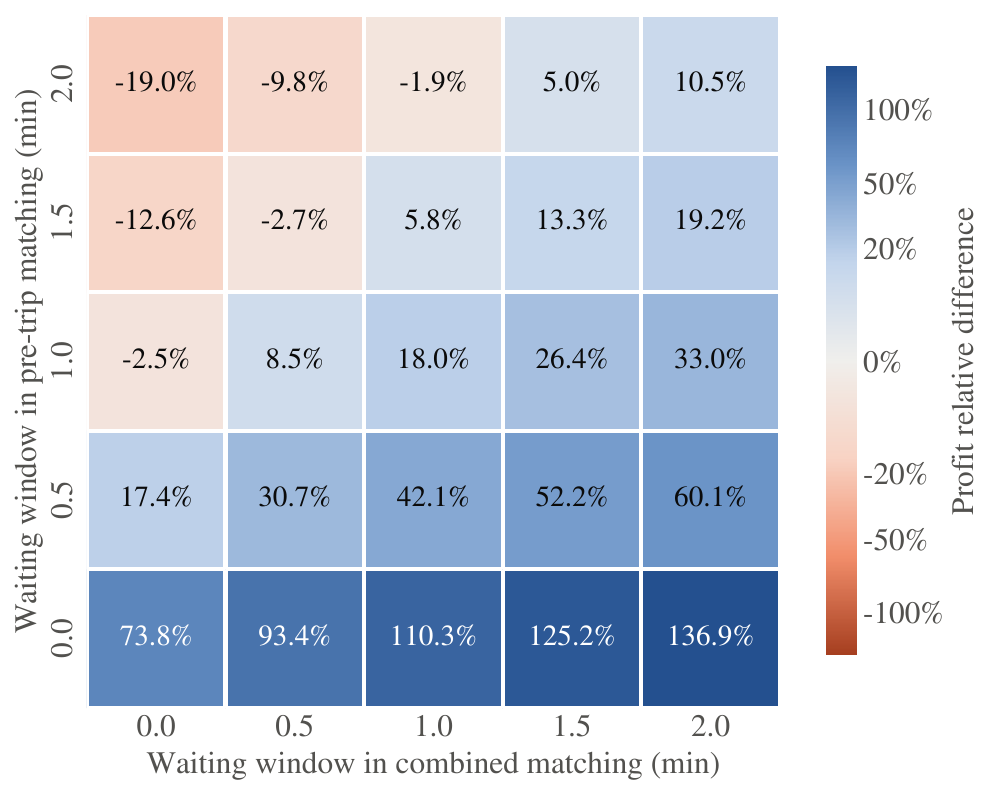}
  \caption{Profit comparison between combined and pre-trip matching policies under different waiting periods ($c=0.7$).
  }
  \label{fig:profit_comparison}
\end{center}
\end{figure}

\paragraph{{Regional Implications.}} Although matching quality increases and prices decrease overall, we find variation in these trends across the various zones of Chicago. In \Cref{fig:maps}, we show  average quoted price, match rate, and cost efficiency by zone, for a waiting window of 1 minute and $c=0.7$. Each row in \Cref{fig:maps} contains three subfigures for a particular metric: one for each of pre-trip and combined matching, and one for the comparison (relative difference for quoted price, and absolute difference for match rate and cost efficiency).

\begin{figure}[htbp]
    \centering
    \begin{subfigure}[b]{\textwidth}
        \centering
        \includegraphics[width=0.9\textwidth]{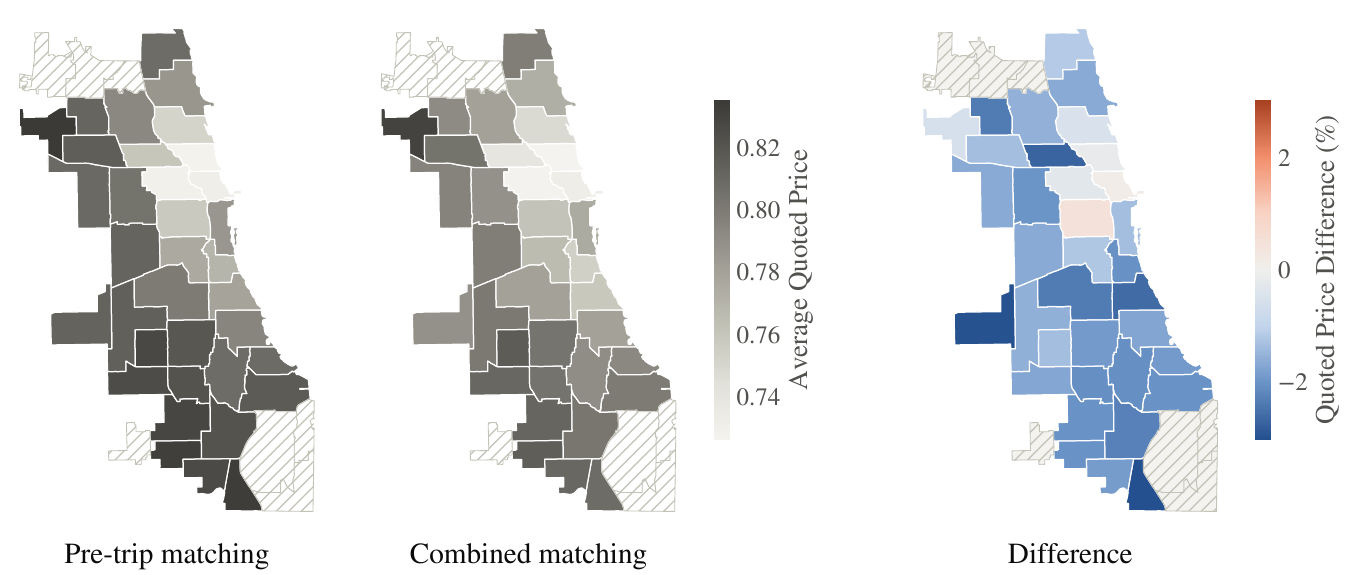}
        \caption{Quoted price}
        \label{fig:Chicago_price}
    \end{subfigure}

    \vspace{0.6em}
    \begin{subfigure}[b]{\textwidth}
        \centering
        \includegraphics[width=0.95\textwidth]{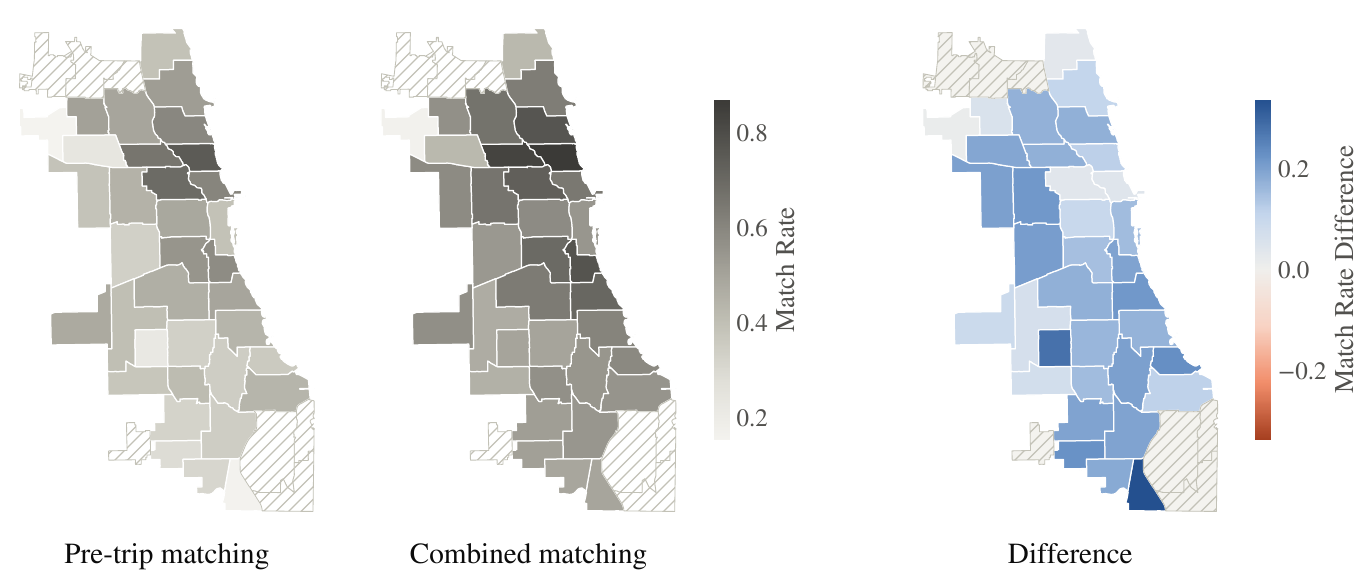}
        \caption{Match rate}
        \label{fig:Chicago_mr}
    \end{subfigure}

    \vspace{0.6em}
    \begin{subfigure}[b]{\textwidth}
        \centering
        \includegraphics[width=0.95\textwidth]{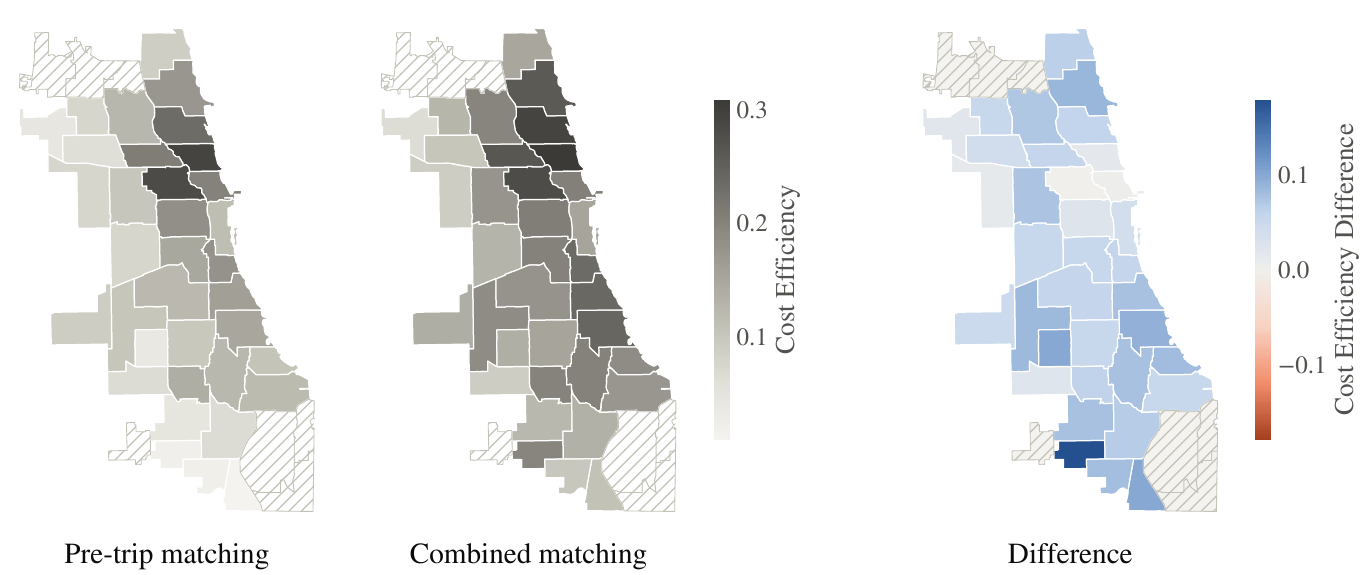}
        \caption{Cost efficiency}
        \label{fig:Chicago_ce}
    \end{subfigure}
    \caption{Geographical distribution of performance metrics in Chicago (waiting window of 1 minute, $c=0.7$).
    Hatched zones are those with fewer than five rider arrivals originating from them in the test data.}
    \label{fig:maps}
\end{figure}

Under pre-trip matching, the outskirts have higher quoted prices (\Cref{fig:Chicago_price}), lower match rates (\Cref{fig:Chicago_mr}), and lower cost efficiency (\Cref{fig:Chicago_ce}) because of the lower demand density: in fact, match rates are low and cost efficiency is close to zero, showing that the platform effectively offers solo ride service in these zones. Moreover, the high prices are a deterrent for shared rides.

Accordingly, the outskirts are where the combined matching policy makes the greatest impact: quoted prices are lower (\Cref{fig:Chicago_price}), driven by higher match rates (\Cref{fig:Chicago_mr}) and cost efficiency (\Cref{fig:Chicago_ce}).
The improvement in match rate and cost efficiency comes from the new matching opportunities created by on-trip matching. Whereas outskirts riders have difficulty getting matched pre-trip because of the sparsity of the demand, they can pick up later arriving riders from  denser zones along the route (we detail the matching tendencies in sparse and dense regions in Appendix \ref{ss:density}). Because these riders are also associated with longer trips (\Cref{fig:length}), the matching opportunity can be substantial. More effective matching then drives a decrease in quoted prices.%

We note that although quoted prices decrease overall, downtown zones can see slightly higher quoted prices (\Cref{fig:Chicago_price}).
This effect can arise from competition between downtown-downtown matches %
and outskirts-downtown matches
under intermediate waiting windows. In such settings, the waiting window is long enough for downtown riders, but too short for outskirts riders, to match effectively pre-trip. Under pre-trip matching, the platform can rely on downtown-downtown matches, and therefore prefers to concentrate more downtown riders by lowering the downtown price. In contrast, combined matching enables more outskirts-downtown matches: an outskirts rider who initially is dispatched solo can likely pick up a downtown rider along the way, as the downtown is dense and the outskirts riders have long trips (\Cref{fig:arrival_rate_and_length}). Hence, competition is created between downtown-downtown and outskirts-downtown matches, and the combined policy pushes downtown prices slightly higher, suppressing pre-trip matches and shifting downtown riders toward on-trip matches with outskirts riders. For more detail, see Appendix \ref{ss:match_types}.

Lastly, \Cref{tab:test_dynamic} replicates \Cref{tab:test_matching} under the dynamic pricing policy described in \Cref{sec:dynamic_pricing}, again comparing combined and pre-trip matching on the test data. (\Cref{tab:train_dynamic} in Appendix \ref{sec:training_result} presents the corresponding training results.)
Under dynamic pricing, combined matching continues to yield higher profits and lower quoted prices than pre-trip matching, and generally improves cost efficiency and reduces payments. Comparing \Cref{tab:test_matching,tab:test_dynamic} further shows that dynamic pricing generally yields higher profits than static pricing under both matching regimes.

\begin{table}[htbp]
  \centering
  \scalebox{0.6}
  {
\begin{tabular}{clrrrrrrrrrrrrrrr}
    \toprule
    \multicolumn{2}{c}{\textbf{Cost} (\$/mile)} & \multicolumn{5}{c}{$c=0.7$} & \multicolumn{5}{c}{$c=0.9$} & \multicolumn{5}{c}{$c=1.1$} \\
    \cmidrule(lr){3-7} \cmidrule(lr){8-12} \cmidrule(lr){13-17}
    \multicolumn{2}{c}{\textbf{Waiting Window} (min)} & \multicolumn{1}{c}{$0$} & \multicolumn{1}{c}{$0.5$} & \multicolumn{1}{c}{$1$} & \multicolumn{1}{c}{$1.5$} & \multicolumn{1}{c}{$2$} & \multicolumn{1}{c}{$0$} & \multicolumn{1}{c}{$0.5$} & \multicolumn{1}{c}{$1$} & \multicolumn{1}{c}{$1.5$} & \multicolumn{1}{c}{$2$} & \multicolumn{1}{c}{$0$} & \multicolumn{1}{c}{$0.5$} & \multicolumn{1}{c}{$1$} & \multicolumn{1}{c}{$1.5$} & \multicolumn{1}{c}{$2$} \\
    \cmidrule(r){1-17}
    \multirow{3}[0]{*}{\shortstack{\textbf{Profit}\\(\$/min)}}
    & Combined & 4.923 & 5.555 & 6.051 & 6.429 & 6.752 & 1.601 & 2.016 & 2.375 & 2.703 & 2.977 & 0.048 & 0.152 & 0.391 & 0.558 & 0.726 \\
    & Pre-trip & 2.734 & 4.365 & 5.185 & 5.755 & 6.221 & 0.302 & 1.085 & 1.670 & 2.115 & 2.464 & 0.000 & 0.000 & 0.000 & 0.000 & 0.050 \\
    & & (-44\%) & (-21\%) & (-14\%) & (-10\%) & (-8\%) & (-81\%) & (-46\%) & (-30\%) & (-22\%) & (-17\%) & (-100\%) & (-100\%) & (-100\%) & (-100\%) & (-93\%) \\
    \cmidrule(r){1-17}
    \multirow{3}[0]{*}{\shortstack{\textbf{Quoted Price}\\(\$/mile)}}
    & Combined & 0.802 & 0.783 & 0.769 & 0.759 & 0.751 & 0.898 & 0.879 & 0.863 & 0.850 & 0.841 & 0.996 & 0.987 & 0.970 & 0.955 & 0.941 \\
    & Pre-trip & 0.850 & 0.799 & 0.778 & 0.766 & 0.756 & 0.950 & 0.906 & 0.878 & 0.862 & 0.850 & 1.000 & 1.000 & 1.000 & 1.000 & 0.989 \\
    &  & (+6\%) & (+2\%) & (+1\%) & (+1\%) & (+1\%) & (+6\%) & (+3\%) & (+2\%) & (+1\%) & (+1\%) & (0\%) & (+1\%) & (+3\%) & (+5\%) & (+5\%) \\
    \cmidrule(r){1-17}
    \multirow{3}[0]{*}{\shortstack{\textbf{Payment}\\(\$/mile)}}
    & Combined & 0.788 & 0.765 & 0.748 & 0.737 & 0.727 & 0.856 & 0.832 & 0.811 & 0.796 & 0.785 & 0.861 & 0.847 & 0.840 & 0.825 & 0.810 \\
    & Pre-trip & 0.850 & 0.779 & 0.754 & 0.742 & 0.731 & 0.950 & 0.859 & 0.819 & 0.801 & 0.789 & - & - & - & - & 0.772 \\
    & & (+8\%) & (+2\%) & (+1\%) & (+1\%) & (0\%) & (+11\%) & (+3\%) & (+1\%) & (+1\%) & (+1\%) & - & - & - & - & (-5\%) \\
    \cmidrule(r){1-17}
    \multirow{3}[0]{*}{\shortstack{\textbf{Throughput}\\(requests/min)}}
    & Combined & 5.795 & 6.372 & 6.775 & 7.035 & 7.284 & 2.984 & 3.536 & 4.013 & 4.359 & 4.651 & 0.114 & 0.383 & 0.886 & 1.322 & 1.715 \\
    & Pre-trip & 4.387 & 5.890 & 6.490 & 6.826 & 7.129 & 1.459 & 2.734 & 3.577 & 4.035 & 4.359 & 0.000 & 0.000 & 0.000 & 0.000 & 0.333 \\
    & & (-24\%) & (-8\%) & (-4\%) & (-3\%) & (-2\%) & (-51\%) & (-23\%) & (-11\%) & (-7\%) & (-6\%) & (-100\%) & (-100\%) & (-100\%) & (-100\%) & (-81\%) \\
    \cmidrule(r){1-17}
    \multirow{2}[0]{*}{\shortstack{\textbf{Match Rate}}}
    & Combined & 0.609 & 0.667 & 0.709 & 0.736 & 0.755 & 0.686 & 0.721 & 0.758 & 0.786 & 0.804 & 0.950 & 0.962 & 0.932 & 0.916 & 0.919 \\
    & Pre-trip & 0.000 & 0.482 & 0.607 & 0.665 & 0.704 & 0.000 & 0.511 & 0.648 & 0.713 & 0.754 & - & - & - & - & 0.875 \\
    \cmidrule(r){1-17}
    \multirow{1}[0]{*}{\shortstack{\textbf{On-Trip  Portion}}}
    & Combined & 1.000 & 0.625 & 0.437 & 0.339 & 0.275 & 1.000 & 0.653 & 0.443 & 0.325 & 0.263 & 1.000 & 0.765 & 0.527 & 0.357 & 0.247 \\
    \cmidrule(r){1-17}
    \multirow{2}[0]{*}{\shortstack{\textbf{Cost Efficiency}}}
    & Combined & 0.164 & 0.203 & 0.232 & 0.254 & 0.272 & 0.184 & 0.218 & 0.246 & 0.269 & 0.286 & 0.270 & 0.297 & 0.309 & 0.323 & 0.338 \\
    & Pre-trip & 0.000 & 0.137 & 0.191 & 0.224 & 0.249 & 0.000 & 0.146 & 0.204 & 0.239 & 0.263 & - & - & - & - & 0.380 \\
    \cmidrule(r){1-17}
    \multirow{2}[0]{*}{\shortstack{\textbf{Detour Rate}\\(matched requests)}}
          & Combined & 0.097 & 0.087 & 0.080 & 0.072 & 0.064 & 0.092 & 0.083 & 0.075 & 0.069 & 0.063 & 0.055 & 0.055 & 0.060 & 0.058 & 0.054 \\
          & Pre-trip & - & 0.090 & 0.079 & 0.071 & 0.063 & - & 0.086 & 0.074 & 0.067 & 0.063 & - & - & - & - & 0.016 \\
    \bottomrule
    \end{tabular}
  }
  \caption{
  Performance of combined matching against pre-trip matching on Chicago data under dynamic pricing (test set). %
  }
  \label{tab:test_dynamic}
\end{table}

\section{Conclusion and Discussion}
\label{sec:conclusion}

Pre-trip matching decides matches before riders are dispatched; hence, a rider who is dispatched solo remains solo for the duration of her trip. By contrast, on-trip matching allows the platform to make matches even after a rider is initially dispatched solo. In the on-trip setting, a key feature is that the rider must be tracked over the course of her entire trip, as her matching opportunity and value change with her position.
In this work, we develop joint pre-trip and on-trip matching policies and investigate the conditions that favor each type of matching. %
We formulate matching and pricing as a Markov decision process and develop a fluid relaxation for static pricing. The optimal fluid objective value bounds optimal static-pricing profit from above, and the dual variables guide a matching heuristic that weighs immediate matches against future opportunities. For dynamic pricing, these dual variables provide an additive value function approximation that yields a heuristic adapting matching decisions and prices to the current system state.

Through computational results on real-world data, we find that pre-trip matching is favorable in dense downtown areas; meanwhile, %
on-trip matching is favorable in sparse outskirts where riders are spatially dispersed.
Adding on-trip matching generally increases profit and efficiency, allowing the platform to lower prices for riders, especially in the sparse outskirts. Under static pricing, the proposed policy also improves profit relative to the tested myopic and rolling horizon batching benchmarks. The benefits of combining both matching phases persist under dynamic pricing, which generally increases profit further. Thus, successful operation of shared rides in large metropolitan areas, where a substantial portion of the demand can be spatially dispersed, can benefit from both pre-trip and on-trip matching, emphasizing the importance of modeling this matching paradigm.

\medskip

\bibliographystyle{informs2014trsc} %
\bibliography{manuscript_references} %

\newpage

\begin{APPENDICES}
\renewcommand{\theHsection}{appendix.\Alph{section}}
\renewcommand{\theHequation}{appendix.\arabic{equation}}
\section{A Stylized Example}
\label{sec:stylized_example}

In this section, we consider a stylized example to examine the value of modeling trip-chaining and considering dynamic pricing policies. We assume that the network consists of $N\geq3$ nodes arranged on a directed cycle, as illustrated in Figure \ref{fig:stylized_cycle}. Each directed link has unit length, and a type-$i$ rider travels from node $i$ to node $i+2$, with node indices interpreted modulo $N$. Thus, $\ell_i=2$ for every rider type.

We assume that all rider types have the same arrival probability $\Lambda\in(0,1/N)$ and that riders' WTP is uniformly distributed on $[0,1]$. The price that induces conversion probability $\lambda$ is therefore $p(\lambda)=1-\lambda$. For tractability, we set $T=0$, so only on-trip matching is possible, and assume $c\in(0,1)$. Each solo trip therefore costs $2c$.

\begin{figure}[htb]
\centering
\begin{tikzpicture}
\Vertex[size=1, x=2, y=0, color=gray!50, label={3}, fontsize=\normalsize]{v2}
\Vertex[size=1, x=1, y=0.866*2, color=gray!50, label={2}, fontsize=\normalsize]{v1}
\Vertex[size=1, x=-1, y=0.866*2, color=gray!50, label={1}, fontsize=\normalsize]{v0}
\Vertex[size=1, x=-2, y=0, color=gray!50, label={6}, fontsize=\normalsize]{v5}
\Vertex[size=1, x=-1, y=-0.866*2, color=gray!50, label={5}, fontsize=\normalsize]{v4}
\Vertex[size=1, x=1, y=-0.866*2, color=gray!50, label={4}, fontsize=\normalsize]{v3}
\Edge[Direct, color=black!50](v0)(v1)
\Edge[Direct, color=black!50](v1)(v2)
\Edge[Direct, color=black!50](v2)(v3)
\Edge[Direct, color=black!50](v3)(v4)
\Edge[Direct, color=black!50](v4)(v5)
\Edge[Direct, color=black!50](v5)(v0)
\Edge[Direct, color=black, lw=3pt](v0)(v1)
\Edge[Direct, color=black, lw=3pt](v1)(v2)
\end{tikzpicture}
\caption{The stylized cycle with $N=6$. Each directed link has unit length. The thick links show a type-1 rider's route from node 1 to node 3. At node 2, this rider can match with a newly converted type-2 rider traveling to node 4.}
\label{fig:stylized_cycle}
\end{figure}

Under the compatibility conditions in \Cref{ss:model_setting}, a newly converted type-$i$ rider can match only with an eligible type-$(i-1)$ rider at node $i$. One period later, the type-$i$ rider reaches node $i+1$, where her only compatible new rider is of type $i+1$, with destination at node $i+3$. If matched, the two riders then travel together to node $i+2$, where the type-$i$ rider exits, and the new rider continues to node $i+3$. Each match saves one unit of vehicle travel, or $c$ in operational cost.

\subsection{Comparison of Profits with and without Trip-Chaining}
\label{ss:stylized_static}

We first compare the optimal profits under static pricing with and without trip-chaining. Without trip-chaining, each rider can be matched at most once, as in the main model. With trip-chaining, a rider can be matched again after the previous co-rider reaches her destination. For ease of exposition, consider a common static price $p(\lambda)=1-\lambda$ for all rider types and a policy that matches each newly converted rider whenever a compatible eligible rider is available, although the result in Proposition \ref{prop:stylized_trip_chaining} does not require these restrictions.

At the beginning of period $t$, only the rider who converted in period $t-1$, if any, is still traveling, with one unit of trip length remaining. Let state 1 indicate that this rider is eligible for matching and state 0 indicate that no eligible rider is available. Without trip-chaining, the rider is eligible only if she has not previously been matched. With trip-chaining, she is eligible regardless of whether she was matched upon conversion, because any previous co-rider has already reached her destination.

In both cases, a conversion in state 0 creates an eligible rider for the next period, so the transition probability to state 1 is $N\Lambda\lambda$. The cases differ when a compatible rider converts in state 1. Without trip-chaining, the new rider becomes ineligible after matching, and the next state is 0. With trip-chaining, the new rider becomes eligible again after the existing rider completes her trip, and the next state is 1. An incompatible conversion leads to state 1 in both cases, whereas no conversion leads to state 0. Thus, the probability of remaining in state 1 is $(N-1)\Lambda\lambda$ without trip-chaining and $N\Lambda\lambda$ with trip-chaining. Figure \ref{fig:stylized_markov_static} summarizes these transitions.

\begin{figure}[htbp]
\centering
\begin{subfigure}[b]{0.48\textwidth}
\centering
\begin{tikzpicture}[
    font=\small, >=stealth', ->, shorten >=2pt,
    line width=0.7pt,
    mcstate/.style={circle,draw=black,fill=gray!10,text=black,
                    minimum size=8mm,inner sep=0pt}
]
\node[mcstate] (ntc0) at (0,0) {$0$};
\node[mcstate] (ntc1) at (2.6,0) {$1$};
\path[draw=black]
    (ntc0) edge[bend left=22] node[above] {$N\Lambda\lambda$} (ntc1)
    (ntc1) edge[bend left=22] node[below] {$1-(N-1)\Lambda\lambda$} (ntc0)
    (ntc0) edge[loop left] node[left] {$1-N\Lambda\lambda$} (ntc0)
    (ntc1) edge[loop right] node[right] {$(N-1)\Lambda\lambda$} (ntc1);
\end{tikzpicture}
\caption{Without trip-chaining}
\label{fig:stylized_markov_ntc}
\end{subfigure}
\hfill
\begin{subfigure}[b]{0.48\textwidth}
\centering
\begin{tikzpicture}[
    font=\small, >=stealth', ->, shorten >=2pt,
    line width=0.7pt,
    mcstate/.style={circle,draw=black,fill=gray!10,text=black,
                    minimum size=8mm,inner sep=0pt}
]
\node[mcstate] (tc0) at (0,0) {$0$};
\node[mcstate] (tc1) at (2.6,0) {$1$};
\path[draw=black]
    (tc0) edge[bend left=22] node[above] {$N\Lambda\lambda$} (tc1)
    (tc1) edge[bend left=22] node[below] {$1-N\Lambda\lambda$} (tc0)
    (tc0) edge[loop left] node[left] {$1-N\Lambda\lambda$} (tc0)
    (tc1) edge[loop right] node[right] {$N\Lambda\lambda$} (tc1);
\end{tikzpicture}
\caption{With trip-chaining}
\label{fig:stylized_markov_tc}
\end{subfigure}
\caption{Markov chains under a common static price without and with trip-chaining. States 0 and 1 indicate the absence and presence of an eligible rider, respectively.}
\label{fig:stylized_markov_static}
\end{figure}

The stationary probability of state 1 is $N\Lambda\lambda/(1+\Lambda\lambda)$ without trip-chaining and $N\Lambda\lambda$ with trip-chaining. Since a match occurs when a compatible rider converts in state 1, multiplying these probabilities by $\Lambda\lambda$ gives the respective match rates $N\Lambda^2\lambda^2/(1+\Lambda\lambda)$ and $N\Lambda^2\lambda^2$. These match rates allow us to calculate the long-run average profits under the common static price and matching policy described above.

To compare the two settings more generally, let $\Pi_{\mathrm{NTC}}$ and $\Pi_{\mathrm{TC}}$ denote the optimal long-run average profits over all policies under static pricing without and with trip-chaining, respectively, allowing prices to differ across rider types. We then establish the following bound.
\begin{proposition}\label{prop:stylized_trip_chaining}
\[
\underbrace{\frac{2(\sqrt{42}-4)}{5}}_{\approx0.992}\,\Pi_{\mathrm{TC}}
\leq \Pi_{\mathrm{NTC}}\leq\Pi_{\mathrm{TC}}.
\]
\end{proposition}

Proposition \ref{prop:stylized_trip_chaining} shows that trip-chaining provides little additional profit in this example, supporting our decision to omit it from the main model. The proof is given in Appendix \ref{ss:proof_stylized_chaining}.

\subsection{Comparison of Profits under Static and Dynamic Pricing}
\label{ss:stylized_dynamic}

We next compare static and dynamic pricing in the model without trip-chaining. Consider a dynamic pricing policy that quotes price $1-\lambda_1$ if a compatible solo rider is available when a potential rider arrives, and price $1-\lambda_0$ otherwise. Upon conversion, it forms a match whenever one is available.

Let state 1 indicate that an eligible solo rider has one period of travel remaining and state 0 indicate that none is available. In state 0, each arriving rider converts with probability $\lambda_0$, creating an eligible solo rider for the next period. The transition probability to state 1 is therefore $N\Lambda\lambda_0$. In state 1, the $N-1$ incompatible types receive price $1-\lambda_0$, and a conversion of one of these types creates a new eligible solo rider. The chain thus remains in state 1 with probability $(N-1)\Lambda\lambda_0$. A compatible conversion produces a match and leads to state 0 because the new rider cannot be matched again. If no rider converts, the existing rider completes the trip, also leading to state 0. Figure \ref{fig:stylized_markov_dynamic} shows the resulting Markov chain.

\begin{figure}[htbp]
\centering
\begin{tikzpicture}[
    font=\small, >=stealth', ->, shorten >=2pt,
    line width=0.7pt,
    mcstate/.style={circle,draw=black,fill=gray!10,text=black,
                    minimum size=8mm,inner sep=0pt}
]
\node[mcstate] (dyn0) at (0,0) {$0$};
\node[mcstate] (dyn1) at (3,0) {$1$};
\path[draw=black]
    (dyn0) edge[bend left=22] node[above] {$N\Lambda\lambda_0$} (dyn1)
    (dyn1) edge[bend left=22] node[below] {$1-(N-1)\Lambda\lambda_0$} (dyn0)
    (dyn0) edge[loop left] node[left] {$1-N\Lambda\lambda_0$} (dyn0)
    (dyn1) edge[loop right] node[right] {$(N-1)\Lambda\lambda_0$} (dyn1);
\end{tikzpicture}
\caption{Markov chain under dynamic pricing. States 0 and 1 indicate the absence and presence of an eligible solo rider, respectively.}
\label{fig:stylized_markov_dynamic}
\end{figure}

These transitions give a stationary probability of $N\Lambda\lambda_0/(1+\Lambda\lambda_0)$ for state 1. A compatible rider converts in this state with probability $\Lambda\lambda_1$, so the match rate is $N\Lambda^2\lambda_0\lambda_1/(1+\Lambda\lambda_0)$. Following the notation in \Cref{ss:model_mdp}, let $\Pi_{\mathrm{stat}}$ and $\Pi_{\mathrm{dyn}}$ denote the optimal long-run average profits under static and dynamic pricing, respectively.
\begin{proposition}\label{prop:stylized_dynamic_pricing}
Let
\[
\bar c:=1-\frac{1}{1+\sqrt{1+\Lambda/4}}>\frac12.
\]
For $1/2\leq c<\bar c$, static pricing yields zero optimal profit, whereas dynamic pricing yields positive optimal profit:
\[
\Pi_{\mathrm{dyn}}>\Pi_{\mathrm{stat}}=0.
\]
\end{proposition}

Proposition \ref{prop:stylized_dynamic_pricing} shows that adjusting prices to the availability of a compatible solo rider extends the range of costs for which shared rides are profitable. The platform can encourage conversions when an immediate match is available while charging a higher price when a rider would initially travel solo. This result supports extending our policies from static to dynamic pricing in \Cref{sec:dynamic_pricing}. The proof is given in Appendix \ref{ss:proof_stylized_pricing}.

\section{Proofs}

\subsection{Proof of Proposition \ref{prop:fluid_bound}}
Fix any stationary matching policy $\boldsymbol{\phi}^*$ and static pricing policy $\boldsymbol{\lambda}^*$, and denote their long-run average profit by $\Pi^{\boldsymbol{\phi}^*,\boldsymbol{\lambda}^*}$. Assume that the system starts empty. Consistent with \eqref{eq:state_transition}, $\boldsymbol{\mathrm{s}}_t$ is the system state immediately before the arrival, conversion, and matching events in period $t$. For all $(j,u)\in\mathcal{C}$ and $i\in\mathcal{N}_{j,u}^+$, define the counting process $\{M_t^*(i,(j,u))\}_{t\geq0}$ as the cumulative number of matches through the end of period $t$ between new type-$i$ riders and solo riders in state $(j,u)$ under the policies $\boldsymbol{\phi}^*$ and $\boldsymbol{\lambda}^*$. Define $\{A_t^*(i,u)\}_{t\geq0}$ as the cumulative number of events in which solo riders $(i,u)\in\mathcal{C}$ continue solo and advance to the next temporal position or complete their trips. For notational convenience, for each $i\in[N]$, define $\{A_t^*(i)\}_{t\geq0}$ as the cumulative number of type-$i$ requests and $\{A_t^*(i,-T)\}_{t\geq0}$ as the cumulative number of events in which a new type-$i$ rider remains unmatched upon conversion and becomes a solo rider. All counting processes include events through the end of period $t$ and are initialized to zero at $t=-1$.

Then we define the steady flow rates under $\boldsymbol{\phi}^*$ and $\boldsymbol{\lambda}^*$ as
\begin{align*}
    &\hat{x}_{i,j}^{u} := \lim_{t\to\infty} 1/t\cdot \mathbb{E}[M_t^*(i,(j,u))], & \forall (j,u)\in\mathcal{C}, i\in\mathcal{N}_{j,u}^+,\\
    &\hat{y}_{i}^{u} := \lim_{t\to\infty} 1/t\cdot \mathbb{E}[A_t^*(i,u)], & \forall (i,u)\in\mathcal{C},\\
    &\hat{y}_{i}^{-T} := \lim_{t\to\infty} 1/t\cdot \mathbb{E}[A_t^*(i,-T)], & \forall i \in [N].
\end{align*}

These limits exist because the induced Markov chain is finite under the stationary policies. We first prove that the flow variables $\hat{x}_{i,j}^{u}$, $\hat{y}_{i}^{u}$, and $\hat{y}_{i}^{-T}$ form a feasible solution to $\textrm{\textsf{CB}}(\boldsymbol{\lambda}^*)$. According to the definitions, the following equations hold:
\begin{align}
  &\sum_{\substack{
    (j,u)\in\mathcal{C}:\:\\ i\in\mathcal{N}^+_{j,u}
    }} M_t^*(i,(j,u)) + A_t^*(i,-T) = A_t^*(i),& \forall i\in[N],t\in\mathbbm{N}, \label{eq:counting_demand}\\
    & \sum_{i\in\mathcal{N}_{j,u}^+}  M_t^*(i,(j,u)) + A_t^*(j,u) = A_{t-1}^*(j,u-1), & \forall (j,u)\in\mathcal{C}, t\in\mathbbm{N}^+,\label{eq:counting_flow_balance}
\end{align}
Equation \eqref{eq:counting_demand} partitions converted requests into those immediately matched and those that become solo riders. Equation \eqref{eq:counting_flow_balance} accounts for riders who entered state $(j,u)$ by period $t$ and either matched or continued solo by the end of that period. Taking
expectations in \eqref{eq:counting_demand} and \eqref{eq:counting_flow_balance}, dividing by $t$, and letting $t\to\infty$, we get
\begin{align}
    &\quad \sum_{\substack{
    (j,u)\in\mathcal{C}:\:\\ i\in\mathcal{N}^+_{j,u}
    }}  \hat{x}_{i,j}^u + \hat{y}_i^{-T} = \lim_{t\to\infty} \frac{1}{t} \mathbb{E}[A_t^*(i)] = \Lambda_i\lambda_i^*, & \forall i\in[N], \label{eq:fluid_solution_demand}\\
    & \quad \sum_{i\in\mathcal{N}_{j,u}^+}  \hat{x}_{i,j}^u + \hat{y}_j^u = \hat{y}_j^{u-1}, & \forall (j,u)\in\mathcal{C},\label{eq:fluid_solution_flow_balance}
\end{align}
thus \eqref{eq:fluid_demand} and \eqref{eq:fluid_flow_balance} are satisfied.

To show that \eqref{eq:fluid_ratio} holds, for each $(j,u)\in\mathcal{C}$, $i\in\mathcal{N}_{j,u}^+$, and $t\in\mathbb{N}$, the expected number of matches during period $t+1$ is:
\begin{align}
    & \mathbb{E}[M_{t+1}^*(i,(j,u)) - M_t^*(i,(j,u))] \notag\\
    =& \sum_{\boldsymbol{\mathrm{s}}\in\mathcal{S}} \textrm{Pr}(\boldsymbol{\mathrm{s}}_{t+1} = \boldsymbol{\mathrm{s}}) \cdot \mathbb{E} [M_{t+1}^*(i,(j,u)) - M_t^*(i,(j,u))|\boldsymbol{\mathrm{s}}_{t+1} = \boldsymbol{\mathrm{s}} ] \notag\\
    \stackrel{(a)}{=}& \sum_{\boldsymbol{\mathrm{s}}\in\mathcal{S}} \textrm{Pr}(\boldsymbol{\mathrm{s}}_{t+1} = \boldsymbol{\mathrm{s}}) \cdot \textrm{Pr} [M_{t+1}^*(i,(j,u)) - M_t^*(i,(j,u))=1|\boldsymbol{\mathrm{s}}_{t+1} = \boldsymbol{\mathrm{s}} ] \notag\\
    \stackrel{(b)}{\leq}& \sum_{\boldsymbol{\mathrm{s}}\in\mathcal{S}:\mathrm{s}_{j,u}=1} \textrm{Pr}(\boldsymbol{\mathrm{s}}_{t+1} = \boldsymbol{\mathrm{s}}) \cdot \textrm{Pr}  [A_{t+1}^*(i) - A_t^*(i)=1|\boldsymbol{\mathrm{s}}_{t+1} = \boldsymbol{\mathrm{s}} ] \notag\\
    =& \sum_{\boldsymbol{\mathrm{s}}\in\mathcal{S}:\mathrm{s}_{j,u}=1} \textrm{Pr}(\boldsymbol{\mathrm{s}}_{t+1} = \boldsymbol{\mathrm{s}}) \cdot \Lambda_i\lambda_i^*, \label{eq:match_ub}
\end{align}
Step (a) follows because $M_{t+1}^*(i,(j,u))-M_t^*(i,(j,u))\in\{0,1\}$. Step (b) follows because a match during period $t+1$ requires both an existing solo rider in state $(j,u)$ and a newly converted type-$i$ rider.

Meanwhile, for every $(j,u)\in\mathcal{C}$ and $t\in\mathbbm{N}$, the expected number of solo-continuation events during period $t+1$ is:
\begin{align}
    & \mathbb{E}[A_{t+1}^*(j,u) - A_t^*(j,u)] \notag\\
    =& \sum_{\boldsymbol{\mathrm{s}}\in\mathcal{S}} \textrm{Pr}(\boldsymbol{\mathrm{s}}_{t+1} = \boldsymbol{\mathrm{s}}) \cdot \mathbb{E} [A_{t+1}^*(j,u) - A_t^*(j,u)|\boldsymbol{\mathrm{s}}_{t+1} = \boldsymbol{\mathrm{s}} ] \notag\\
    =& \sum_{\boldsymbol{\mathrm{s}}\in\mathcal{S}} \textrm{Pr}(\boldsymbol{\mathrm{s}}_{t+1} = \boldsymbol{\mathrm{s}}) \cdot \textrm{Pr} [A_{t+1}^*(j,u) - A_t^*(j,u)=1|\boldsymbol{\mathrm{s}}_{t+1} = \boldsymbol{\mathrm{s}} ] \notag\\
    =& \sum_{\boldsymbol{\mathrm{s}}\in\mathcal{S}:\mathrm{s}_{j,u}=1} \textrm{Pr}(\boldsymbol{\mathrm{s}}_{t+1} = \boldsymbol{\mathrm{s}}) \cdot \textrm{Pr} \left[\sum_{k\in\mathcal{N}_{j,u}^+}M_{t+1}^*(k,(j,u)) - \sum_{k\in\mathcal{N}_{j,u}^+}M_{t}^*(k,(j,u)) =0\middle|\boldsymbol{\mathrm{s}}_{t+1} = \boldsymbol{\mathrm{s}} \right] \notag\\
    \stackrel{(a)}{\geq}&\sum_{\boldsymbol{\mathrm{s}}\in\mathcal{S}:\mathrm{s}_{j,u}=1} \textrm{Pr}(\boldsymbol{\mathrm{s}}_{t+1} = \boldsymbol{\mathrm{s}}) \cdot \mathbb{E} \left[1-\sum_{k\in\mathcal{N}_{j,u}^+}\Lambda_k\lambda_k^* \middle|\boldsymbol{\mathrm{s}}_{t+1} = \boldsymbol{\mathrm{s}} \right] \notag\\
    =& \sum_{\boldsymbol{\mathrm{s}}\in\mathcal{S}:\mathrm{s}_{j,u}=1} \textrm{Pr}(\boldsymbol{\mathrm{s}}_{t+1} = \boldsymbol{\mathrm{s}}) \cdot \left( 1-\sum_{k\in\mathcal{N}_{j,u}^+}\Lambda_k\lambda_k^*\right), \label{eq:unmatch_ub}
\end{align}
Step (a) follows because, when no compatible rider of any type $k\in\mathcal{N}_{j,u}^+$ converts during period $t+1$, the existing solo rider $(j,u)$ cannot be matched and must continue solo.

Combining \eqref{eq:match_ub} and \eqref{eq:unmatch_ub}, we get
\begin{align*}
&\Lambda_i\lambda_i^*
\mathbb{E}[A_{t+1}^*(j,u)-A_t^*(j,u)]\geq
\left(1-\sum_{k\in\mathcal{N}_{j,u}^+}\Lambda_k\lambda_k^*\right)
\mathbb{E}[M_{t+1}^*(i,(j,u))-M_t^*(i,(j,u))],
\end{align*}
which holds for all $(j,u)\in\mathcal{C}, i\in\mathcal{N}_{j,u}^+$ and $t\in\mathbbm{N}$. This inequality remains valid when $\Lambda_i\lambda_i^*=0$. Summing over $t=0,\ldots,n-1$, dividing by $n$, and letting $n\to\infty$ yields \eqref{eq:fluid_ratio}. Nonnegativity follows directly from the definitions of the flow variables. Therefore, $\hat{x}_{i,j}^{u}$, $\hat{y}_{i}^{u}$, and $\hat{y}_{i}^{-T}$ form a feasible solution to $\textrm{\textsf{CB}}(\boldsymbol{\lambda}^*)$.

\smallskip

In the second part of the proof, we show that the fluid objective value at conversion probabilities $\boldsymbol{\lambda}^*$ and fluid variables $\hat{x}_{i,j}^{u}$, $\hat{y}_{i}^{u}$ and $\hat{y}_{i}^{-T}$ equals
the long-run average profit of these policies under static pricing. For this comparison, let $\tilde{\Pi}_t^{\boldsymbol{\phi}^*,\boldsymbol{\lambda}^*}$ denote the total profit from solo trips completed by the beginning of period $t$ and all pairs matched during periods $0,\ldots,t-1$. Using the counting processes defined above, we have
\begin{align*}
    \tilde{\Pi}_{t}^{\boldsymbol{\phi}^*,\boldsymbol{\lambda}^*} =& \sum_{i\in[N]} A_{t-1}^*(i,\ell_i-1) \left[ p_i(\lambda_i^*) - c\ell_i \right] + \sum_{i\in[N]} \sum_{\substack{(j,u)\in\mathcal{C}:\:\\i\in \mathcal{N}_{j,u}^+}} M_{t-1}^*(i, (j,u))\left[ p_i(\lambda_i^*) + p_j(\lambda_j^*) -c \max\{u,0\} -c \ell_{i,j}^{u} \right].
\end{align*}
The difference between $\Pi_t^{\boldsymbol{\phi}^*,\boldsymbol{\lambda}^*}$ and $\tilde{\Pi}_t^{\boldsymbol{\phi}^*,\boldsymbol{\lambda}^*}$ consists only of the collected fares and accrued travel costs of solo riders still in the system. There are at most $|\mathcal{C}|$ such riders. For fixed static prices, this difference is bounded independently of $t$, so the two quantities have the same long-run average.

Taking the long-run time average and noting that shifting a counting-process index by one period does not change its limiting average rate, we get
$$
\Pi^{\boldsymbol{\phi}^*,\boldsymbol{\lambda}^*}= \lim_{t\to\infty} \frac{1}{t} \mathbbm{E} [\tilde{\Pi}_{t}^{\boldsymbol{\phi}^*,\boldsymbol{\lambda}^*}] = \quad \sum_{i\in[N]} \hat{y}_{i}^{\ell_i-1}\left[ p_i(\lambda_i^*) - c\ell_i \right]  + \sum_{i\in[N]} \sum_{\substack{(j,u)\in\mathcal{C}:\:\\i\in \mathcal{N}_{j,u}^+}} \hat{x}_{i,j}^u \left[ p_i(\lambda_i^*) + p_j(\lambda_j^*) -c \max\{u,0\} -c \ell_{i,j}^{u} \right].
$$
Telescoping \eqref{eq:fluid_solution_flow_balance} over the temporal states and accounting for the cost incurred in each travel period gives
\begin{align}
&\sum_{i\in[N]}\hat y_i^{\ell_i-1}
[p_i(\lambda_i^*)-c\ell_i] \notag\\
={}&\sum_{i\in[N]}\hat y_i^{-T}p_i(\lambda_i^*)
-c\sum_{i\in[N]}\sum_{u=0}^{\ell_i-1}\hat y_i^u -\sum_{i\in[N]}
\sum_{\substack{(j,u)\in\mathcal C:\\i\in\mathcal N_{j,u}^+}}
\hat x_{i,j}^u
[p_j(\lambda_j^*)-c\max\{u,0\}].
\label{eq:proof_prop_1}
\end{align}

Hence, by substituting $\sum_{i\in[N]} \hat{y}_{i}^{\ell_i-1}\left[ p_i(\lambda_i^*) - c\ell_i \right]$ using \eqref{eq:proof_prop_1}, $\Pi^{\boldsymbol{\phi}^*,\boldsymbol{\lambda}^*}$ can be expressed as
\begin{align*}
    \Pi^{\boldsymbol{\phi}^*,\boldsymbol{\lambda}^*} =& \sum_{i\in[N]} \hat{y}_{i}^{-T} p_i(\lambda_i^*) - \sum_{i\in[N]} \sum_{\substack{(j,u)\in\mathcal{C}:\:\\i\in \mathcal{N}_{j,u}^+}} \hat{x}_{i,j}^{u}\left[ p_j(\lambda_j^*) - c\max\{u,0\} \right] - c \sum_{i\in[N]} \sum_{u =0}^{\ell_i-1} \hat{y}_{i}^u \\ &+ \sum_{i\in[N]} \sum_{\substack{(j,u)\in\mathcal{C}:\:\\i\in \mathcal{N}_{j,u}^+}} \hat{x}_{i,j}^u \left[ p_i(\lambda_i^*) + p_j(\lambda_j^*) -c \max\{u,0\} -c \ell_{i,j}^{u} \right] \\
    =& \sum_{i\in[N]} \hat{y}_{i}^{-T} p_i(\lambda_i^*) + \sum_{i\in[N]} \sum_{\substack{(j,u)\in\mathcal{C}:\:\\i\in \mathcal{N}_{j,u}^+}} \hat{x}_{i,j}^u p_i(\lambda_i^*) - c \sum_{i\in[N]} \sum_{u =0}^{\ell_i-1} \hat{y}_{i}^u - c \sum_{i\in[N]} \sum_{\substack{(j,u)\in\mathcal{C}:\:\\i\in \mathcal{N}_{j,u}^+}} \hat{x}_{i,j}^u \ell_{i,j}^{u} \\
    \stackrel{(a)}{=} & \sum_{i\in[N]} \Lambda_i \lambda_i^* p_i(\lambda_i^*) - c \sum_{i\in[N]} \sum_{u =0}^{\ell_i-1} \hat{y}_{i}^u - c \sum_{i\in[N]} \sum_{\substack{(j,u)\in\mathcal{C}:\:\\i\in \mathcal{N}_{j,u}^+}} \hat{x}_{i,j}^u \ell_{i,j}^{u},
\end{align*}
where (a) is due to \eqref{eq:fluid_solution_demand}. Therefore, $\Pi^{\boldsymbol{\phi}^*,\boldsymbol{\lambda}^*}$ equals the fluid objective value under the feasible solution $\boldsymbol{\lambda}^*$,  $\hat{x}_{i,j}^{u}$, $\hat{y}_{i}^{u}$ and $\hat{y}_{i}^{-T}$, and thus we have $g \geq \Pi^{\boldsymbol{\phi}^*,\boldsymbol{\lambda}^*}$. Taking the supremum over stationary matching policies and static pricing policies yields $g\geq\Pi_{\mathrm{stat}}$. $\hfill\square$

\subsection{Proof of Proposition \ref{prop:policy_on_trip_wo_ratio}}

Fix $T\geq1$ and a static conversion vector $\boldsymbol{\lambda}$. Because two riders of the same type can share their entire trip when matched before departure, $i\in\mathcal{N}_{i,-T+1}^+$ and $\ell_{i,i}^{-T+1}=\ell_i$. Consider the fluid solution
\[
\tilde{y}_i^{-T}=\frac{1}{2}\Lambda_i\lambda_i,
\qquad
\tilde{x}_{i,i}^{-T+1}=\frac{1}{2}\Lambda_i\lambda_i,
\qquad\forall i\in[N],
\]
with all other flows equal to zero. This solution satisfies \eqref{eq:fluid_demand} and \eqref{eq:fluid_flow_balance}. Once \eqref{eq:fluid_ratio} is removed, it is therefore feasible and has cost $(c/2)\sum_{i\in[N]}\Lambda_i\lambda_i\ell_i$.

To obtain a matching lower bound, let $\widetilde C(\boldsymbol{\lambda})$ denote the minimum fluid cost without the ratio constraints. Its dual is
\begin{subequations}
\begin{align}
\widetilde C(\boldsymbol{\lambda})=\max_{\boldsymbol{\gamma},\boldsymbol{\xi}}
&\quad\sum_{i\in[N]}\Lambda_i\lambda_i\gamma_i \notag\\
\textrm{s.t.}
&\quad\gamma_i-\xi_i^{-T+1}\leq0,
&&\forall i\in[N],\label{eq:fluid_relaxation_dual_entry}\\
&\quad\xi_i^u-\xi_i^{u+1}\leq c\mathbbm{1}\{u\geq0\},
&&\forall i\in[N],\ -T+1\leq u\leq\ell_i-2,\label{eq:fluid_relaxation_dual_progress}\\
&\quad\xi_i^{\ell_i-1}\leq c,
&&\forall i\in[N],\label{eq:fluid_relaxation_dual_terminal}\\
&\quad\gamma_i+\xi_j^u\leq c\ell_{i,j}^u,
&&\forall(j,u)\in\mathcal{C},\ i\in\mathcal{N}_{j,u}^+.
\label{eq:fluid_relaxation_dual_match}
\end{align}
\end{subequations}
Define
\[
\tilde{\gamma}_i=\frac{c\ell_i}{2},
\qquad
\tilde{\xi}_i^u=\frac{c\ell_i}{2}-c\max\{u,0\},
\qquad\forall i\in[N],\ (i,u)\in\mathcal{C}.
\]
These values satisfy \eqref{eq:fluid_relaxation_dual_entry}--\eqref{eq:fluid_relaxation_dual_terminal}. To verify \eqref{eq:fluid_relaxation_dual_match}, observe that the route consisting of rider $j$'s completed solo segment and the remaining shared trip serves both riders. Its length is at least each rider's shortest solo-trip length. Hence
\[
\ell_{i,j}^u+\max\{u,0\}\geq\max\{\ell_i,\ell_j\}
\geq\frac{\ell_i+\ell_j}{2},
\]
and therefore
\[
\tilde{\gamma}_i+\tilde{\xi}_j^u
=\frac{c(\ell_i+\ell_j)}{2}-c\max\{u,0\}
\leq c\ell_{i,j}^u.
\]
This dual feasible solution has the same objective value as the primal feasible solution above. Weak duality establishes
\[
\widetilde C(\boldsymbol{\lambda})=\frac{c}{2}\sum_{i\in[N]}\Lambda_i\lambda_i\ell_i.
\]
Thus, the flow solution constructed above is optimal for any arrival rates and any positive waiting window.

$\hfill\square$

\subsection{Proof of Proposition \ref{prop:stylized_trip_chaining}}
\label{ss:proof_stylized_chaining}

For the profit calculations in this proof, we charge cost $2c$ for each rider upon conversion and credit the saving $c$ for each match. This preserves long-run average profit because it changes the cumulative cost only by a bounded amount for trips still in progress.

With trip-chaining, each match saves $c$ without removing future matching opportunities, since the new rider becomes eligible again after the existing rider completes her trip. Matching decisions also do not affect future conversions under static pricing, so forming every available match is optimal for any fixed prices. We next verify that a common static price is optimal even when prices may differ across rider types. For type-specific conversion probabilities $\lambda_i$, the request rate is $\Lambda\sum_i\lambda_i$ and the match rate is $\Lambda^2\sum_i\lambda_i\lambda_{i+1}$, with indices interpreted modulo $N$. The resulting profit is
\[
\Lambda\left[(1-2c)\sum_i\lambda_i-\sum_i\lambda_i^2
+c\Lambda\sum_i\lambda_i\lambda_{i+1}\right]
\leq\Lambda\sum_i\left[(1-2c)\lambda_i-(1-c\Lambda)\lambda_i^2\right],
\]
where the inequality follows from $\sum_i(\lambda_i-\lambda_{i+1})^2\geq0$. The upper bound is attained when all $\lambda_i$ equal the maximizer of the same scalar quadratic. Therefore,
\begin{equation}
\Pi_{\mathrm{TC}}
=\max_{\lambda\in[0,1]}N\Lambda\left[\lambda(1-2c)
-\left(1-c\Lambda\right)\lambda^2\right].
\label{eq:stylized_static_chain}
\end{equation}
The quadratic objective in \eqref{eq:stylized_static_chain} is strictly concave. Its maximum is zero for $c\geq1/2$. For $0<c<1/2$, the maximizing conversion probability and profit are
\begin{equation}
\lambda_{\mathrm{TC}}^*=\frac{1-2c}{2(1-c\Lambda)},
\qquad
\Pi_{\mathrm{TC}}=\frac{N\Lambda(1-2c)^2}{4(1-c\Lambda)}.
\label{eq:stylized_static_chain_opt}
\end{equation}

To obtain a lower bound on $\Pi_{\mathrm{NTC}}$, consider the feasible policy represented by the Markov chain without trip-chaining in Figure \ref{fig:stylized_markov_static}. Its request rate is $N\Lambda\lambda$, and its match rate is $N\Lambda^2\lambda^2/(1+\Lambda\lambda)$. Let $\Pi_{\mathrm{NTC}}(\lambda)$ denote the profit of this policy. Counting the solo-trip cost for each conversion and the saving for each match gives
\begin{equation}
\Pi_{\mathrm{NTC}}(\lambda)
=N\Lambda\left[\lambda(1-2c)-\lambda^2
+\frac{c\Lambda\lambda^2}{1+\Lambda\lambda}\right].
\label{eq:stylized_static_two}
\end{equation}
The optimal profit $\Pi_{\mathrm{NTC}}$ is at least this policy's profit for any $\lambda\in[0,1]$. Also, $\Pi_{\mathrm{NTC}}\leq\Pi_{\mathrm{TC}}$, since every policy without trip-chaining remains feasible when trip-chaining is allowed. In particular, both optimal static profits are zero for $c\geq1/2$.

For $0<c<1/2$, evaluating \eqref{eq:stylized_static_two} at the conversion probability in \eqref{eq:stylized_static_chain_opt} gives
\begin{align*}
\frac{\Pi_{\mathrm{NTC}}}{\Pi_{\mathrm{TC}}}
\geq
\frac{\Pi_{\mathrm{NTC}}(\lambda_{\mathrm{TC}}^*)}{\Pi_{\mathrm{TC}}}
=1-\frac{c(1-2c)\Lambda^2}{(1-c\Lambda)(2+\Lambda-4c\Lambda)}\geq1-\frac{c(1-2c)}{4c^2-19c+21}.
\end{align*}
For the second inequality, divide the denominator of the fraction in the first line by $\Lambda^2$ to obtain $(1/\Lambda-c)(2/\Lambda+1-4c)$. Both factors are positive and decrease as $\Lambda$ increases. Since $\Lambda<1/N\leq1/3$, replacing $\Lambda$ by $1/3$ gives the stated lower bound. Minimizing this bound over $c\in(0,1/2)$ gives $c=(42-5\sqrt{42})/34$ and the value $2(\sqrt{42}-4)/5>0.992$. Together with the zero profits for $c\geq1/2$, this proves the proposition for the full range $c\in(0,1)$. $\hfill\square$

\subsection{Proof of Proposition \ref{prop:stylized_dynamic_pricing}}
\label{ss:proof_stylized_pricing}

Fix $1/2\leq c<\bar c$. Without trip-chaining, $\Pi_{\mathrm{stat}}=\Pi_{\mathrm{NTC}}$. Appendix \ref{ss:proof_stylized_chaining} establishes that $\Pi_{\mathrm{TC}}=0$ for $c\geq1/2$. Since $0\leq\Pi_{\mathrm{stat}}\leq\Pi_{\mathrm{TC}}$, we have $\Pi_{\mathrm{stat}}=0$.

Consider the policy under dynamic pricing described in Appendix \ref{ss:stylized_dynamic}. The stationary probability of state 1 is $N\Lambda\lambda_0/(1+\Lambda\lambda_0)$. This is also the long-run rate of conversions without an available match, since each such conversion creates an eligible solo rider for the next period. Charging the solo-trip cost upon conversion and crediting the saving for each match, each of these conversions contributes $1-\lambda_0-2c$. In state 1, a compatible rider converts with probability $\Lambda\lambda_1$ and contributes $1-\lambda_1-c$. Thus, the policy's long-run average profit is
\begin{equation}
\Pi_{\mathrm{dyn}}(\lambda_0,\lambda_1)
=\frac{N\Lambda\lambda_0}{1+\Lambda\lambda_0}
\left[1-\lambda_0-2c
+\Lambda\lambda_1\left(1-\lambda_1-c\right)\right].
\label{eq:stylized_dynamic_profit}
\end{equation}
Define
\[
a:=1-2c+\frac{\Lambda(1-c)^2}{4}.
\]
This expression decreases strictly with $c$ and equals zero at $c=\bar c$. Therefore, $0<a\leq\Lambda/16<1$ for $1/2\leq c<\bar c$. Choose
\[
\lambda_0=\frac{a}{2},\qquad \lambda_1=\frac{1-c}{2},
\]
which both lie in $(0,1)$. Substituting into \eqref{eq:stylized_dynamic_profit} gives
\[
\Pi_{\mathrm{dyn}}\geq\Pi_{\mathrm{dyn}}(\lambda_0,\lambda_1)
=\frac{N\Lambda a^2}{4(1+\Lambda a/2)}>0.
\]
Together with $\Pi_{\mathrm{stat}}=0$, this proves the proposition. $\hfill\square$

\bigskip
\section{Details on Benchmarks} \label{ss:benchmarks}

In this section, we describe our batch-matching benchmarks. To isolate the effect of different matching policies, the associated pricing policy is the same static pricing policy solved from \eqref{eq:profit_fluid}.

Both batch-matching benchmarks collect unmatched riders over a batch window $B$ and match them at epochs $t\in\{0, B, 2B, \dots\}$. We set $B$ equal to the waiting window $T$ used in our proposed matching policy so every request will wait at most $T$ periods before being dispatched.
Consequently, the benchmarks are only defined for $T >0 $.
Waiting requests that do not obtain a match in a batch will be dispatched solo immediately.

\textit{Feasible matches.} Let $\mathcal C_t$ be the set of unmatched riders in the system at epoch $t$, including two types: (i) pre-trip riders $(i,u)$, $u\le 0$,
\footnote{In the benchmarks, with a slight abuse of notation, we still use $(i,u)$ with $u\in[-T+1, 0]$ to represent the waiting requests that have not been dispatched and are waiting pre-trip. However, it does not mean these requests have to finish the remaining $|u|$ waiting periods before dispatching; instead, they will be dispatched solo immediately at the end of a batch window if they are not matched.}
who have arrived but have not been dispatched and are still waiting pre-trip, and (ii) on-trip riders $(j,v), v>0$ who have already been dispatched solo but remain unmatched.
For a set of riders $\mathcal P$, let $\mathcal E(\mathcal P)$ be the set of pairs $e=\{(i,u),(j,v)\} \subseteq \mathcal P$ with $u \le 0$ and $i \in \mathcal N^+_{j,v}$, that is, at least one rider has not been dispatched and the two are compatible with positive saving
$g_e := \ell_i + \ell_j - \max\{v,0\} - \ell^v_{i,j}.
$
Here $\ell_i$ and $\ell_j-\max\{v,0\}$ are the remaining solo trip lengths of the two riders, and $\ell^v_{i,j}$ is their shared trip length.

\textit{Myopic batching.} At each epoch, the myopic benchmark solves a maximum-weight matching problem on the current batch:
\begin{equation}
\max_{\boldsymbol{z} \in \{0,1\}^{\mathcal E(\mathcal C_t)}}
\Big\{ \sum_{e \in \mathcal E(\mathcal C_t)} g_e z_e \;:\;
\sum_{e \in \mathcal E(\mathcal C_t):\, r \in e} z_e \le 1
\;\; \forall r \in \mathcal C_t \Big\}.
\end{equation}

\textit{Rolling horizon batching.} The myopic benchmark ignores that a rider left unmatched now may find a better match later while on-trip. Our second benchmark accounts for this through a rolling horizon method solved by sample average approximation.

\begin{itemize}
\item \textit{Lookahead.} At epoch $t$, the benchmark considers a sequence of future requests over a horizon of $H$ periods forward, $(t, t+H]$. Let $\bar q := \lfloor H/B \rfloor$ be the number of epochs covered by the lookahead horizon. In our experiments, we set $H = 2{,}256$ periods (about 8 minutes).

\item \textit{Sampling.} The benchmark draws $S$ independent scenarios of future request sequences based on expected arrival rates of riders. In scenario $s$, $\omega^s$ is a set of future conversions $(\tau, k)$, with conversion period $\tau \in (t, t+H]$ and type $k$. Scenarios follow the demand model of \Cref{sec:model} at the static prices: in each period, at most one request converts, and it is of type $k$ with probability $\Lambda_k \hat\lambda_k$, independent across periods. In our experiments, we use $S = 40$ scenarios.

\item \textit{Projected batches.} For each scenario and future epoch, the benchmark projects every rider's position, assuming the rider is still unmatched; a pre-trip request left unmatched at an epoch is dispatched solo at that epoch. A current rider $(i,u) \in \mathcal C_t$ is therefore at
$(i, qB)$ at epoch $t+qB$ if $u \le 0$, and at $(i, u+qB)$ if $u > 0$, provided this position does
not exceed $\ell_i - 1$. A future request $(\tau,k) \in \omega^s$ first enters the next batch at epoch $t+q_1B$, with $q_1 := \lceil (\tau - t + 1)/B \rceil$; at any later epoch $t+qB$, $q > q_1$, this future rider is projected to an on-trip rider in the same way.

\item \textit{Forward-looking matching optimization.} The first-stage variables $z_e \in \{0,1\}$, $e \in \mathcal E(\mathcal C_t)$, are the matches made now. The second-stage variables $w^s_{e,q} \in \{0,1\}$, $e \in \mathcal E^s_{q}$, are matches at future epoch $q$ in scenario $s$. The benchmark solves
\begin{align*}
\max_{\boldsymbol{z},\,\boldsymbol{w}}\quad & \sum_{e \in \mathcal E(\mathcal C_t)} g_e z_e
  + \frac{1}{S} \sum_{s=1}^S \sum_{q=1}^{\bar q} \sum_{e \in \mathcal E^s_{q}} g_e\, w^s_{e,q} \\
\text{s.t.}\quad & \sum_{e \in \mathcal E(\mathcal C_t):\, r \in e} z_e
  + \sum_{q=1}^{\bar q} \sum_{e \in \mathcal E^s_{q}:\, r \in e} w^s_{e,q} \le 1,
  && \forall r \in \mathcal C_t,\; s = 1,\dots,S, \\
& \sum_{q=1}^{\bar q} \sum_{e \in \mathcal E^s_{q}:\, f \in e} w^s_{e,q} \le 1,
  && \forall f \in \omega^s,\; s = 1,\dots,S, \\
& z_e,\, w^s_{e,q} \in \{0,1\}.
\end{align*}
The first constraint states that each current rider is matched at most once, now or at a future epoch, in every scenario. The second constraint matches each future request at most once.

\item \textit{Implementation.} The benchmark implements only the first-stage matches $\boldsymbol{z}$ and discards $\boldsymbol{w}$. It advances to the next epoch and solves a new program with new sampled scenarios. %
\end{itemize}

\section{Other Computational Results}

\subsection{Synthetic Instances} \label{ss:synthetic}

We first compare the profits of the pre-trip and on-trip matching policies (with pricing) in two synthetic examples. For both instances, we use cost $c=0.7$ and total per-period arrival probability $\sum_{i=1}^N \Lambda_i = 0.1$, which is equally divided among all rider types.
Following \cite{bimpikis2019spatial} and \cite{yan2023pricing}, we let riders’ willingness to pay be uniformly distributed in [0, 1] per unit of trip length, and thus $\lambda_i = 1 - p_i/\ell_i$ for each rider type $i \in [N]$. We consider pre-trip waiting periods $T\in\{1,3,5,7,10\}$. %
To evaluate each policy, we simulate each policy on synthetic arrivals over $10^7$ periods and compute the resulting profit.%

\medskip
\paragraph{Example 1 (Two rider types: one long and one short).} The first synthetic instance uses two rider types illustrated in \Cref{fig:2od}. The two rider types, each with the same arrival probability $\Lambda = 0.05$ per period, travel from left to right along a line toward a common destination. One rider type has a longer length of $L = 100$ periods, and the other has a shorter length of $l \in \{1,\cdots, L\}$ periods. To represent the difference between the two rider types, we define a \emph{demand heterogeneity}  parameter $\Delta=1-l/L\in[0, 1]$: a smaller $\Delta$ indicates that the riders are more similar, with  $\Delta=0$ corresponding to two riders with identical origin-destinations, who form an ``ideal" match.

\begin{figure}[htbp]
     \centering
     \begin{subfigure}[b]{0.45\textwidth}
        \begin{center}  %

        \begin{tikzpicture}[scale=0.82, transform shape]

  \node[person, fill=gray!25] (Oi) at (0,0) {\personicon};
  \node[person] (Oj) at (2.8,0) {\personicon};
  \node[halfgrayperson] (Dij) at (6.0,0) {\personicon};

  \draw[-, >=stealth, line width=0.8pt]
    (Oi.east) -- (Oj.west);

  \draw[->, >=stealth, line width=0.8pt]
    (Oj.east) -- (Dij.west);

  \node[below=12pt] at (Oi) {$O_i$};
  \node[below=12pt] at (Oj) {$O_j$};
  \node[below=12pt] at (Dij) {$D_i=D_j$};

  \draw[
    decorate,
    decoration={brace, amplitude=6pt, mirror},
    gray,
    line width=0.8pt
  ]
    (0,-1.45) -- (6.0,-1.45)
    node[midway, below=7pt, text=black] {$L$};

  \draw[
    decorate,
    decoration={brace, amplitude=6pt, mirror},
    gray,
    line width=0.8pt
  ]
    (2.8,-2.5) -- (6.0,-2.5)
    node[midway, below=7pt, text=black] {$l$};

     \path (0,-5) coordinate (bottom);

\end{tikzpicture}

          \caption{Two rider types.}
          \label{fig:2od}
        \end{center}
     \end{subfigure}
     \begin{subfigure}[b]{0.5\textwidth}
        \begin{center}  \includegraphics[height=6cm]{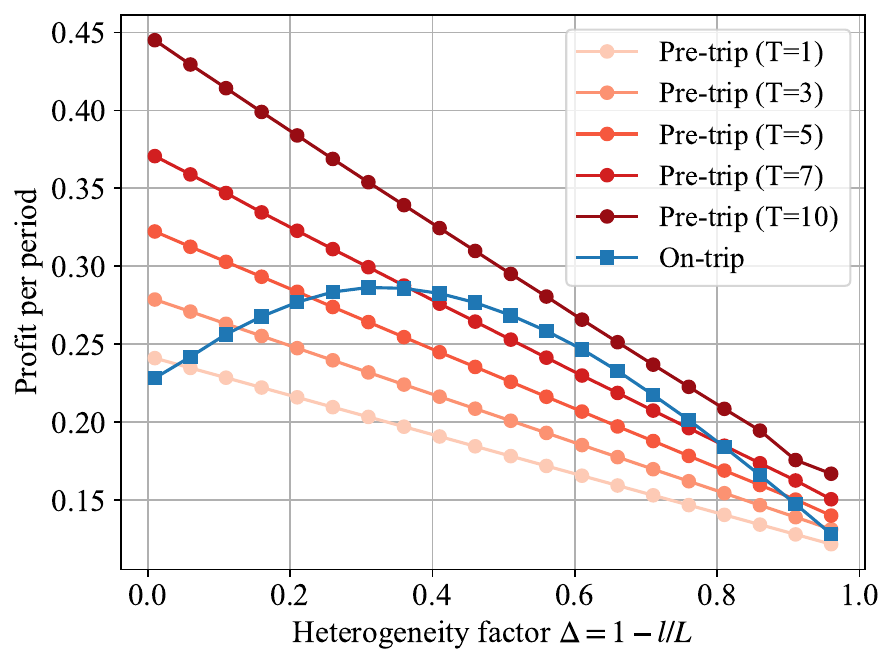}
          \caption{Profits of pre-trip and on-trip matching.}
          \label{fig:heterogeneity_2od}
        \end{center}
     \end{subfigure}
    \caption{Illustration of the two-rider-type example, and profit comparison between pre-trip and on-trip matching under varying demand heterogeneity. %
    }
    \label{fig:example1}
\end{figure}

\Cref{fig:heterogeneity_2od} compares the profits of pre-trip matching (shades of red) and on-trip matching (blue) for varying $\Delta$. Naturally, the pre-trip matching profit increases in $T$, because a longer waiting period coalesces demand more effectively and enables better matches. More importantly, the pre-trip matching profit decreases monotonically in $\Delta$, whereas the on-trip matching profit first increases, then decreases with $\Delta$. This contrasting behavior highlights the different strengths of the two matching policies. Pre-trip matching is primarily concerned with matching \emph{value} (i.e., the cost saved by matching two riders together), with the ideal scenario being $\Delta = 0$, a single origin-destination. Hence, when the demand is homogeneous ($\Delta$ is low), the matching value is high, and pre-trip matching performs well. In contrast, on-trip matching exhibits a trade-off between matching value and matching \emph{opportunity} (i.e., the likelihood of a match occurring). When the demand is homogeneous ($\Delta$ is low), the matching value is high, but the matching opportunity is low. This is because the distance between the two riders' origins is short, and once the vehicle passes the short rider's origin, it can no longer match the riders because of the backtracking condition. Meanwhile, when the demand is heterogeneous ($\Delta$ is high), the matching opportunity is high, but the matching value is low. At a moderate level of demand heterogeneity, both matching opportunity and matching value are present, benefiting on-trip matching. Thus, it is the importance of matching opportunity, as opposed to matching value alone, that drives the differing behavior of on-trip versus pre-trip matching.

\medskip
\paragraph{Example 2 ($N$ rider types on a grid network).} Next, we illustrate the impact of demand heterogeneity on a larger instance. Now, we have a $10\times10$ grid network, where each edge length is ten periods. Keeping the total per-period arrival probability fixed as $\sum_{i=1}^N \Lambda_i = 0.1$, we generate $N$ distinct rider types of equal length $L = 100$, varying $N$ from 1 to 30, and generating five instances for each $N$. \Cref{fig:heterogeneity_od_diagram} shows three examples of rider types generated for $N = 1,5$, and $10$. As $N$ increases, the demand becomes more spatially dispersed, and we can examine the effect of the dispersion on pre-trip and on-trip matching.

\begin{figure}[htbp]
     \centering
     \begin{subfigure}[b]{0.25\textwidth}
        \begin{center}  \includegraphics[width=\textwidth]{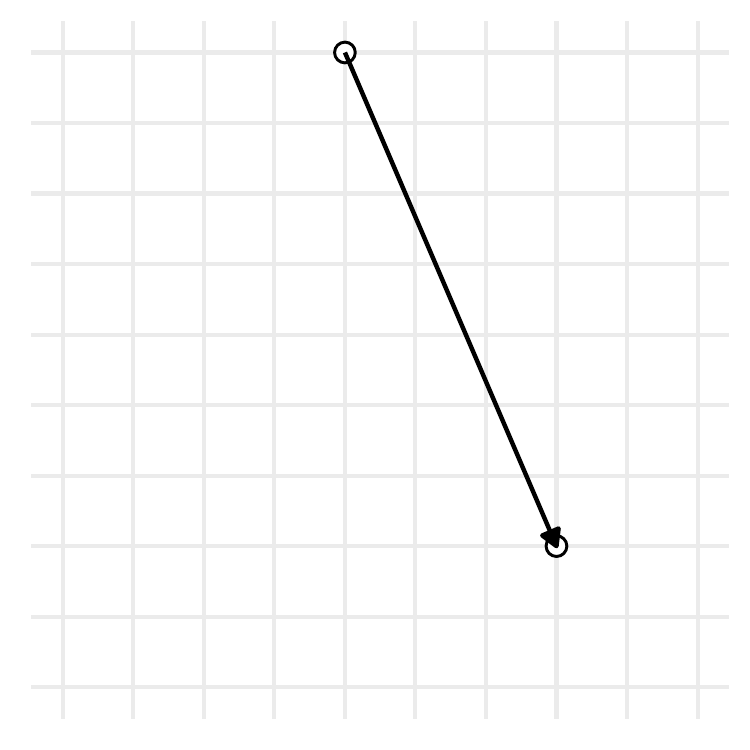}
          \caption{$N = 1$}
          \label{fig:heterogeneity-n1}
        \end{center}
     \end{subfigure}
     ~
     \begin{subfigure}[b]{0.25\textwidth}
        \begin{center}  \includegraphics[width=\textwidth]{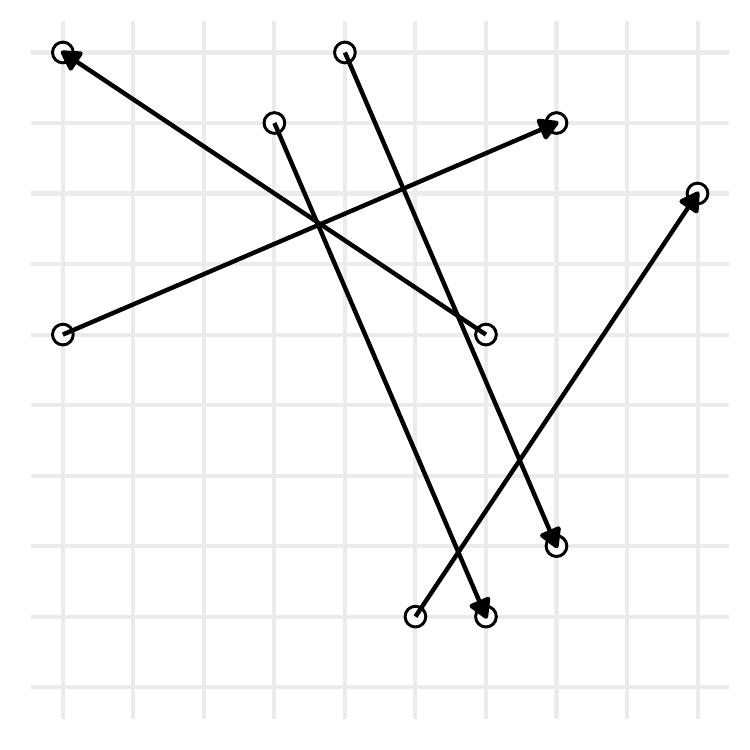}
          \caption{$N = 5$}
          \label{fig:heterogeneity-n5}
        \end{center}
     \end{subfigure}
     ~
     \begin{subfigure}[b]{0.25\textwidth}
        \begin{center}  \includegraphics[width=\textwidth]{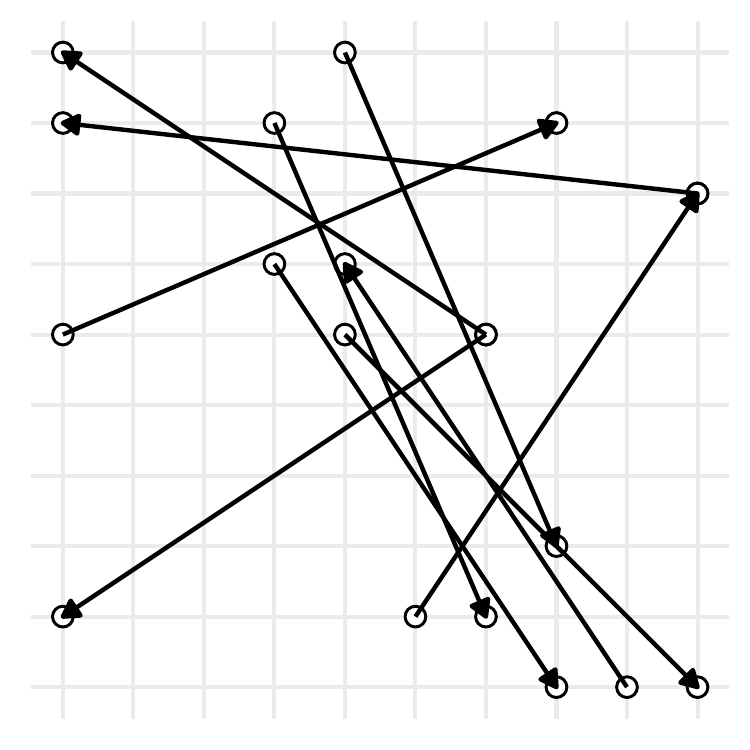}
          \caption{$N = 10$}
          \label{fig:heterogeneity-n10}
        \end{center}
     \end{subfigure}
    \caption{Examples of the distributions of the random rider types on the grid network ($N=1,5,10$; $L=100$).}
    \label{fig:heterogeneity_od_diagram}
\end{figure}

\begin{figure}[htbp]
\begin{center}  \includegraphics[width=0.5\textwidth]
{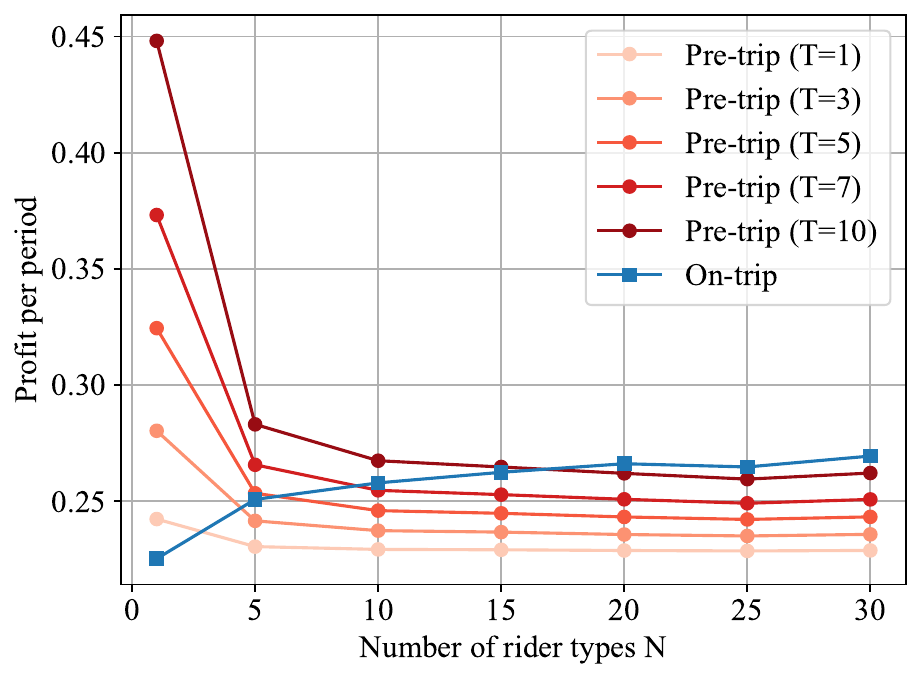}
  \caption{Profit comparison between on-trip and pre-trip matching policies under different numbers of rider types on the grid network. %
  }
  \label{fig:heterogeneity}
\end{center}
\end{figure}

\Cref{fig:heterogeneity} presents the pre-trip (shades of red) and on-trip matching (blue) profits for varying $N$. Pre-trip matching profit decreases markedly with $N$, reflecting the earlier observation that the ideal scenario for pre-trip matching is a single origin-destination: self-matches are common when demand is highly concentrated ($N$ is low), but become rarer as demand disperses ($N$ increases). Meanwhile, on-trip matching profit increases with $N$. As demand spreads out, more on-trip matching opportunities are created, which offsets the effect of lower matching value between riders.

The favorable performance of on-trip matching when riders are spatially dispersed proves to be important in real-world settings, which, despite showing high demand in dense corridors between major points of interest, also show substantial demand spread across greater metropolitan areas. %

\subsection{Density Effects in Chicago} \label{ss:density}

As discussed in the synthetic data of Example 1 (Appendix \ref{ss:synthetic}), on-trip matching is favorable when riders are spatially dispersed. To investigate this effect in a real data setting, we split the 42 zones of Chicago  into two regions (\Cref{fig:dense_regions}): the dense region (dark gray) consists of 10 zones with the highest demand density (0.23 riders per minute per square mile), while the sparse region (light gray) consists of the remaining 32 zones (0.04 riders per minute per square mile). %

Following the heterogeneity factor $\Delta$ introduced in Example 1 (Appendix \ref{ss:synthetic}), we define $\Delta$ for a match as one minus the ratio of overlap distance to total shared ride distance. As before, a smaller $\Delta$ corresponds to a match that is closer to the ideal match between two riders of identical origin-destinations.
\Cref{fig:matching_distribution} compares the number of matched requests per minute between dense and sparse regions, broken down by heterogeneity level (low, medium, and high, shown as separate bars)\footnote{We define a match with $\Delta<1/3$ as low heterogeneity, $1/3\le\Delta<2/3$ as medium heterogeneity, and $\Delta\ge2/3$ as high heterogeneity.} and by matching phase (pre-trip matches at the bottom of each bar, on-trip matches on top). %
The results show a clear contrast in matching tendency between dense and sparse regions. In the dense region where demand is concentrated, most matches are low-heterogeneity pre-trip matches.
In contrast, in the sparse region where demand is spatially dispersed, medium-to-high-heterogeneity {on-trip}  matches become more common.
This trend corroborates both synthetic examples (Appendix \ref{ss:synthetic}), in which on-trip matching becomes favorable as demand spatially disperses with a moderate level of heterogeneity---now with real-world data.

\begin{figure}[htbp]
     \centering
     \begin{subfigure}[b]{0.25\textwidth}
        \begin{center}  \includegraphics[height=6cm]{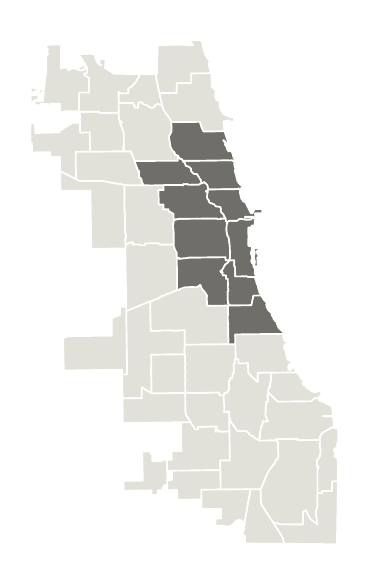}
          \caption{Region split}
          \label{fig:dense_regions}
        \end{center}
     \end{subfigure}
     \begin{subfigure}[b]{0.6\textwidth}
        \begin{center}  \includegraphics[height=6cm]{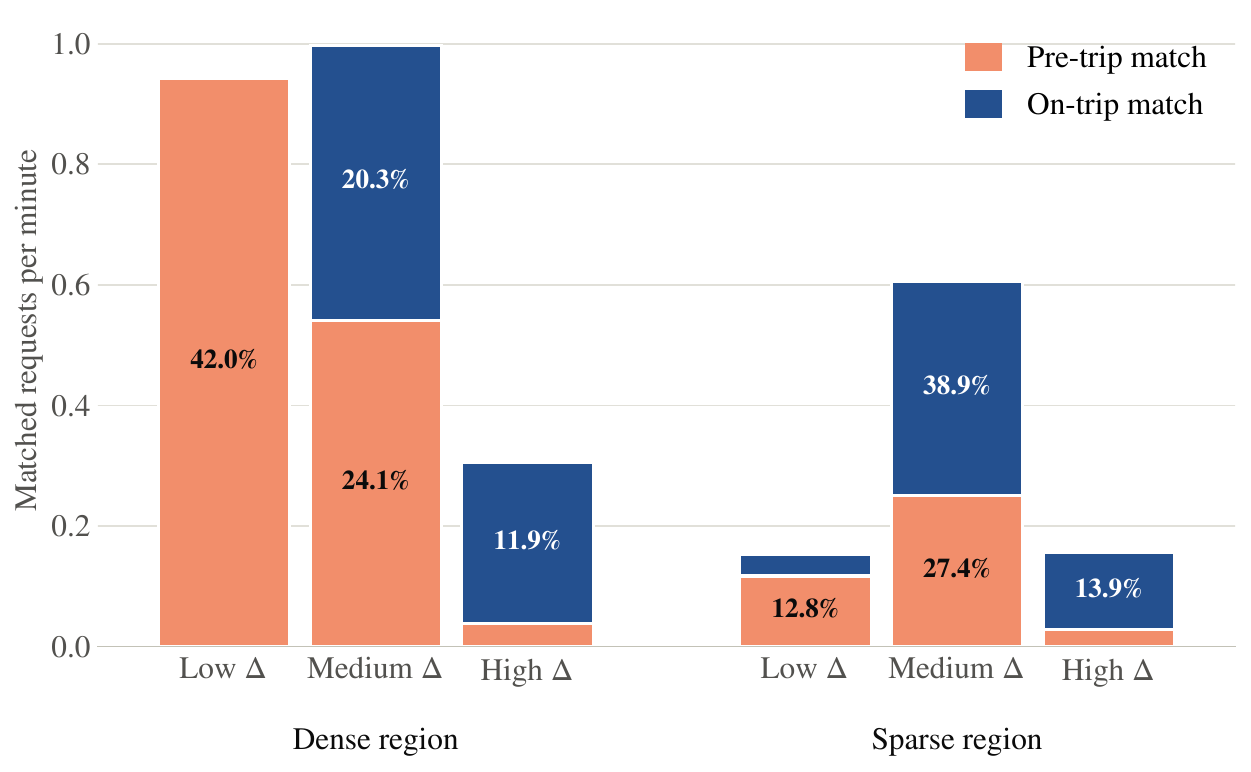}
          \caption{Matching outcome distribution }
          \label{fig:matching_distribution}
        \end{center}
     \end{subfigure}
    \caption{Region split in Chicago and matching outcome distribution when waiting period is $1.5$ min and  $c=0.9$.}
    \label{fig:chicago_heterogeneity}
\end{figure}

\subsection{Chicago Training Data Results} \label{sec:training_result}

\Cref{tab:train_benchmarks,tab:train_matching,tab:train_dynamic} present results on a semi-synthetic training dataset generated over $10^6$ periods with mean arrival rates estimated from the true Chicago training data.

\begin{table}[htbp]
  \centering
  \scalebox{0.68}
  {
\begin{tabular}{clrrrrrrrrrrrr}
    \toprule
    \multicolumn{2}{c}{\textbf{Cost} (\$/mile)} & \multicolumn{4}{c}{$c=0.7$} & \multicolumn{4}{c}{$c=0.9$} & \multicolumn{4}{c}{$c=1.1$} \\
    \cmidrule(lr){3-6} \cmidrule(lr){7-10} \cmidrule(lr){11-14}
    \multicolumn{2}{c}{\textbf{Waiting Window} (min)} & \multicolumn{1}{c}{$0.5$} & \multicolumn{1}{c}{$1$} & \multicolumn{1}{c}{$1.5$} & \multicolumn{1}{c}{$2$} & \multicolumn{1}{c}{$0.5$} & \multicolumn{1}{c}{$1$} & \multicolumn{1}{c}{$1.5$} & \multicolumn{1}{c}{$2$} & \multicolumn{1}{c}{$0.5$} & \multicolumn{1}{c}{$1$} & \multicolumn{1}{c}{$1.5$} & \multicolumn{1}{c}{$2$} \\
    \cmidrule(r){1-14}
    \multirow{5}[0]{*}{\shortstack{\textbf{Profit}\\(\$/min)}}
    & Dual-Based & 5.229 & 5.669 & 6.029 & 6.310 & 1.744 & 2.080 & 2.351 & 2.596 & 0.060 & 0.194 & 0.324 & 0.496 \\
    & Batch (Myopic) & 4.228 & 4.720 & 5.143 & 5.521 & 1.008 & 1.301 & 1.589 & 1.877 & 0.000 & 0.000 & 0.005 & 0.084 \\
    & & (-19\%) & (-17\%) & (-15\%) & (-12\%) & (-42\%) & (-37\%) & (-32\%) & (-28\%) & (-100\%) & (-100\%) & (-98\%) & (-83\%) \\
    & Batch (Rolling) & 5.073 & 5.386 & 5.684 & 5.944 & 1.650 & 1.863 & 2.061 & 2.252 & 0.040 & 0.102 & 0.161 & 0.273 \\
    & & (-3\%) & (-5\%) & (-6\%) & (-6\%) & (-5\%) & (-10\%) & (-12\%) & (-13\%) & (-34\%) & (-47\%) & (-50\%) & (-45\%) \\
    \cmidrule(r){1-14}
    \multirow{3}[0]{*}{\shortstack{\textbf{Match Rate}}}
    & Dual-Based & 0.634 & 0.676 & 0.704 & 0.721 & 0.651 & 0.686 & 0.717 & 0.742 & 0.747 & 0.754 & 0.788 & 0.823 \\
    & Batch (Myopic) & 0.694 & 0.722 & 0.737 & 0.753 & 0.712 & 0.742 & 0.758 & 0.773 & - & - & 0.842 & 0.861 \\
    & Batch (Rolling) & 0.651 & 0.687 & 0.710 & 0.734 & 0.668 & 0.708 & 0.731 & 0.753 & 0.747 & 0.757 & 0.811 & 0.849 \\
    \cmidrule(r){1-14}
    \multirow{3}{*}{\shortstack{\textbf{On-Trip Portion}}}
    & Dual-Based & 0.692 & 0.496 & 0.394 & 0.350 & 0.745 & 0.529 & 0.406 & 0.339 & 0.808 & 0.655 & 0.468 & 0.339 \\
    & Batch (Myopic) & 0.529 & 0.306 & 0.198 & 0.142 & 0.602 & 0.362 & 0.228 & 0.154 & - & - & 0.325 & 0.196 \\
    & Batch (Rolling) & 0.742 & 0.557 & 0.438 & 0.343 & 0.788 & 0.604 & 0.460 & 0.366 & 0.878 & 0.706 & 0.510 & 0.378 \\
    \cmidrule(r){1-14}
    \multirow{3}{*}{\shortstack{\textbf{Cost Efficiency}}}
    & Dual-Based & 0.189 & 0.215 & 0.235 & 0.250 & 0.186 & 0.212 & 0.232 & 0.249 & 0.239 & 0.241 & 0.266 & 0.290 \\
    & Batch (Myopic) & 0.136 & 0.166 & 0.191 & 0.212 & 0.136 & 0.164 & 0.187 & 0.209 & - & - & 0.220 & 0.242 \\
    & Batch (Rolling) & 0.181 & 0.200 & 0.217 & 0.232 & 0.179 & 0.198 & 0.215 & 0.230 & 0.229 & 0.224 & 0.243 & 0.264 \\
    \cmidrule(r){1-14}
    \multirow{3}{*}{\shortstack{\textbf{Detour Rate}\\(matched requests)}}
    & Dual-Based & 0.090 & 0.084 & 0.077 & 0.070 & 0.091 & 0.084 & 0.078 & 0.072 & 0.055 & 0.072 & 0.067 & 0.065 \\
    & Batch (Myopic) & 0.148 & 0.123 & 0.109 & 0.098 & 0.152 & 0.129 & 0.113 & 0.101 & - & - & 0.111 & 0.100 \\
    & Batch (Rolling) & 0.097 & 0.091 & 0.087 & 0.082 & 0.097 & 0.092 & 0.088 & 0.083 & 0.061 & 0.081 & 0.082 & 0.079 \\
    \bottomrule
    \end{tabular}
  }
  \caption{
  Performance of dual-based policy against batch matching benchmarks on Chicago data (training set). %
  }
  \label{tab:train_benchmarks}
\end{table}

\begin{table}[htbp]
  \centering
  \scalebox{0.6}
  {
    \begin{tabular}{clrrrrrrrrrrrrrrr}
    \toprule
    \multicolumn{2}{c}{\textbf{Cost} (\$/mile)} & \multicolumn{5}{c}{$c=0.7$} & \multicolumn{5}{c}{$c=0.9$} & \multicolumn{5}{c}{$c=1.1$} \\
    \cmidrule(lr){3-7} \cmidrule(lr){8-12} \cmidrule(lr){13-17}
    \multicolumn{2}{c}{\textbf{Waiting Window} (min)} & \multicolumn{1}{c}{$0$} & \multicolumn{1}{c}{$0.5$} & \multicolumn{1}{c}{$1$} & \multicolumn{1}{c}{$1.5$} & \multicolumn{1}{c}{$2$} & \multicolumn{1}{c}{$0$} & \multicolumn{1}{c}{$0.5$} & \multicolumn{1}{c}{$1$} & \multicolumn{1}{c}{$1.5$} & \multicolumn{1}{c}{$2$} & \multicolumn{1}{c}{$0$} & \multicolumn{1}{c}{$0.5$} & \multicolumn{1}{c}{$1$} & \multicolumn{1}{c}{$1.5$} & \multicolumn{1}{c}{$2$} \\
    \cmidrule(r){1-17}
    \multirow{3}[0]{*}{\shortstack{\textbf{Profit}\\(\$/min)}}
    & Combined & 4.734 & 5.229 & 5.669 & 6.029 & 6.310 & 1.444 & 1.744 & 2.080 & 2.351 & 2.596 & 0.000 & 0.060 & 0.194 & 0.324 & 0.496 \\
    & Pre-trip & 2.616 & 3.898 & 4.691 & 5.265 & 5.647 & 0.290 & 0.692 & 1.231 & 1.651 & 2.005 & 0.000 & 0.000 & 0.000 & 0.000 & 0.040 \\
    & & (-45\%) & (-25\%) & (-17\%) & (-13\%) & (-11\%) & (-80\%) & (-60\%) & (-41\%) & (-30\%) & (-23\%) & - & (-100\%) & (-100\%) & (-100\%) & (-92\%) \\
    \cmidrule(r){1-17}
    \multirow{3}[0]{*}{\shortstack{\textbf{Quoted Price}\\(\$/mile)}}
    & Combined & 0.799 & 0.778 & 0.765 & 0.755 & 0.747 & 0.891 & 0.870 & 0.853 & 0.841 & 0.832 & 1.000 & 0.983 & 0.960 & 0.945 & 0.929 \\
    & Pre-trip & 0.850 & 0.795 & 0.774 & 0.762 & 0.752 & 0.950 & 0.900 & 0.870 & 0.852 & 0.841 & 1.000 & 1.000 & 1.000 & 1.000 & 0.985 \\
    & & (+6\%) & (+2\%) & (+1\%) & (+1\%) & (+1\%) & (+7\%) & (+3\%) & (+2\%) & (+1\%) & (+1\%) & (0\%) & (+2\%) & (+4\%) & (+6\%) & (+6\%) \\
    \cmidrule(r){1-17}
    \multirow{3}[0]{*}{\shortstack{\textbf{Payment}\\(\$/mile)}}
    & Combined & 0.790 & 0.770 & 0.755 & 0.745 & 0.736 & 0.869 & 0.847 & 0.827 & 0.811 & 0.801 & - & 0.880 & 0.877 & 0.849 & 0.830 \\
    & Pre-trip & 0.850 & 0.789 & 0.764 & 0.751 & 0.741 & 0.950 & 0.885 & 0.842 & 0.820 & 0.808 & - & - & - & - & 0.773 \\
    &  & (+8\%) & (+2\%) & (+1\%) & (+1\%) & (+1\%) & (+9\%) & (+4\%) & (+2\%) & (+1\%) & (+1\%) & - & - & - & - & (-7\%) \\
    \cmidrule(r){1-17}
    \multirow{3}[0]{*}{\shortstack{\textbf{Throughput}\\(requests/min)}}
    & Combined & 5.682 & 6.235 & 6.631 & 6.906 & 7.128 & 3.052 & 3.645 & 4.135 & 4.485 & 4.726 & 0.000 & 0.470 & 1.115 & 1.546 & 1.980 \\
    & Pre-trip & 4.197 & 5.765 & 6.352 & 6.704 & 6.972 & 1.398 & 2.780 & 3.639 & 4.126 & 4.442 & 0.000 & 0.000 & 0.000 & 0.000 & 0.407 \\
    & & (-26\%) & (-8\%) & (-4\%) & (-3\%) & (-2\%) & (-54\%) & (-24\%) & (-12\%) & (-8\%) & (-6\%) & - & (-100\%) & (-100\%) & (-100\%) & (-79\%) \\
    \cmidrule(r){1-17}
    \multirow{2}[0]{*}{\shortstack{\textbf{Match Rate}}}
          & Combined & 0.599 & 0.634 & 0.676 & 0.704 & 0.721 & 0.611 & 0.651 & 0.686 & 0.717 & 0.742 & - & 0.747 & 0.754 & 0.788 & 0.823 \\
          & Pre-trip & 0.000 & 0.396 & 0.533 & 0.602 & 0.635 & 0.000 & 0.342 & 0.514 & 0.605 & 0.655 & - & - & - & - & 0.731 \\
    \cmidrule(r){1-17}
    \multirow{1}[0]{*}{\shortstack{\textbf{On-Trip Portion}}}
          & Combined & 1.000 & 0.692 & 0.496 & 0.394 & 0.350 & 1.000 & 0.745 & 0.529 & 0.406 & 0.339 & - & 0.808 & 0.655 & 0.468 & 0.339 \\
    \cmidrule(r){1-17}
    \multirow{2}[0]{*}{\shortstack{\textbf{Cost Efficiency}}}
          & Combined & 0.159 & 0.189 & 0.215 & 0.235 & 0.250 & 0.161 & 0.186 & 0.212 & 0.232 & 0.249 & - & 0.239 & 0.241 & 0.266 & 0.290 \\
          & Pre-trip & 0.000 & 0.104 & 0.159 & 0.194 & 0.216 & 0.000 & 0.079 & 0.147 & 0.187 & 0.213 & - & - & - & - & 0.366 \\
    \cmidrule(r){1-17}
    \multirow{2}[0]{*}{\shortstack{\textbf{Detour Rate}\\(matched requests)}}
          & Combined & 0.099 & 0.090 & 0.084 & 0.077 & 0.070 & 0.098 & 0.091 & 0.084 & 0.078 & 0.072 & - & 0.055 & 0.072 & 0.067 & 0.065 \\
          & Pre-trip & - & 0.099 & 0.084 & 0.077 & 0.067 & - & 0.109 & 0.087 & 0.078 & 0.071 & - & - & - & - & 0.000 \\
    \bottomrule
    \end{tabular}
  }
  \caption{
  Performance of combined matching against pre-trip matching on Chicago data (training set).
  }
  \label{tab:train_matching}
\end{table}

\begin{table}[htbp]
  \centering
  \scalebox{0.6}
  {
\begin{tabular}{clrrrrrrrrrrrrrrr}
    \toprule
    \multicolumn{2}{c}{\textbf{Cost} (\$/mile)} & \multicolumn{5}{c}{$c=0.7$} & \multicolumn{5}{c}{$c=0.9$} & \multicolumn{5}{c}{$c=1.1$} \\
    \cmidrule(lr){3-7} \cmidrule(lr){8-12} \cmidrule(lr){13-17}
    \multicolumn{2}{c}{\textbf{Waiting Window} (min)} & \multicolumn{1}{c}{$0$} & \multicolumn{1}{c}{$0.5$} & \multicolumn{1}{c}{$1$} & \multicolumn{1}{c}{$1.5$} & \multicolumn{1}{c}{$2$} & \multicolumn{1}{c}{$0$} & \multicolumn{1}{c}{$0.5$} & \multicolumn{1}{c}{$1$} & \multicolumn{1}{c}{$1.5$} & \multicolumn{1}{c}{$2$} & \multicolumn{1}{c}{$0$} & \multicolumn{1}{c}{$0.5$} & \multicolumn{1}{c}{$1$} & \multicolumn{1}{c}{$1.5$} & \multicolumn{1}{c}{$2$} \\
    \cmidrule(r){1-17}
    \multirow{3}[0]{*}{\shortstack{\textbf{Profit}\\(\$/min)}}
    & Combined & 4.867 & 5.438 & 5.902 & 6.295 & 6.604 & 1.684 & 2.053 & 2.412 & 2.718 & 2.968 & 0.059 & 0.167 & 0.437 & 0.618 & 0.795 \\
    & Pre-trip & 2.616 & 4.169 & 4.971 & 5.542 & 5.935 & 0.290 & 1.031 & 1.635 & 2.071 & 2.437 & 0.000 & 0.000 & 0.000 & 0.000 & 0.060 \\
    & & (-46\%) & (-23\%) & (-16\%) & (-12\%) & (-10\%) & (-83\%) & (-50\%) & (-32\%) & (-24\%) & (-18\%) & (-100\%) & (-100\%) & (-100\%) & (-100\%) & (-92\%) \\
    \cmidrule(r){1-17}
    \multirow{3}[0]{*}{\shortstack{\textbf{Quoted Price}\\(\$/mile)}}
    & Combined & 0.800 & 0.780 & 0.766 & 0.756 & 0.748 & 0.895 & 0.875 & 0.858 & 0.846 & 0.837 & 0.996 & 0.987 & 0.968 & 0.952 & 0.936 \\
    & Pre-trip & 0.850 & 0.797 & 0.775 & 0.763 & 0.754 & 0.950 & 0.904 & 0.874 & 0.858 & 0.846 & 1.000 & 1.000 & 1.000 & 1.000 & 0.987 \\
    & & (+6\%) & (+2\%) & (+1\%) & (+1\%) & (+1\%) & (+6\%) & (+3\%) & (+2\%) & (+1\%) & (+1\%) & (0\%) & (+1\%) & (+3\%) & (+5\%) & (+5\%) \\
    \cmidrule(r){1-17}
    \multirow{3}[0]{*}{\shortstack{\textbf{Payment}\\(\$/mile)}}
    & Combined & 0.785 & 0.762 & 0.745 & 0.734 & 0.726 & 0.851 & 0.828 & 0.807 & 0.793 & 0.783 & 0.854 & 0.843 & 0.837 & 0.820 & 0.808 \\
    & Pre-trip & 0.850 & 0.777 & 0.752 & 0.739 & 0.730 & 0.950 & 0.856 & 0.816 & 0.798 & 0.786 & - & - & - & - & 0.768 \\
    & & (+8\%) & (+2\%) & (+1\%) & (+1\%) & (+1\%) & (+12\%) & (+3\%) & (+1\%) & (+1\%) & (0\%) & - & - & - & - & (-5\%) \\
    \cmidrule(r){1-17}
    \multirow{3}[0]{*}{\shortstack{\textbf{Throughput}\\(requests/min)}}
    & Combined & 5.622 & 6.185 & 6.583 & 6.863 & 7.078 & 2.956 & 3.482 & 3.968 & 4.314 & 4.579 & 0.103 & 0.361 & 0.912 & 1.355 & 1.802 \\
    & Pre-trip & 4.197 & 5.726 & 6.303 & 6.660 & 6.897 & 1.398 & 2.663 & 3.506 & 3.975 & 4.299 & 0.000 & 0.000 & 0.000 & 0.000 & 0.363 \\
    &  & (-25\%) & (-7\%) & (-4\%) & (-3\%) & (-3\%) & (-53\%) & (-24\%) & (-12\%) & (-8\%) & (-6\%) & (-100\%) & (-100\%) & (-100\%) & (-100\%) & (-80\%) \\
    \cmidrule(r){1-17}
    \multirow{2}[0]{*}{\shortstack{\textbf{Match Rate}}}
    & Combined & 0.641 & 0.685 & 0.725 & 0.754 & 0.774 & 0.723 & 0.749 & 0.781 & 0.806 & 0.823 & 1.000 & 0.987 & 0.960 & 0.945 & 0.943 \\
    & Pre-trip & 0.000 & 0.482 & 0.609 & 0.672 & 0.703 & 0.000 & 0.512 & 0.661 & 0.725 & 0.764 & - & - & - & - & 0.889 \\
    \cmidrule(r){1-17}
    \multirow{1}[0]{*}{\shortstack{\textbf{On-Trip Portion}}}
          & Combined & 1.000 & 0.642 & 0.450 & 0.352 & 0.301 & 1.000 & 0.668 & 0.461 & 0.344 & 0.288 & 1.000 & 0.735 & 0.529 & 0.331 & 0.242 \\
    \cmidrule(r){1-17}
    \multirow{2}[0]{*}{\shortstack{\textbf{Cost Efficiency}}}
          & Combined & 0.174 & 0.211 & 0.238 & 0.260 & 0.278 & 0.196 & 0.227 & 0.253 & 0.275 & 0.290 & 0.297 & 0.317 & 0.320 & 0.336 & 0.348 \\
          & Pre-trip & 0.000 & 0.136 & 0.192 & 0.225 & 0.248 & 0.000 & 0.144 & 0.208 & 0.242 & 0.266 & - & - & - & - & 0.400 \\
    \cmidrule(r){1-17}
    \multirow{2}[0]{*}{\shortstack{\textbf{Detour Rate}\\(matched requests)}}
          & Combined & 0.096 & 0.085 & 0.078 & 0.070 & 0.063 & 0.090 & 0.082 & 0.075 & 0.068 & 0.063 & 0.055 & 0.054 & 0.062 & 0.057 & 0.054 \\
          & Pre-trip & - & 0.089 & 0.075 & 0.070 & 0.060 & - & 0.086 & 0.073 & 0.066 & 0.060 & - & - & - & - & 0.015 \\
    \bottomrule
    \end{tabular}
  }
  \caption{
  Performance of combined matching against pre-trip matching on Chicago data under dynamic pricing (training set).
  }
  \label{tab:train_dynamic}
\end{table}

\subsection{Chicago Regional Pricing Effects} \label{ss:match_types}

In this section, to understand the reason behind the small price increase around the densest downtown zones in \Cref{fig:Chicago_price}, we compare three policies: (1) pre-trip matching and the associated pricing policy, (2) combined matching under the pre-trip pricing policy, and (3) combined matching and the associated pricing policy. We separate the small downtown region where prices increase versus the outskirts as shown in \Cref{fig:dt_region}. \Cref{fig:ce} presents the cost efficiency for each region achieved by each of the three policies. We also define five types of matching outcomes for any request: (1) matched with a downtown rider pre-trip, (2) matched with an outskirts rider pre-trip, (3) matched with a downtown rider on-trip, (4) matched with an outskirts rider on-trip, or (5) unmatched and dispatched solo. \Cref{fig:match_distribution_type} presents the matching outcome distributions over downtown and outskirts requests for each of the three policies.

Under the pre-trip matching and pricing policies (first column in \Cref{fig:ce,fig:match_distribution_type}), downtown requests achieve high cost efficiency (0.234) and are mainly matched pre-trip with other downtown requests ($46.8\%$ of all downtown requests). Meanwhile, outskirts requests achieve a lower cost efficiency (0.128) despite a high pre-trip match rate ($35.2\%$ with other outskirts riders). In other words, pre-trip matching alone is highly efficient downtown, but inefficient in the outskirts.

Fixing the pricing policy and adding on-trip matching (second column in \Cref{fig:ce,fig:match_distribution_type}), outskirts riders benefit significantly (cost efficiency increases from 0.128 to 0.185), whereas downtown riders only see modest improvement (cost efficiency increases only from 0.234 to 0.247). \Cref{fig:match_distribution_type} illustrates that rather than simply adding on new matching opportunities, the addition of on-trip matching in fact \emph{suppresses} pre-trip matching: a downtown (or outskirts) request is less likely to be matched pre-trip with another downtown (or outskirts) request, and more likely to be matched on-trip (\Cref{fig:match_distribution_type}). Then, re-optimizing the pricing policy (third column in \Cref{fig:ce,fig:match_distribution_type}) reinforces this shift. The (slight) increase in downtown prices further reduces pre-trip matching between downtown riders, and shifts the downtown riders towards more on-trip matches with outskirts riders. Meanwhile, in the outskirts, the price decrease is driven by the dominant effect of more efficient on-trip matches.

\begin{figure}[htbp]
     \centering
     \begin{subfigure}[b]{0.25\textwidth}
        \begin{center}  \includegraphics[height=6cm]{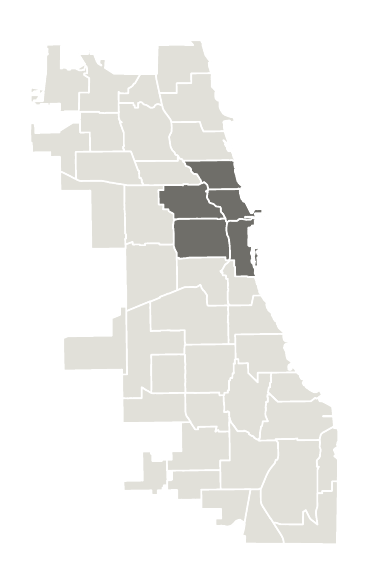}
          \caption{Region Split}
          \label{fig:dt_region}
        \end{center}
     \end{subfigure}
     \begin{subfigure}[b]{0.6\textwidth}
        \begin{center}  \includegraphics[height=6cm]{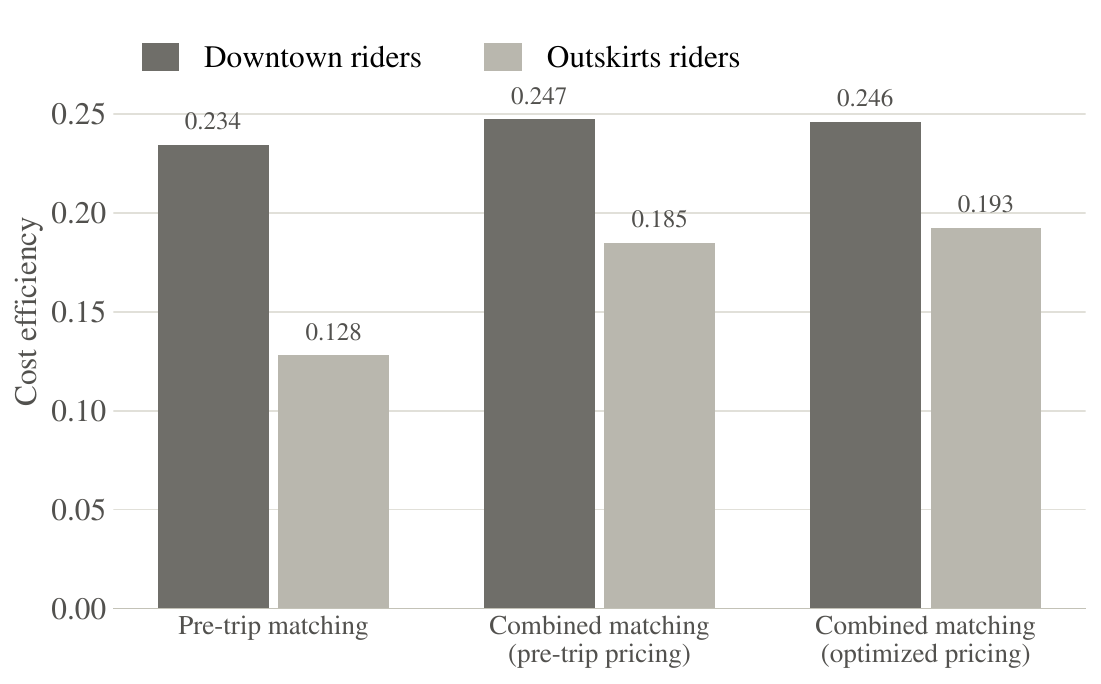}
          \caption{Cost Efficiency Distribution}
          \label{fig:ce}
        \end{center}
     \end{subfigure}
    \caption{Downtown/outskirts region split and the cost efficiency distribution in downtown versus outskirts across the three policies ($c=0.7$, waiting period is $1$ min).}
    \label{fig:chicago_cost_efficiency}
\end{figure}

\begin{figure}[htbp]
    \centering
    \includegraphics[width=\textwidth]{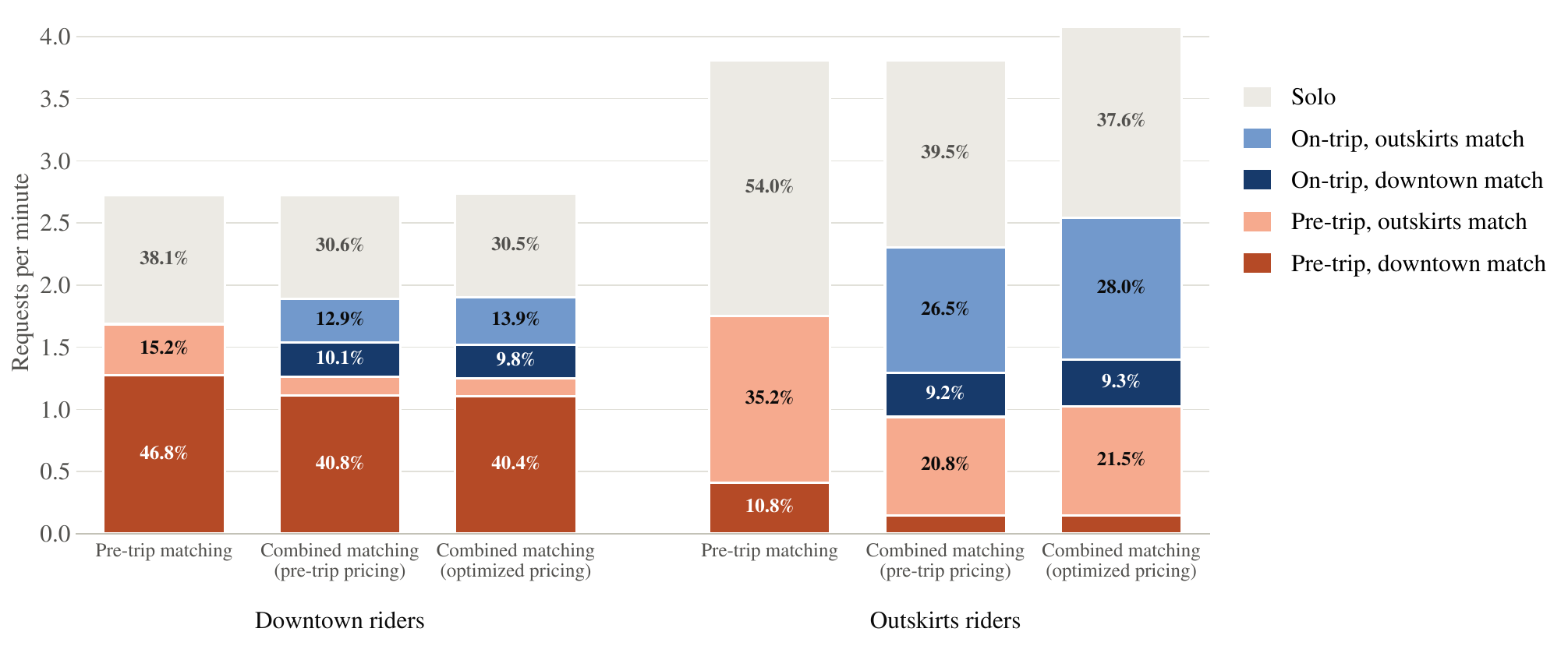}
    \caption{Matching outcome distribution over downtown and outskirts requests ($c=0.7$, waiting window = 1 min).
    A downtown (outskirts) match is a match with a rider originating downtown (in the outskirts).}
    \label{fig:match_distribution_type}
\end{figure}

\end{APPENDICES}

\end{document}